\documentclass[11pt,oneside, a4paper]{amsart} 

\usepackage[english]{babel}
\usepackage[T1]{fontenc}
\usepackage{amsmath}
\usepackage{amssymb}
\usepackage{float}
\usepackage{amsfonts}
\usepackage{mathtools}
\usepackage[utf8]{inputenc}
\usepackage[mathcal]{eucal}
\usepackage{enumerate}
\usepackage{amsthm}
\usepackage{hyperref}
\usepackage[totalwidth=14cm,totalheight=20cm]{geometry}
\usepackage{epsfig,fancyhdr,color}
\usepackage{multicol} 
\usepackage{csquotes}

\numberwithin{equation}{section}

\newtheorem{cor}{Corollary}[section]
\newtheorem{corx}{Corollary}

\newtheorem{teo}[cor]{Theorem}
\newtheorem{thmx}[corx]{Theorem}

\newtheorem{prop}[cor]{Proposition}

\newtheorem{questionx}[corx]{Question}

\newtheorem{lemma}[cor]{Lemma}

\theoremstyle{definition}
\newtheorem{defi}[cor]{Definition}
\theoremstyle{remark}
\newtheorem{remark}[cor]{Remark}
\newtheorem*{remark*}{Remark}

\newtheorem*{notation}{Notation}

\newcommand{\Pp}{\mathbb{P}}
\newcommand{\R}{\mathbb{R}}

\newcommand{\C}{\mathbb{C}}
\newcommand{\Z}{\mathbb{Z}}

\newcommand{\h}{\mathbb{H}}

\newcommand{\sE}{\mathcal{E}}

\newcommand{\diag}{\mathrm{diag}}

\newcommand{\Stab}{\mathrm{Stab}}
\newcommand{\SL}{\mathrm{SL}}
\newcommand{\Gr}{\mathrm{Gr}}

\newcommand{\pr}{\mathrm{pr}}
\newcommand{\trace}{\mathrm{tr}}

\newcommand{\Sp}{\mathrm{Sp}}
\newcommand{\U}{\mathrm{U}}
\newcommand{\SU}{\mathrm{SU}}
\newcommand{\GL}{\mathrm{GL}}
\newcommand{\Sym}{\mathrm{Sym}}
\newcommand{\CP}{\mathbb{C}\mathbb{P}}

\newcommand{\uu}{\tilde{u}}

\newcommand{\SO}{\mathrm{SO}}
\newcommand{\Fix}{\mathrm{Fix}}
\newcommand{\Hom}{\mathrm{Hom}}
\newcommand{\Span}{\mathrm{Span}}
\newcommand{\Ein}{\mathrm{Ein}}
\newcommand{\Id}{\mathrm{Id}}
\newcommand{\dVol}{\mathrm{dVol}}
\newcommand{\Aut}{\mathrm{Aut}}

\DeclareMathAlphabet{\mathpzc}{OT1}{pzc}{m}{it}

\title[Planar $\Sp(4,\R)$-minimal surfaces with polynomial growth]{Planar minimal surfaces with polynomial growth in the $\Sp(4, \R)$-symmetric space}

\author{Andrea Tamburelli}
\author{Michael Wolf}

\begin{document}

\begin{abstract} We study the asymptotic geometry of a family of conformally planar minimal surfaces with polynomial growth in the $\Sp(4,\R)$-symmetric space. We describe a homeomomorphism between the ``Hitchin component'' of wild $\Sp(4,\R)$-Higgs bundles over $\CP^1$ with a single pole at infinity and a component of maximal surfaces with light-like polygonal boundary in $\h^{2,2}$. Moreover, we identify those surfaces with convex embeddings into the Grassmannian of symplectic planes of $\R^{4}$. We show, in addition, that our planar maximal surfaces are the local limits of equivariant maximal surfaces in $\h^{2,2}$ associated to $\Sp(4,\R)$-Hitchin representations along rays of holomorphic quartic differentials.
\end{abstract}

\date{\today}
\maketitle
\setcounter{tocdepth}{1}

\tableofcontents

\section*{Introduction}
Let $S$ be a surface of finite type and let $G$ be a real semisimple Lie group. 
Higher Rank Teichm\"uller theory is a quickly growing area of research that studies dynamical, geometric and algebraic properties of representations of $\pi_{1}(S)$ into $G$ (\cite{Wienhard_ICM}), whose origins can be traced back to the pioneering work of Hitchin on Higgs bundles (\cite{Hitchin_selfduality}), and to the foundational work of Corlette (\cite{Corlette}), Donaldson (\cite{Donaldson_selfduality}), and Simpson (\cite{Simpson_harmonicbundles}). In brief, they developed a theory, generally referred to as nonabelian Hodge correspondence, that provides homeomorphisms between three natural objects: the character variety $\chi(\Gamma, G)$, the de-Rahm moduli space of flat connections on principal $G$-bundles over a Riemann surface $X=(S,J)$; and the Dolbeaut moduli space of $G$-Higgs bundles on $X$. A fundamental role in this theory is played by equivariant harmonic maps from the universal cover of $X$ to the symmetric space $G/K$ (\cite{QL_introduction}). In particular, when $G$ is real split of rank $2$ the theory is very rich and well-understood: by work of Hitchin (\cite{Hitchin_component}) there is a connected component in the character variety that generalizes Teichm\"uller space; there is a unique preferred choice of conformal structure on $S$ that makes the associated equivariant harmonic maps conformal, and thus (branched) minimal immersions (\cite{Lab:cyclicsurf}); and representations in this connected component all arise as holonomy of geometric structures on (bundles over) $S$ (\cite{Baraglia}, \cite{CTT}, \cite{Labourie_cubic}, \cite{Loftin_RP2}, \cite{Mess}, \cite{GW_PSL4}).\\

More recently, the nonabelian Hodge correspondence has been extended to include surfaces with punctures and Higgs bundles with meromorphic Higgs field with tame or irregular singularities at the punctures (\cite{BB_wild}, \cite{Boalch_Stokes}, \cite{Simpson_harmonicbundles}). The main aim of this paper is to describe harmonic maps arising from Higgs bundles over $\C\Pp^{1}$ with polynomial Higgs field for the Lie group $\Sp(4,\R)$ and the associated geometric structures. \\

In turn, the study of these planar Higgs bundles provides tools for studying families of representations that leave compacta in a character variety. In particular, and in the setting of the Hitchin component of $\Sp(4,\R)$ representations of surface groups, consider a family of representations that has associated harmonic maps from a Riemann surface $X$ that are conformal and of energies growing without bound. In that case, we show that the high energy harmonic maps {\it localize} in the sense that the restriction of global harmonic maps of the surface to a small neighborhood on the surface is well-approximated by the harmonic maps associated to a Higgs bundle over $\C\Pp^{1}$ with polynomial Higgs field (for the Lie group $\Sp(4,\R)$).\\

Let us first introduce some notations and terminology. Let $f: \C \rightarrow \SL(n,\R)/\SO(n)$ be a harmonic map. We can interpret the differential $df$ of $f$ as a $1$-form with values in the vector space $\mathfrak{m}=\{ A \in \mathfrak{sl}(n,\R) \ | \ A=A^{t}\}$. The harmonicity of $f$ implies that the $(1,0)$-part $\varphi=\partial f$ of its differential is holomorphic. We can thus associate to $f$ holomorphic $k$-differentials $q_k$ on the complex plane defined by $q_{k}=\trace(\varphi^{k})$. Notice, in particular, that $q_{1}=0$ and $q_{2}$ is the Hopf differential of the harmonic map $f$. We say that $f$ has polynomial growth if $q_{k}$ are all polynomials. One naturally asks,

\begin{questionx}\label{question}
Given holomorphic polynomial $k$-differentials $(q_{2}, \dots, q_{n})$, is there a harmonic map $f: \C \rightarrow \SL(n,\R)/\SO(n)$ such that $q_{k}=\trace((\partial f)^{k})$? Moreover, can we describe its image? 
\end{questionx}

In this generality, the above question is still open. The first work in this direction is due to Han, Tam, Treibergs and Wan (\cite{HTTW}) who studied harmonic maps from the complex plane to the hyperbolic plane with polynomial Hopf differential, showing that the quadratic differential uniquely determines a harmonic diffeomorphism from $\C$ to an ideal polygon in $\h^{2}$ with $m+2$ vertices if $m$ is the degree of the polynomial. A simple dimension count, however, shows that there cannot be a one-to-one correspondence between polynomial quadratic differentials of degree $m$ on the complex plane and ideal polygons in $\h^{2}$ with $m+2$ vertices. In a recent work (\cite{Gupta_wild}) Gupta explained this phenomenon in terms of rate in which the harmonic maps take the end of $\C$ into the cusps and obtained a homeomorphism between these moduli spaces by prescribing the principal part at infinity of the Hopf differential. Another interpretation was given by the first author (\cite{Tambu_poly}), who, using anti-de Sitter geometry, showed that there is actually a homeomorphism between polynomial quadratic differentials on the complex plane and pairs of ideal polygons in $\h^{2}$ with the same number of vertices. \\

An answer to Question \ref{question} is also known for $n=3$ in case of conformal harmonic maps. In joint work with David Dumas (\cite{DWpolygons}), the second author used techniques from affine differential geometry to show that there is a conformal equivariant harmonic map from $\C$ to $\SL(3,\R)/\SO(3)$ with prescribed polynomial cubic differential $q_{3}$. Moreover, they constructed a homeomorphism between the moduli space of polynomial cubic differentials of degree $m$ and convex polygons in $\R\Pp^{3}$ with $m+3$ vertices, exploiting the fact that these harmonic maps arise as Gauss maps of hyperbolic affine spheres in $\R^{3}$, which project to convex sets, in this case polygons, in $\R\Pp^{2}$. In terms of the geometry of the minimal surface in the symmetric space, this result can be interpreted as the solution of an asymptotic Dirichlet problem for minimal surfaces in $\SL(3,\R)/\SO(3)$: the minimal surfaces found in \cite{DWpolygons} are asymptotic to $2(m+3)$ flats at infinity with the property that each consecutive pair shares three adjacent Weyl chambers at infinity. \\

In this paper we extend this result to conformal harmonic maps with polynomial growth into $\Sp(4,\R)/\U(2)$. We prove the following:

\begin{thmx}\label{thmA}
Assume that $q_{4}$ is a polynomial holomorphic quartic differential of degree $n$. Then there exists a conformal harmonic map $f: \C \rightarrow \Sp(4,\R)/\U(2)$ such that $q_{4}=\trace((\partial f)^{4})$. Moreover, the associated minimal surface $f(\C)$ is asymptotic to $2(n+4)$ flats as $|z| \to +\infty$, with the property that any consecutive pair shares four adjacent Weyl chambers at infinity. Such a collection determines the minimal surface and $q_{4}$ uniquely.
\end{thmx}

Although the general idea of the proof resembles that in \cite{DWpolygons}, the techniques used are very different for two main reasons. First of all, we construct the harmonic map using Higgs bundles: we associate to $q_{4}$ an irregular Higgs bundle over $\C\Pp^{1}$ and find the solution to Hitchin's self-duality equations. We then obtain the minimal surface by parallel transport of a unitary frame using the associated flat connection. In particular, the study of the geometry at infinity of the minimal surface requires precise estimates on the parallel transport as $|z| \to + \infty$; this in turn involves the study of the asymptotic behavior of solutions of a (coupled) system of elliptic PDE on the complex plane (unlike the $\SL(3,\R)$ case where the equation can be reduced to a scalar PDE). These techniques have the advantage that they might be easily adapted to every cyclic Higgs bundle and thus used for the study of the asymptotic geometry of planar minimal surfaces into $\SL(n,\R)/\SO(n)$ where only $q_{n}$ does not vanish identically. \\

The other main difference is that, in order to prove the second part of Theorem \ref{thmA}, we use techniques from pseudo-Riemannian geometry. Exploiting the low-dimensional isomorphism $\Pp\Sp(4,\R) \cong \SO_{0}(2,3)$, we interpret the harmonic maps found before as Gauss maps of maximal surfaces in the pseudo-hyperbolic space $\h^{2,2}$ bounding a future-directed negative light-like polygon in the Einstein Universe $\Ein^{1,2}$, which we view as the Lorentzian conformal boundary at infinity of $\h^{2,2}$. We show the following:

\begin{thmx}\label{thmB}
There is a homeomorphism between the moduli space of polynomial quartic differentials of degree $n$ on the complex plane and a connected component of the moduli space of future-directed negative light-like polygons with $n+4$ vertices in the Einstein Universe.
\end{thmx}

One cannot ignore the theme running through Theorem~\ref{thmB}, \cite{DWpolygons} and even \cite{HTTW} (compare \cite{WolfRays}). All of these works identify a \enquote{Stokes phenomenon} in which certain cyclic Higgs bundles on $\C\Pp^{1}$ (whose Higgs field has a wild singularity at $\infty$) define geometric shapes -- ideal polygons in $\h^2$, convex real projective polygons in $\R\Pp^2$, or future-directed negative light-like polygons in $\Ein^{1,2}$ -- which arise in a common way. In particular, the associated harmonic maps from $\C$ to the symmetric space which have a constant holomorphic differential, and (hence) map onto a flat, provide asymptotic solutions (for the solutions of the Hitchin equations under study) in a region of the plane defined by the geometry of the quadratic, cubic or (in the present case) quartic differential; passing from one of these regions to another through a Stokes direction in the plane provides that the solutions transition to be asymptotic to a different flat in the symmetric space (typically sharing a collection of Weyl chambers). \\

However, unlike \cite{HTTW} or \cite{DWpolygons}, we have reasons to believe that in our case this moduli space of geometric structures that cyclic Higgs bundles on $\C\Pp^{1}$ induce is not connected. Note that in both \cite{HTTW} and \cite{DWpolygons}, the geometric objects under study (harmonic diffeomorphisms onto ideal polygons in $\h^{2}$ and  affine spheres projecting onto convex polygons in $\R\Pp^{2}$) can arise only from one family of wild ($\SL(2,\R)$ or $\SL(3,\R)$) Higgs bundles on $\C\Pp^{1}$ with singularity at infinity that are themselves reminiscent of the Higgs bundles in the Hitchin component in the case of surfaces with negative Euler characteristic. In our context, however, already in the classical setting of closed surfaces of genus at least $2$, complete maximal surfaces in $\h^{2,2}$ can be obtained from different families of Higgs bundles (\cite{CTT}) belonging to different connected components in the moduli space. This suggests that there should be a one-to-one correspondence between the connected components of the moduli space of wild $\Sp(4,\R)$-Higgs bundles over $\C\Pp^{1}$ and the connected components of future-directed negative light-like polygons in $\Ein^{1,2}$. In support of this idea, we find an explicit parametrization of the moduli space of future-directed negative light-like hexagons in $\Ein^{1,2}$ and show that this has two connected components. Our Theorem \ref{thmB} gives a homeomorphism with the component that does not contain the unique (up to the action of $\SO_{0}(2,3)$) future-directed negative light-like hexagon in $\Ein^{1,1}\subset \Ein^{1,2}$. This is consistent with our conjecture as the family of wild $\Sp(4,\R)$-Higgs bundles over $\C\Pp^{1}$ that we consider in this paper are those belonging to the Hitchin section (cfr. \cite{FN}) that can never be reduced to $\SO_{0}(2,2)$-Higgs bundles.  \\

We believe that this study is also relevant to developing a harmonic map compactification of the Hitchin component in the spirit of (\cite{Wolf_thesis}). By work of Labourie (\cite{Lab:cyclicsurf}) the $\Sp(4,\R)$-Hitchin component of a closed surface $S$ may be parametrized by the bundle $\mathcal{Q}^{4}$ of quartic differentials over the Teichm\"uller space of $S$. A natural question is to understand the asymptotic behavior of the representations $\rho_{s}$ and of the associated $\rho_{s}$-equivariant harmonic maps $g_{s}: \tilde{S} \rightarrow \Sp(4,\R)/\U(2)$ along a ray $q_{s}=sq_{0}$ of quartic differentials. Works of Collier-Li (\cite{Collier-Li}) and Mochizuki (\cite{Mochizuki_harmonicbundles}) give a precise picture away from the zeros of the quartic differential $q_{0}$ (as well as for general $n$-differentials). A consequence of Theorem~\ref{thmA} is an initial study of the asymptotics of the harmonic maps $g_s$ on all neighborhoods of the surface $S$ in this case of quartic differentials.  In particular, we imagine rescaling the coordinate chart in a neighborhood so that $q_{s}$ converges to the polynomial quartic differential $z^{k}dz^{4}$ over $\C$. We can then use the solution of Hitchin's equations on the plane, found in Theorem \ref{thmA}, to give the following asymptotic estimates, which extend the ones found in \cite{DW_rays} for the Blaschke metrics along rays of cubic differentials to the present $\Sp(4,\R)/\U(2)$ setting:

\begin{thmx}\label{thmD} Let $q_{s}=sq_{0}$ be a ray of holomorphic quartic differentials on a closed Riemann surface $X=(S,J)$. Let $\sigma$ be the conformal hyperbolic metric on $X$ and let $g_{s}=\diag(g_{1,s}, g_{2,s}^{-1}, g_{1,s}^{-1}, g_{2,s})$ be the harmonic metric on the Higgs bundle $(\sE, \varphi_{s})$ over $X$, where
\[
    \sE=K^{\frac{3}{2}}\oplus K^{-\frac{1}{2}}\oplus K^{-\frac{3}{2}}\oplus K^{\frac{1}{2}}
    \ \ \ \text{and} \ \ \ \varphi_{s}=\begin{pmatrix}
			0 & 0 & q_{s} & 0 \\
			0 & 0 & 0 & 1 \\
			0 & 1 & 0 & 0 \\
			1 & 0 & 0 & 0  \\
		\end{pmatrix} \ .
\]
Let $p$ be a zero of order $k$ for $q_{0}$. Then, as $s \to +\infty$, there exists a sequence of rays $r_{s} \to 0$ such that
\[
    {g_{1,s}^{-1}}_{|_{B(p,r_{s})}}=O(s^{\frac{3}{k+4}}\sigma^{\frac{3}{4}}) \ \ \ \text{and} \ \ \
    {g_{2,s}^{-1}}_{|_{B(p,r_{s})}}=O(s^{\frac{1}{k+4}}\sigma^{\frac{1}{4}}) \ .
\]
\end{thmx}

As a consequence, we deduce (see Corollary~\ref{cor: localization of maximal surfaces along rays}) that the family of maximal surfaces in $\h^{2,2}$ arising from the Higgs bundles described above (\cite{CTT}) converge in the pointed Gromov-Hausdorff topology to the planar maximal surfaces of Theorem \ref{thmB} with polynomial quartic differential $z^{k}dz^{4}$, where $k$ is the vanishing order of $q_{0}$ at the chosen base point. This \enquote{localization} result mirrors the result in \cite{DW_rays} that rays of affine spheres in the Labourie-Loftin coordinates converge to affine spheres over regular polygons.
We believe also that these estimates should play a role in the study of the asymptotic holonomy along paths that go through some zeros of $q_{0}$ and in the description of the (rescaled) limiting harmonic map to a building. We leave these aspects to future work.\\

Finally, we compare the present work to another recent response to \cite{CTT}.  Labourie, Toulisse and the second author \cite{LabTW20} study the case of spacelike maximal surfaces in $\h^{2,n}$ with positive boundary on $\Ein^{1,n}$ and no characterization of the conformal type of the maximal surface (instead, their focus is on removing the restriction in \cite{CTT} to a cocompact group action). In contrast, Theorems~\ref{thmA} and \ref{thmB} in the present work study boundary maps which are polygonal, hence only semi-positive, on planar surfaces.

\subsection*{Acknowledgements.} The first author gratefully acknowledges support from the NSF GEAR network (NSF grants DMS-1107452, 1107263, 1107367 \enquote{RNMS: GEometric structures And Representation varieties}). The second author acknowledges both his secondary role on this paper as well as support from NSF DMS-1564374, the Simons Foundation and MSRI, where some of the final preparation of the paper took place. Both are grateful to Fran\c{c}ois Labourie (especially for discussions surrounding convex surfaces and general positivity), Qiongling Li for important conversations, and J\'er\'emy Toulisse for his comments on a earlier draft of this paper and for suggesting that the space of negative light-like polygons in $\Ein^{1,2}$ may be disconnected.

\section{Background material}\label{sec:background} 

\subsection{Lie theory for $\Sp(4,\R)$}
We recall briefly the relevant Lie theory for the Lie group $\Sp(4,\R)$. In particular, we fix once and for all an identification of $\mathfrak{sp}(4,\R)$ as subalgebra of $\mathfrak{sl}(4,\C)$. \\
\\
\indent We consider on $\C^{4}$ the symplectic form given by
\[ \label{eqn: definition of Omega}
	\Omega=\begin{pmatrix}
			0 & Id \\
			-Id & 0 \\
		    \end{pmatrix} \ .
\]
The complex symplectic group $\Sp(4,\C)$ consists of all linear transformations $g$ in $\GL(4,\C)$ such that $g^{t}\Omega g=\Omega$. Hence, its Lie algebra is 
\[
	\mathfrak{sp}(4,\C)=\{ X \in \mathfrak{gl}(4,\C) \ | \ X^{t}\Omega+\Omega X=0 \} \ .
\]
A simple computation shows that $X \in \mathfrak{sp}(4,\C)$ if and only if it can be written as
\[
	X=\begin{pmatrix}
		A & B \\
		C & -A^{t} \\
		\end{pmatrix}
\]
for some $A \in \GL(2, \C)$ and $B,C \in \Sym(2,\C)$. The anti-linear involution 
\begin{align*}
	\rho: \Sp(4, \C) &\rightarrow \Sp(4,\C) \\
		g &\mapsto (\overline{g^{-1}})^{t}
\end{align*}
fixes a maximal compact subgroup isomorphic to $\SU(4)$. \\
We identify $\Sp(4,\R)$ with the fixed points in the complex group $\Sp(4, \C)$ of the anti-linear involution
\begin{align*}
	\lambda: \Sp(4, \C) &\rightarrow \Sp(4, \C) \\
			g &\mapsto \begin{pmatrix}
					0 & Id \\
					Id & 0 \\
					\end{pmatrix}\overline{g} \begin{pmatrix}
										0 & Id \\
										Id & 0 \\
									\end{pmatrix} \ .
\end{align*}
\begin{remark}\label{rmk:conjugation} This group is conjugate to the standard $\Sp(4,\R)$ consisting of matrices with real coefficient preserving the symplectic form $\Omega$ via 
\[
	A=\frac{1}{\sqrt{2}}\begin{pmatrix}
					1 & 0 & i & 0 \\
					0 & 1 & 0 & i \\
					1 & 0 & -i & 0 \\
					0 & 1 & 0 & -i
				\end{pmatrix} \in \SU(4)
\]
\end{remark}
At the Lie algebra level, this identification of $\Sp(4,\R)$ provides for the identification
\[
	\mathfrak{sp}(4,\R)=\left\{ \begin{pmatrix}
					A & B \\
					\overline{B} & -A^{t} 	\\
					\end{pmatrix} \ | \ A \in \mathfrak{u}(2); \ B\in \Sym(2, \C) \right\} \ .
\]
The involutions $\rho$ and $\lambda$ commute and the composition $\sigma=\lambda \circ \rho$ acts on $\mathfrak{sp}(4, \R)$ as
\[
	\sigma \begin{pmatrix}
			A & B \\
			\overline{B} & -A^{t} \\
			\end{pmatrix}= \begin{pmatrix}
						A & -B \\
						-\overline{B} & -A^{t} \\
					     \end{pmatrix} \ .
\]
We deduce that $\sigma$ is a Cartan involution for $\mathfrak{sp}(4, \R)$ and induces the (Cartan) decomposition 
\[
	\mathfrak{sp}(4, \R)=\mathfrak{u}(2) \oplus (\Sym(2, \R) \oplus \Sym(2, \R)) .
\]
By complexifying, we obtain the splitting
\[
	\mathfrak{sp}(4,\C)=\mathfrak{gl}(2,\C) \oplus (\Sym(2, \C) \oplus \Sym(2, \C)).
\]

\subsection{$\Sp(4,\R)$-Higgs bundles}
We recall here the definition of $\Sp(4, \R)$-Higgs bundles over closed Riemann surfaces and their connection with harmonic maps in the symmetric space $\Sp(4, \R)/\U(2)$. 

\begin{defi} An $\Sp(4,\R)$-Higgs bundle on a closed Riemann surface $\Sigma$ is a triple $(V, \beta, \gamma)$, where $V$ is a holomorphic vector bundle of rank $2$, and the forms $\beta \in H^{0}(\Sigma, \Sym(V)\otimes K)$ and $\gamma \in H^{0}(\Sigma, \Sym(V)^{*}\otimes K)$, where $K$ is the canonical bundle over $\Sigma$.
\end{defi}

The associated $\SL(4, \C)$-Higgs bundle is given by the holomorphic vector bundle $\sE=V\oplus V^{*}$ on $\Sigma$ and the Higgs field $\varphi: \sE \rightarrow \sE \otimes K$ represented by the matrix
\[
	\varphi=\begin{pmatrix}
			0 & \beta \\
			\gamma & 0 \\
		\end{pmatrix} \ .
\]
This bundle comes equipped with a symplectic form $\Omega$ and an orthogonal structure $Q:\sE \rightarrow \sE^{*}$, which, in the above splitting $\sE=V\oplus V^{*}$, are given by 
\[
    \Omega=\begin{pmatrix} 
			0 & Id \\
			-Id & 0 \\
		    \end{pmatrix} \ \ \ \ \ \text{and} \ \ \ \ \ 
	Q=\begin{pmatrix}
			0 & Id \\
			Id & 0 \\
		    \end{pmatrix} \ .
\]
More generally, we will say that a frame for $\sE$ is $\Omega$-symplectic and $Q$-adapted, if the symplectic form and the orthogonal structure are represented by the above matrices.\\
We are interested in Higgs bundles in the $\Sp(4, \R)$-Hitchin component. Those are parameterized \cite{Lab:cyclicsurf} by a point in Teichm\"uller space (corresponding to the complex structure on $\Sigma$) and a holomorphic quartic differential $q$, and they are given by the triple
\[
	V=K^{\frac{3}{2}} \oplus K^{-\frac{1}{2}} \ \ \ \gamma=\begin{pmatrix}
										0 & 1 \\
										1 & 0 \\
										\end{pmatrix} \ \ \ \beta=\begin{pmatrix}
														  q & 0\\
														0 & 1 \\
														\end{pmatrix} \ .
\]
Hitchin's equations look for a (harmonic) hermitian metric $H$ on $\sE$ such that the $\Sp(4,\R)$-connection 
\[
	\nabla=D_{H}+\varphi+\varphi^{*H}
\]
is flat, where $D_{H}$ denotes the Chern connection of $H$. It is well-known that the solution is unique \cite{Simpson_Hodge} and diagonal \cite{Simpson_Katz} of the form $H=\diag(h_{1},h_{2}^{-1}, h_{1}^{-1},h_{2})$ in the above splitting. Notice that the hermitian metric is compatible with the symplectic structure $\Omega$ and the orthogonal structure $Q$ in the sense that $H^{t}\Omega H=\Omega$ and $H^{t}QH=Q$. The monodromy of the flat connection $\nabla$ defines then a representation $\rho: \pi_{1}(S) \rightarrow \Sp(4, \R)$.\\

Moreover, the metric $H$ induces a $\rho$-equivariant harmonic map 
\[
    \tilde{f}_{\rho}: \tilde{\Sigma} \rightarrow \Sp(4, \R)/\U(2)
\]
as follows. Fix a point $\tilde{p}_{0} \in \tilde{\Sigma}$ and fix a holomorphic, $Q$-adapted, $\Omega$-symplectic and $H$-unitary frame $N(\tilde{p})$ for the bundle $\sE$ at every point $\tilde{p} \in \tilde{\Sigma}$. For every $\tilde{p} \in \tilde{\Sigma}$, we denote by $\mathcal{N}(\tilde{p})$ the parallel transport of $N(\tilde{p}_{0})$ at $\tilde{p}$. Notice that in general (i.e. when $\varphi \neq 0$), the frame $\mathcal{N}(\tilde{p})$ will not be unitary. If we identify the symmetric space $\SL(4,\C)/\SU(4)$ with the space of hermitian metrics on $\C^{4}$, the harmonic map is given by
\begin{align*}
	\tilde{f}_{\rho}: \tilde{\Sigma} &\rightarrow \SL(4,\C)/\SU(4)\\
		\tilde{p} &\mapsto H^{\mathcal{N}(\tilde{p})} \ .
		%=\{\text{metric $H$ expressed in the frame $\mathcal{N}(\tilde{p})$}\} .
\end{align*}
Here $H^{\mathcal{N}(\tilde{p})}$ is the metric $H$ expressed in the frame $\mathcal{N}(\tilde{p})$.
We then notice that the image of $\tilde{f}_{\rho}$ is actually contained in the copy of $\Sp(4, \R)/\U(2)$ consisting of hermitian metrics $H$ on $\C^{4}$ that are $Q$-symmetric (i.e. $H^{t}QH^{-1}=Q$) and $\Omega$-symplectic (i.e. $H^{t}\Omega H=\Omega$). In fact, if we denote by $g(\tilde{p}) \in \Sp(4,\R)$ the family of matrices such that $\mathcal{N}(\tilde{p})g(\tilde{p})=N(\tilde{p})$, then
\begin{equation}\label{eqn: computation of harmonic map}
    \tilde{f}_{\rho}(\tilde{p})=\overline{(g(\tilde{p})^{-1})^{t}}g(\tilde{p})^{-1} \ ,
\end{equation}
and an easy computation shows that $g(\tilde{p}) \in \Sp(4,\R)$ is equivalent to the hermitian metric $\tilde{f}_{\rho}(\tilde{p})$ being $\Omega$-symplectic and $Q$-symmetric. In addition, noting that $\trace(\varphi^2)$ vanishes, we see that the map $\tilde{f}_{\rho}$ is conformal \cite{Corlette} and thus parameterizes a minimal surface in the symmetric space. 

\subsection{Planar minimal surfaces with polynomial growth}
In this paper we are interested in the study of a particular class of minimal surfaces in $\Sp(4, \R)/\U(2)$ which are described by conformal, harmonic maps $f: \C \rightarrow \Sp(4,\R)/\U(2)$ with polynomial growth. \\

Given a map $f: \C \rightarrow \Sp(4, \R)/\U(2)$, we recall (\cite{Corlette}) that if $f$ is harmonic then, for a lift $\tilde{f}: \C \rightarrow \Sp(4,\R)$, we have that $\varphi=(\partial\tilde{f})^{\perp}$ is holomorphic, where $(\partial\tilde{f})^{\perp}$ denotes the component of the (1,0)-part of the differential of $\tilde{f}$, which is orthogonal to $\mathfrak{u}(2)$ with respect to the Killing form
\begin{align*}
	B: \mathfrak{sp}(4,\R) \times \mathfrak{sp}(4,\R) &\rightarrow \C\\
		(X,Y) &\mapsto \trace(XY)\  \ .
\end{align*}
\noindent In particular the quadratic differential
\[
	q_{2}=\trace(\varphi^{2})
\]
and the quartic differential
\[ \label{eqn: defn of q}
	q_{4}=\trace(\varphi^{4})
\]
are holomorphic. The Killing form $B$ induces a Riemannian metric $g$ on the symmetric space, and its pull-back via $f$ is 
\[
	f^{*}g_{p}(X,Y)=B((\varphi+\varphi^{*H})(X), (\varphi+\varphi^{*H})(Y))
\]
where $H=f(p)$. Therefore, $q_{2}$ is the Hopf differential of the harmonic map $f$ and the vanishing of $q_{2}$ is equivalent to the map $f$ being conformal. In this case the pull-back metric reduces to 
\[
	f^{*}g_{p}=\trace(\varphi\varphi^{*H}) \ .
\]
Finally, we say that {\it $f$ has polynomial growth} if the quartic differential $q_{4}=q$ is a polynomial over $\C$.\\

We can actually interpret the harmonic map $f$ as the harmonic metric induced by some Higgs bundle over $\CP^ {1}$ with singularity at infinity. It is sufficient to consider the holomorphic bundle $\sE=\oplus_{i=1}^{4}\mathcal{O}(\alpha_{i})$ over $\CP^{1}$ endowed with the Higgs field
\[
	\varphi=\begin{pmatrix}
			0 & 0 & q & 0 \\
			0 & 0 & 0 & 1 \\
			0 & 1 & 0 & 0 \\
			1 & 0 & 0 & 0  \\
		\end{pmatrix} \ .
\]
The Higgs bundle $(\sE, \varphi)$ is the $\SL(4, \C)$-Higgs bundle associated to an $\Sp(4, \R)$-Higgs bundle with singularity at infinity. Here, we consider $(\sE, \varphi)$ as a good filtered Higgs bundle (\cite{Mochizuki_harmonicbundles},\cite{FN}) with weights $(\alpha_{1}, \alpha_{2}, \alpha_{3}, \alpha_{4})$. In addition, for a meromorphic section $s=(s_{1}, s_{2}, s_{3}, s_{4})$ of $\sE$, we define $v_{\infty}(s)=\max_{i=1, \dots, 4}\{ v_{\infty}(s_{i})+\alpha_{i}\}$, where $v_{\infty}(s_{i}) \in \mathbb{Z}$ is the order of singularity at infinity of the section $s_{i}$. We then seek a hermitian metric $H$ on $\sE$ satisfying the self-duality equations
\begin{equation}\label{eq:self-duality}
    F_{H}+[\varphi, \varphi^{*H}]=0 \ ,
\end{equation}
which is compatible with the filtration in the following sense: for every meromorphic section $s$ of $\sE$ we require that
\[
    H(s(z), s(z))=O(|z|^{-2v_{\infty}(s)}) \ \ \ \text{as} \ \ \ |z| \to +\infty \ .
\]
If such $H$ exists, then the map $f$ coincides with the conformal harmonic map induced by $H$ via the procedure described in the previous subsection.\\

In the section~\ref{sec:existence}, we will find a solution to Equation (\ref{eq:self-duality}) for the Higgs bundle $(\sE, \varphi)$ with weights
\[
    (\alpha_{1}, \alpha_{2}, \alpha_{3}, \alpha_{4})= \left( \frac{3n}{8}, -\frac{n}{8}, -\frac{3n}{8}, \frac{n}{8}\right) \ ,
\]
where we assume that $q$ is a polynomial quartic differential of degree $n$. 

\subsection{Moduli space of polynomial quartic differentials}
A polynomial quartic differential is a holomorphic differential on the complex plane of the form $q(z)dz^{4}$, where $q(z)$ is a polynomial function. We denote by $\mathcal{Q}_{n}$ the space of polynomial quartic differentials of degree $n$. The group $\Aut(\C)$ of biholomophisms of $\C$ acts on this space by push-forward. Let $\mathcal{MQ}_{n}$ be the quotient of $\mathcal{Q}_{n}$ by this action. The geometry of the resulting moduli space is analogous to that described for polynomial cubic differentials in \cite{DWpolygons}.

\begin{prop}\label{prop:moduli_poly}The moduli space $\mathcal{MQ}_{n}$ is a complex orbifold of real dimension $2(n-1)$ if $n\geq 1$.
\end{prop}
\begin{proof}Every polynomial quartic differential may be written as 
\[
	q=(a_{n}z^{n}+a_{n-1}z^{n-1}+ \cdots +a_{0})dz^{4}
\]
for some $a_{i}\in \C$ and $a_{n} \in \C^{*}$. An element $T(z)=bz+c \in \Aut(\C)$ acts on q via
\[
	T_{*}q=(a_{n}b^{n+4}(z+c/b)^{n}+a_{n-1}b^{n+3}(z+c/b)^{n-1}+\cdots +b^{4}a_{0})dz^{4}.
\]
Hence by choosing $b=a_{n}^{-1/(n+4)}$ we may make $T_{*}q$ monic (i.e. with leading coefficient equal to $1$); then a suitable choice of the translation component $c$ allows us to assume that $T_{*}q$ is centered (i.e. with $a_{n-1}=0$). Notice that these choices are unique up to multiplying $b$ by an $(n+4)$-root of unity. Thus we can describe the moduli space $\mathcal{MQ}_{n}$ as the quotient
\[
	\mathcal{MQ}_{n}=\mathcal{TQ}_{n}/\Z_{n+4}
\]
where $\mathcal{TQ}_{n}$ is the space of monic and centered polynomials of degree $n$ and $\Z_{n+4}$ denotes the cyclic group of order $n+4$ generated by $T(z)=\zeta_{n+4}^{-1}z$ for a primitive $(n+4)$-root of unity $\zeta_{n+4}$. Since $\mathcal{TQ}_{n}$ is naturally identified with $\C^{n-1}$ by
\begin{align*}
	\mathcal{TQ}_{n} &\rightarrow \C^{n-1}\\
		(z^{n}+a_{n-2}z^{n-2}+\cdots a_{0}) &\mapsto (a_{n-2}, \dots, a_{0}) \ ,
\end{align*}
it follows that $\mathcal{MQ}_{n}$ is a complex orbifold of real dimension $2(n-1)$.
\end{proof}

\begin{remark} If $n=0$, the space $\mathcal{MQ}_{0}$ consists of only one point, represented by the quartic differential $q=dz^{4}$.
\end{remark}

We put on $\mathcal{MQ}_{n}$ the topology induced by the identification 
\[
	\mathcal{MQ}_{n}=\mathcal{TQ}_{n}/\Z_{n+4}
\]
found in Proposition \ref{prop:moduli_poly}.

\section{Existence}\label{sec:existence}
In this section we prove the existence of a conformal harmonic map $f:\C \rightarrow \Sp(4,\R)/\U(2)$ with given polynomial quartic differential $q_4=q$ (cf. \eqref{eqn: defn of q}). We will provide also precise estimates of the behaviour of the associated harmonic metric $H$ when $|z| \to \infty$. 

\begin{teo}\label{thm:existence} Let $q$ be a polynomial quartic differential of degree $n$. Consider the good-filtered $\Sp(4,\R)$-Higgs bundle $(\sE, \varphi)$ over $\C\Pp^{1}$ where 
\[
 \sE=\mathcal{O}\left(\frac{3n}{8}\right)\oplus \mathcal{O}\left(-\frac{n}{8}\right)\oplus \mathcal{O}\left(-\frac{3n}{8}\right)\oplus \mathcal{O}\left(\frac{n}{8}\right)
 \]
 and 
 \[
    \varphi=\begin{pmatrix}
			0 & 0 & q & 0 \\
			0 & 0 & 0 & 1 \\
			0 & 1 & 0 & 0 \\
			1 & 0 & 0 & 0  \\
		\end{pmatrix}  \ .
\]
Then there exists a unique diagonal harmonic metric $H$ satisfying Hitchin's self-duality equation $F_{H}+[\varphi, \varphi^{*H}]=0$.
\end{teo}

Inspired by the solution of Hitchin's equations for $\Sp(4,\R)$-Higgs bundles over closed Riemann surfaces (see Section \ref{sec:background}), we look for a diagonal metric of the form $H=\diag(h_{1}, h_{2}^{-1}, h_{1}^{-1}, h_{2})$. Under this assumption, the equation $F_{H}+[\varphi, \varphi^{*H}]=0$ simplifies into the following 
%fully 
coupled system of elliptic PDE
\[
	\begin{cases}
	\Delta \log(h_{1})+h_{1}^{-1}h_{2}-h_{1}^{2}|q|^{2}=0 \\
	\Delta \log(h_{2})+h_{2}^{-2}-h_{1}^{-1}h_{2}=0  \ .
	\end{cases} 
\]

Note that here we adopt the convention that $\Delta = \partial_z \partial_{\bar{z}}$; while this convention is more common for authors writing on Hitchin equations, it differs from that invoked often by authors writing from a harmonic maps or conformal variational problem viewpoint.

It is convenient to define $u_{i}=\log(\frac{1}{h_{i}})$ and study the system in the following form
\begin{equation}\label{eq:system}
	\begin{cases}
	\Delta u_{1}=e^{u_{1}-u_{2}}-e^{-2u_{1}}|q|^{2} \\
	\Delta u_{2}=e^{2u_{2}}-e^{u_{1}-u_{2}} \ .
	\end{cases}
\end{equation}
Namely, if we define
	\begin{align*}
	F: \R^{2} &\rightarrow \R^{2}\\
		F(u_{1}, u_{2})&=(e^{u_{1}-u_{2}}-e^{-2u_{1}}|q|^{2}, e^{2u_{2}}-e^{u_{1}-u_{2}})=(F_1, F_2) 
	\end{align*}
the above system may be written as
\[
	\Delta u=F(u)
\]
where $u=(u_{1},u_{2})$ and the map $F$ satisfies a monotone condition
\[
	\frac{\partial F_{j}}{\partial u_{i}} \leq 0 \ \ \ \ \ \ \text{for} \ \ i\neq j \ .
\]
In this setting we can apply a super- and sub-solution method to prove the existence of a smooth solution defined over all $\C$. \\

Since we did not manage to find a precise reference for this method applied to a system of PDE, we provide a detailed description of its application to Equation (\ref{eq:system}). \\

Let $B_{R}$ be the ball of radius $R$ centered at $0$. We start by proving the existence of a solution to Equation (\ref{eq:system}) on the domain $B_{R}$ with some smooth boundary values $(w_{1}, w_{2})$ and for $R$ sufficiently large. 

\begin{defi} We say that $u^{+}$ is a super-solution of Equation (\ref{eq:system}) with boundary values $(w_{1},w_{2})$ on the ball $B_{R}$ if it is continuous and satisfies 
\[
\begin{cases}
	\Delta u_{i}^{+} \leq F_{i}(u^{+}) \ \ \ \ \ \ \text{for} \ \ i=1,2 \\
	u_{i}^{+}\geq w_{i} \ \ \ \ \ \ \ \ \ \ \ \  \ \ \text{on} \ \ \ \ \partial B_{R}
\end{cases}
\]
in the weak sense. Similarly $u^{-}$ is a sub-solution if it is continuous and satisfies 
\[
\begin{cases}
	\Delta u_{i}^{-} \geq F_{i}(u^{-}) \ \ \ \ \ \ \text{for} \ \ i=1,2  \\
	u_{i}^{-} \leq w_{i} \ \ \ \ \ \ \ \ \  \ \ \ \ \ \text{on} \ \ \ \ \partial B_{R}
\end{cases}
\]
in the weak sense.
\end{defi}

\indent Our sub- and super-solution for System (\ref{eq:system}) will be slightly modifications of  
\[
	(u_{1}, u_{2})=\left(\frac{3}{4}\log(|q|), \frac{1}{4}\log(|q|)\right) \ ,
\]
which is the exact solution of the system if $q$ is a non-zero constant quartic differential (or an exact solution in regions where $q$ does not vanish). We will also choose the boundary values
\[
	(w_{1}, w_{2})=\left( \frac{3}{4} \log(|q|), \frac{1}{4}\log(|q|) \right) \ ,
\]
which are smooth on $\partial B_{R}$ as soon as it does not contain any zeros of $q$.

\begin{lemma}\label{lm:subsolution} The following function $u^{-}=(u_{1}^{-}, u_{2}^{-})$ is a sub-solution of Equation (\ref{eq:system}):
\[
	u_{1}^{-}=\begin{cases} \log(|q|^{\frac{3}{4}})  \ \ \ \ \  \ \ \ \ \ \ \ \ \ \ \ \ \ \ \  \ \ \ \text{if} \ \ \ |z|>d  \\
					\max( \log(g_{2d}^{\frac{3}{2}}), \log(|q|^{\frac{3}{4}})) \ \ \ \ \ \text{if} \ \ \ |z| \leq d
			\end{cases}
\]
\[
	u_{2}^{-}=\begin{cases} \log(|q|^{\frac{1}{4}})  \ \ \ \ \  \ \ \ \ \ \ \ \ \ \ \ \ \ \ \ \ \ \ \text{if} \ \ \ |z|>d  \\
					\max( \log(g_{2d}^{\frac{1}{2}}), \log(|q|^{\frac{1}{4}})) \ \ \ \ \ \text{if} \ \ \ |z| \leq d
			\end{cases}
\]
where $g_{2d}$ denotes the density of the metric with constant curvature $-2$ on the ball $B(0,2d)$ centred at the origin with radius $2d$
\[
	g_{2d}=\frac{1}{2}\left(\frac{4d}{4d^{2}-|z|^{2}}\right)^{2}
\]
and $d$ is a positive real number that depends only on the quartic differential $q$. 
\end{lemma}
\begin{proof} Let us verify first that $u_{i}^{-}$ are continuous. We can choose $d$ sufficiently large such that $\{ z \ | \ |q(z)| \leq 1\} \subset B(0,d)$ and we can suppose that $d>\frac{4}{3}$ in such a way that $\log(g_{2d})$ is negative for $|z|\leq d$. This implies that the functions $u_{i}^{-}$ are continuous in a neighbourhood of $|z|=d$.  Moreover, they are continuous in a neighbourhood of the zeros of $q$ because $\log(|q|)$ tends to $-\infty$ at the zeros of $q$, whereas $g_{2d}$ is bounded away from $0$. We notice also that the functions $u_{i}^{-}$ are piece-wise smooth and thus locally Lipschitz. 

Let us now verify that $u^{-}$ is a sub-solution of the system. Since 
\[
	\max( \log(g_{2d}^{\frac{3}{2}}), \log(|q|^{\frac{3}{4}}))=\log(|q|^{\frac{3}{4}}) \Leftrightarrow \max( \log(g_{2d}^{\frac{1}{2}}), \log(|q|^{\frac{1}{4}}))= \log(|q|^{\frac{1}{4}})
\]
it is sufficient to verify that the pairs $(\log(|q|^{\frac{3}{4}}), \log(|q|^{\frac{1}{4}}))$ and $(\log(g_{2d}^{\frac{3}{2}}),  \log(g_{2d}^{\frac{1}{2}}))$ are sub-solutions. Away from the zeros of $q$, the pair $(\log(|q|^{\frac{3}{4}}), \log(|q|^{\frac{1}{4}}))$ is a solution of the system, hence in particular it is a sub-solution. As for the second pair, the density of the metric with constant curvature $-2$ satisfies the differential equation
\[
	\Delta \log(g_{2d})=g_{2d} \ 
\]
therefore, 
\[
	F_{1}(\log(g_{2d}^{\frac{3}{2}}), \log(g_{2d}^{\frac{1}{2}}))=g_{2d}-g_{2d}^{-1}|q|^{2} \leq \frac{3}{2}g_{2d}= \Delta \log(g_{2d}^{\frac{3}{2}})
\]
\[
	F_{2}(\log(g_{2d}^{\frac{3}{2}}), \log(g_{2d}^{\frac{1}{2}}))=0 \leq \frac{1}{2}g_{2d}= \Delta \log(g_{2d}^{\frac{1}{2}})
\]
We deduce that at every point the function $u^{-}$ is a sub-solution or the maximum of two sub-solutions, hence it is a sub-solution.
Notice also that the boundary conditions are satisfied as soon as $R>d$. 
\end{proof}

\begin{lemma} \label{lm:supersolution} There exists a constant $C>1$ such that, for any choice of $R$ and consequent boundary values $(w_1, w_2)$ on $\partial B_R$, the pair
\[
	(u_{1}^{+}, u_{2}^{+})=\left(\frac{3}{8}\log(|q|^{2}+C), \frac{1}{8}\log(|q|^{2}+3C)\right) 
\]
is a super-solution of System (\ref{eq:system}) with those boundary values $(w_{1},w_{2})$.
\end{lemma}
\begin{proof} Of course, as soon as $C>0$, we have that $u_{i}^{+} > w_i$, so the boundary conditions for a supersolution are satisfied on $\partial B_R$ for any $R$ sufficiently large. Then to find a constant $C$ for which $(u_{1}^{+}, u_{2}^{+})$ is a supersolution of equation (\ref{eq:system}), we begin by noting that a simple computation shows that
\[
	\Delta u_{1}^{+}=\frac{3}{8} \frac{|q_{z}|^{2}C}{(|q|^{2}+C)^{2}}
\]
and
\[
	\Delta u_{2}^{+}=\frac{3}{8} \frac{|q_{z}|^{2}C}{(|q|^{2}+3C)^{2}}.
\]
Moreover,
\[
 F_{1}(u_{1}^{+}, u_{2}^{+})=(|q|^{2}+C)^{\frac{1}{8}}(|q|^{2}+3C)^{-\frac{1}{8}}-(|q|^{2}+C)^{\frac{3}{4}}|q|^{2}
\]
\[
 F_{2}(u_{1}^{+},u_{2}^{+})=(|q|^{2}+3C)^{\frac{1}{4}}-(|q|^{2}+C)^{\frac{3}{8}}(|q|^{2}+3C)^{-\frac{3}{8}} \ .
\]
Therefore, we need to show that there exists a constant $C>0$ such that
\[
	\begin{cases}
  (|q|^{2}+C)^{\frac{19}{8}}(|q|^{2}+3C)^{-\frac{1}{8}}-(|q|^{2}+C)^{\frac{5}{4}}|q|^{2}\geq \frac{3}{8}|q_{z}|^{2}C   \\
  (|q|^{2}+3C)^{\frac{9}{4}}-(|q|^{2}+C)^{\frac{3}{8}}(|q|^{2}+3C)^{\frac{15}{8}} \geq \frac{3}{8}|q_{z}|^{2}C \\
	\end{cases} \ .
\]
Let us consider the following one-parameter family of functions:
\[
	f_{C}(z)=(|q|^{2}+C)^{\frac{19}{8}}(|q|^{2}+3C)^{-\frac{1}{8}}-(|q|^{2}+C)^{\frac{5}{4}}|q|^{2}-\frac{3}{8}|q_{z}|^{2}C
\]
\[
	g_{C}(z)= (|q|^{2}+3C)^{\frac{9}{4}}-(|q|^{2}+C)^{\frac{3}{8}}(|q|^{2}+3C)^{\frac{15}{8}}-\frac{3}{8}|q_{z}|^{2}C .
\]
We will show that $f_{C}$ and $g_{C}$ diverge uniformly to infinity when $C\to +\infty$. \\
We first remark that for every $C>0$ the functions $f_{C}$ and $g_{C}$ admit a global minimum. Namely, since $|q|\to +\infty$ when $|z|\to +\infty$ and $q$ is a polynomial, the leading terms of the asymptotic expansions of $f_{C}$ and $g_{C}$ for $|z|\to +\infty$ are given by
\begin{align*}
	f_{C}(z) &=\frac{3}{4}C(|q|^{2})^{\frac{5}{4}}+o(|q|^{2})\\
	g_{C}(z) &=\frac{3}{4}C(|q|^{2})^{\frac{5}{4}}+o(|q|^{2})
\end{align*}
and thus, for $C$ fixed, they are unbounded when $|z|\to +\infty$. Let us denote by $z_{f}(C)$ and $z_{g}(C)$ the point of global minimum of $f_{C}$ and $g_{C}$, respectively. It is sufficient to show that $f_{C}(z_{f}(C))$ and $g_{C}(z_{g}(C))$ tend to infinity when $C \to +\infty$. Since it seems difficult to find an explicit expression for $z_{f}(C)$ and $z_{g}(C)$, we give an abstract argument by considering two different cases. We explain the complete argument for the function $f_{C}$, the other being analogous. \\
Suppose first that $z_{f}(C)$ is uniformly bounded when $C \to +\infty$. In this case, there exists a ball $B_{r}$ of radius $r$ centred at the origin such that $z_{f}(C) \in B_{r}$ for every $C$. Let us denote $M'=\max_{B_{r}}(|q_{z}|^{2})$ and $M=\max_{B_{r}}(|q|^{2})$. Then,
\[
	f_{C}(z_{f}(C)) \geq C^{\frac{19}{8}}(M+3C)^{-\frac{1}{8}}-(M+C)^{\frac{3}{8}}M -\frac{3}{8}CM'
\]
and it is clear that the right-hand side tends to infinity when $C \to +\infty$. \\
Let us then suppose that $z_{f}(C)$ is unbounded. This implies that $|q(z_{f}(C))|$ is diverging as a function of $C$ and taking the asymptotic expansion of $f_{C}$ as a function in the only variable $C$, distinguishing cases where $|q(z_{f}(C))|$ has linear, sublinear or super-linear growth, the claim follows, concluding the proof.\\
%\indent In addition, since the constant $C$ is positive, the boundary conditions are clearly satisfied. We thus conclude that the pair $(u_{1}^{+}, u_{2}^{+})$ is a super-solution of the system.
\end{proof}

\begin{teo}\label{thm:localexistence} Let $d>0$ be the constant appearing in Lemma \ref{lm:subsolution}. For every $R>d$, there exists an analytic solution 
$u^{R}=(u_{1}^{R}, u_{2}^{R})$ of the following boundary value problem
\[
	\begin{cases}
	\Delta u_{1}^{R}=e^{u_{1}^{R}-u_{2}^{R}}-e^{-2u_{1}^{R}}|q|^{2} \\
	\Delta u_{2}^{R}=e^{2u_{2}^{R}}-e^{u_{1}^{R}-u_{2}^{R}} \\
	u_{1}^{R}=\frac{3}{4} \log(|q|)       \ \ \ \ \ \ \ \text{on} \ \ \ \partial B_{R} \\
	u_{2}^{R}=\frac{1}{4}\log(|q|)         \ \ \ \ \ \ \ \text{on} \ \ \  \partial B_{R} \\
	\end{cases}
\]
Moreover, $u_{i}^{-} \leq u_{i}^{R} \leq u_{i}^{+}$.
\end{teo}
\begin{proof} For this proof we remove the dependence on $R$ in the notation. Let us define the sequence of functions $u^{k}=(u_{1}^{k}, u_{2}^{k})$ by 
\begin{equation}\label{eq:proofexistence}
	\begin{cases}
	\Delta u_{1}^{k}=-\Omega_{1}u_{1}^{k-2}+F_{1}(u_{1}^{k-2}, u_{2}^{k-2})+\Omega_{1}u_{1}^{k} 	\\
	\Delta u_{2}^{k}=-\Omega_{2}u_{2}^{k-2}+F_{2}(u_{1}^{k}, u_{2}^{k-2})+\Omega_{2}u_{2}^{k}\\
	u_{1}^{k}=\frac{3}{4}\log(|q|) \ \ \ \  \ \ \ \text{on} \ \ \ \ \partial B_{R} \\
	u_{2}^{k}=\frac{1}{4}\log(|q|) \ \ \ \ \ \ \  \text{on} \ \ \ \  \partial B_{R}
	\end{cases}
\end{equation}
where $\Omega_{i}=\sup\left\{\left|\frac{\partial F_{i}}{\partial u_{i}}\right| \ | \ u \in [u^{0}, u^{-1}] \right\}$ and $u^{0}, u^{-1}$ are a sub-solution and a super-solution of Equation (\ref{eq:system}), respectively. \\
We claim that
\begin{equation}\label{eq:chaininequality}
	u^{0} \leq u^{2} \leq \cdots \leq u^{2k} \leq u^{2k-1} \leq u^{2k-3} \leq \cdots \leq u^{1} \leq u^{-1}
\end{equation}
for every $k\geq 1$. Then the result follows from the Schauder fixed point theorem \cite[p.660]{Amann} applied to the differential operator defined by (\ref{eq:proofexistence}) on the Banach space of pairs of H\"older functions on $B_{R}$, standard bootstrap arguments and Morrey's regularity theorem \cite[p.198]{Morrey_regularity}. Moreover, the above inequalities imply that the solution is bounded from above by (any of) the super-solution(s) and from below by (any of) the sub-solution(s).\\
\indent Let us now prove the claim \eqref{eq:chaininequality}. We first show that 
\[
	u_{1}^{0} \leq u_{1}^{1} \leq u_{1}^{-1} \ .
\]
By definition and the monotonicity properties of the function $F_{1}$, the following equations hold
\begin{align*}
	\Delta u_{1}^{1}&=-\Omega_{1}u_{1}^{-1}+\Omega_{1}u_{1}^{1}+F_{1}(u_{1}^{-1},u_{2}^{-1}) \\
	\Delta u_{1}^{0}&\geq F_{1}(u_{1}^{0},u_{2}^{0}) \geq F_{1}(u_{1}^{0}, u_{2}^{-1})-\Omega_{1}u_{1}^{0}+\Omega_{1}u_{1}^{0} \\
	\Delta u_{1}^{-1} &\leq F_{1}(u_{1}^{-1}, u_{2}^{-1})-\Omega_{1}u_{1}^{-1}+\Omega_{1}u_{1}^{-1}
\end{align*}
and the claim follows from the maximum principle (for Sobolev functions, at this first iteration of the process) applied to differences of the above equations. Namely,
\[
	\Delta (u_{1}^{1}-u_{1}^{-1}) \geq \Omega_{1}(u_{1}^{1}-u_{1}^{-1})
\]
and the maximum principle implies that $u_{1}-u_{1}^{-1} \leq 0$. 
Similarly,
\[
	\Delta (u_{1}^{1}-u_{1}^{0}) \leq \Omega_{1}(u_{1}^{1}-u_{1}^{0})+F_{1}(u_{1}^{-1}, u_{2}^{-1})-F_{1}(u_{1}^{0}, u_{2}^{-1})-\Omega_{1}(u_{1}^{-1}-u_{1}^{0}) \leq \Omega_{1}(u_{1}^{1}-u_{1}^{0})
\]
by definition of $\Omega_{1}$, and from the maximum principle we deduce that $u_{1}^{1}-u_{1}^{0} \geq 0$. \\
\indent The reasoning is similar also for the second components. In this case we have
\begin{align*}
	\Delta u_{2}^{1}&=-\Omega_{2}u_{2}^{-1}+F_{2}(u_{1}^{1},u_{2}^{-1})+\Omega_{2}u_{2}^{1}\\
	\Delta u_{2}^{0}&\geq F_{2}(u_{1}^{0}, u_{2}^{0})\geq F_{2}(u_{1}^{1}, u_{2}^{0})-\Omega_{2}u_{2}^{0}+\Omega_{2}u_{2}^{0} \\
	\Delta u_{2}^{-1} &\leq F_{2}(u_{1}^{-1}, u_{2}^{-1}) \leq  F_{2}(u_{1}^{1},u_{2}^{-1})-\Omega_{2}u_{2}^{-1}+\Omega_{2}u_{2}^{-1}
\end{align*}
and the inequalities $u_{2}^{0} \leq u_{2}^{1} \leq u_{2}^{-1}$ follow from the maximum principle as above. \\
With the same argument, one can show that  
\[
	u^{0} \leq u^{2} \leq u^{1}
\]
and the chain of inequalities (\ref{eq:chaininequality}) follows then by induction. (Note that elliptic regularity implies that the functions $u_j^k$ are increasingly smooth, so that for $k \geq 2$, $u_j^k \in C^2$.)
\end{proof}

We now deduce the existence of an analytic solution $u=(u_{1}, u_{2})$ to Equation (\ref{eq:system}) defined on the whole complex plane, via a limiting argument. \\

By Theorem \ref{thm:localexistence}, we obtain a sequence of analytic functions $u_{i}^{R}$ defined on the ball $B_{R}$ for every $R>d$. By using the fact that $u_{i}^{R}$ is bounded between the sub-solution and the super-solution for every $R$, we deduce a uniform bound on $\Delta u_{i}^{R}$ on every compact set, which is independent of $R$. By elliptic regularity, the functions $u_{i}^{R}$ are  bounded in the $C^{1,\alpha}$ norm, uniformly on every compact set. By Ascoli-Arzel\'a, this implies that the sequences $u_{i}^{R}$ converge in the $C^{1}$ norm on compact sets for every $i=1,2$. In particular, the limit functions $u_{i}$ are defined over all $\C$ and are weak solutions of the system. By elliptic regularity of Poisson equations (applied to each single equation), we deduce that $u_{i}$ are smooth and hence are strong solutions of Equation (\ref{eq:system}). By Morrey's results \cite{Morrey_regularity}, the functions $u_{i}$ are analytic. 
Moreover, by construction we have
\begin{equation}\label{eq:bound_u1}
	\frac{3}{8}\log(|q|^{2}) \leq u_{1} \leq \frac{3}{8}\log(|q|^{2}+C) 
\end{equation}
\begin{equation}\label{eq:bound_u2}
	\frac{1}{8}\log(|q|^{2}) \leq u_{2} \leq \frac{1}{8}\log(|q|^{2}+3C)  \ .
\end{equation}
Of course, here it is important that we found in Lemma~\ref{lm:supersolution} a single constant $C$ for which the right-hand sides were supersolutions on $B_R$ for all $R$ sufficiently large.

%\begin{cor} $u_{1}-3u_{2}=o(1)$ for $|z| \to +\infty$.
%\end{cor}
%\begin{proof} By the previous estimates, we have
%\[
%	\frac{3}{8}\log\left(\frac{|q|^{2}}{|q|^{2}+3C} \right) \leq u_{1}-3u_{2} \leq \frac{3}{8}\log %\left( \frac{|q|^{2}+C}{|q|^{2}} \right)
%\]
%and the result follows
%\end{proof}

\begin{cor}\label{cor:decay_uj} There exist constants $A,R>0$ and an exponent $\alpha>1$ as follows. If the $|q|^{\frac{1}{2}}$-distance of a point $p \in \C$ from the zeros of $q$ is $r>R$, then
\[
	0\leq u_{1}(p)-\frac{3}{8}\log(|q|^{2})\leq Ar^{-\alpha}
\]
\[
	0 \leq u_{2}(p)-\frac{1}{8}\log (|q|^{2}) \leq Ar^{-\alpha} \ .
\]
\end{cor}
\begin{proof} Outside a disc $D$ containing the zeros of $q$, the polynomial $q$ is comparable to $z^{n}$ up to multiplicative constants, where $n$ is the degree of $q$. As a consequence, the $|q|^{\frac{1}{2}}$-distance $r$ of a point $p \notin D$ from a zero of $q$ is bounded from above by a multiple of the $|z|^{\frac{n}{2}}$-distance of $p$ from the origin. We deduce that there exists a constant $c>0$ such that
\[
	r<c|p|^{\frac{(n+4)}{4}} \ .
\]
Since $|q|$ is bounded also from below by a multiple of $|z|^{n}$, we have
\[
	|q(p)|>c'|p|^{n}\geq c'' r^{\frac{4n}{(n+4)}} \ .
\]
From the previous theorem, we obtain that
\[
	u_{1}-\frac{3}{8}\log( |q|^{2}) \leq u_{1}^{+}-\frac{3}{8} \log(|q|^{2}) \leq \frac{M}{|q|^{2}} \leq \frac{A}{r^{\frac{8n}{n+4}}}
\]
\[
	u_{2}-\frac{1}{8}\log( |q|^{2}) \leq u_{2}^{+}-\frac{1}{8} \log(|q|^{2}) \leq \frac{M}{|q|^{2}} \leq \frac{A}{r^{\frac{8n}{n+4}}} \ .
\]
By noticing that $\alpha=\frac{8n}{n+4}>1$ for every $n\geq 1$, the result follows if we fix $R$ big enough such that $r>R$ implies $p \notin D$.
\end{proof}

\begin{remark}[On uniqueness]\label{rmk:uniqueness} By work of Mochizuki (\cite{Mochizuki_harmonicbundles}), the solution $H=(h_{1}, h_{2}^{-1},h_{1}^{-1},h_{2})$ found above is the unique diagonal solution of the self-duality equation 
\[
    F_{H}+[\varphi, \varphi^{*H}]=0
\]
on the Higgs bundle $(\sE, \varphi)$ on $\C\Pp^{1}$. Moreover, in recent work (\cite{Li_Mochizuki_cyclic_noncompact}), Li and Mochizuki applied similar sub- and super-solution techniques to show existence and uniqueness of diagonal solutions of Hitchin's self-duality equation on every cyclic Higgs bundle with wild singularities over non-compact Riemann surfaces.
\end{remark}

\section{Geometry of the minimal surface}
In this section we study the geometry of the minimal surface with polynomial growth induced by the harmonic metric found in Section \ref{sec:existence}. 

In particular, the results in this section will imply Theorem~\ref{thmA}. Moreover, for $S$ the minimal surface in the symmetric space $\Sp(4,\R)/\U(2)$ associated to a monic polynomial $q$ of degree $n \geq 1$, we find in Theorem~\ref{thm:asymptotic_flats} that $S$ is asymptotic to $2(n+4)$ maximal flats in $\Sp(4,\R)/\U(2)$; two consecutive flats that are asymptotic to $S$ share four adjacent Weyl chambers (Proposition~\ref{prop: shared Weyl chambers}). Intrinsically, by Proposition~\ref{prop:quasi-iso}, the metric on $S$ induced by this immersion is asymptotically $4|q|^{\frac{1}{2}}$, up to an (additive) error that decays at a rate of $O(|q|^{\frac{1}{4}})$.

We organize the argument as follows.  After some preliminaries, we display the solution for the case of $q_0=dz^4$.  Then we choose good charts away from a compact set which contains the zeroes that respect the geometry that $q$ imposes on the plane $\C$: each such plane cuts off a region in $\C$ which is roughly a half-plane in the $|q|$ metric, positioned to develop in a controlled manner, with overlaps that also develop in a controlled way. In those charts, we find, roughly, that the minimal surface in the symmetric space $\Sp(4,\R)/\U(2)$ may be well-approximated by an isometric image of the flat defined by $q_0$. Describing those asymptotics carefully, up to some estimates deferred until the end of the section, occupies the first half of this section, and culminates in the proof of 
Theorem~\ref{thm:asymptotic_flats}, Proposition~\ref{prop: shared Weyl chambers} and Proposition~\ref{prop:quasi-iso}.  A careful treatment of the error estimates completes the section.

\subsection{Construction of the minimal surface}\label{subsec:construction_min_surface}
In Section \ref{sec:background}, we recalled how a solution to Hitchin's equation induces a harmonic map into a symmetric space. The construction goes as follows. 
Let $H$ denote the associated Hermitian metric on $\sE$ (guaranteed by Theorem~\ref{thm:existence}).
Let $\{N(z)\}_{z \in \C}$ be a holomorphic, $\Omega$-symplectic, $Q$-adapted and $H$-unitary frame for the bundle $\sE$. The frame $\{N(z)\}_{z \in \C}$ is not parallel for the $\Sp(4, \R)$-connection $\nabla=D_{H}+\varphi+\varphi^{*H}$. Fix a base point $z_{0}$. We denote by $\{\mathcal{N}(z)\}_{z \in \C}$ the parallel transport of the frame $N(z_{0})$ via the connection $\nabla$. By expressing the metric $H$ in the frame $\{\mathcal{N}(z)\}_{z \in \C}$, we obtain a map
\begin{align*}
	f:\C &\rightarrow \SL(4, \C))/\SU(4)\\
	 	z &\mapsto H^{\mathcal{N}(z)} \ .
\end{align*}
We then notice that the image of $f$ is in fact contained in the copy of $\Sp(4,\R)/\U(2)$ inside $\Sp(4, \C)/\SU(4)$, consisting of $\Omega$-symplectic and $Q$-symmetric hermitian matrices with determinant $1$.

Let us now find an explicit expression for $H^{\mathcal{N}(z)}$ (cf. \eqref{eqn: computation of harmonic map}). Let $\{F(z)\}_{z \in \C}$ be the standard holomorphic frame of $\sE$ where $H^{F(z)}=\diag(h_{1},h_{2}^{-1}, h_{1}^{-1}, h_{2})$. We denote by $\{\mathcal{F}(z)\}_{z \in \C}$ the parallel transport of $F(z_{0})$ with respect to $\nabla$. For every $z \in \C$, we can find a matrix $\psi(z)$ such that
\[
	\mathcal{F}(z)\psi(z)=F(z) \ ,
\]
i.e. $\psi(z)$ expresses the change of frame from $\mathcal{F}(z)$ to $F(z)$ at every point. Let $\gamma(s)=z_{0}+se^{i\theta}$ be a path connecting the base point $z_{0}$ with $z$. We observe that the one-parameter family of matrices $\psi(s)=\psi(\gamma(s))$ satisfies the ordinary differential equation
\begin{equation}\label{eq:psi}
	\begin{cases}
		\frac{d\psi}{ds}(s)=\psi(s)(e^{i\theta}U+e^{-i\theta}V) \\
		\psi(0)=\Id
	\end{cases} \ ,
\end{equation}
where we denoted by
\[
 	U=D_{H}+\varphi \ \ \ \text{and} \ \ \ \ V=\varphi^{*H}
\]
the $(1,0)$-part and the $(0,1)$-part of the connection $\nabla$, respectively. To see this, note that equation (\ref{eq:psi}) is a direct consequence of the fact that $\{\mathcal{F}(z)\}_{z \in \C}$ is parallel. Namely, 
\begin{align*}
	\mathcal{F}(\gamma(s))\frac{d\psi}{ds}(s)&=\nabla(\mathcal{F}\psi)(\gamma(s))\\
	&=\nabla_{\frac{\partial}{\partial s}}F(\gamma(s))\\
	&=F(\gamma(s))(e^{i\theta}U+e^{-i\theta}V)\\
	&=\mathcal{F}(\gamma(s))\psi(\gamma(s))(e^{i\theta}U+e^{-i\theta}V) \ .
\end{align*}
Moreover, since the connection $\nabla$ is flat, there exists a constant matrix $P \in \SL(4, \C)$ such that
$\mathcal{N}(z)=\mathcal{F}(z)P$. In fact, $P$ is the change of frame between $N(z_{0})$ and $F(z_{0})$, i.e. $N(z_0)=F(z_0)P$. We thus deduce that
\begin{align*}
	f(z)&=H^{\mathcal{N}(z)}=H^{\mathcal{F}(z)P}=H^{F(z)\psi^{-1}(z)P}=\overline{P^{t}(\psi(z)^{-1})^{t}}H^{F(z)}(\psi(z))^{-1}P \\
	 &=\overline{P^{t}(\psi(z)^{-1})^{t}}\diag(h_{1}(z), h_{2}^{-1}(z), h_{1}^{-1}(z), h_{2}(z))(\psi(z))^{-1}P \ .
\end{align*}	

We notice in particular that the geometry of the minimal surface will depend not only on the functions $h_{i}(z)$, but also on the solution to the ODE (\ref{eq:psi}). This will play a fundamental role in Section \ref{sec:general_case}.

\subsection{The case of constant quartic differential}\label{subsec:constant_case}
In the special case, when the quartic differential is constant, the solution of the ODE (\ref{eq:psi}) can be written explicitly and the minimal surface turns out to be a flat in $\Sp(4,\R)/\U(2)$. \\

Up to biholomorphisms of $\C$, we can suppose that $q=dz^{4}$. As mentioned in Section \ref{sec:existence}, the solution to Hitchin's equations in this case is
\[
	H=\diag(1,1,1,1) \ .
\]
As a consequence, the system of ODE (\ref{eq:psi}) simplifies into
\begin{equation}\label{eq:ODEconstant}
	\begin{cases}
		\frac{d\psi_{0}}{ds}(s)=\psi_{0}(s)(e^{i\theta}U_{0}+e^{-i\theta}V_{0}) \\
		\psi_{0}(0)=\Id 
	\end{cases}
\end{equation}
where we are using the notation $\psi_{0}$ to indicate that we are dealing with the special case of constant quartic differential. Moreover, the $(1,0)$-part and the $(0,1)$-part of the connection $\nabla_{0}=D_{H}+\varphi+\varphi^{*H}$ are respectively
\[
	U_{0}=\begin{pmatrix}
			0 & 0 & 1 & 0 \\
			0 & 0 & 0 & 1 \\
			0 & 1 & 0 & 0 \\
			1 & 0 & 0 & 0 \\
		  \end{pmatrix}  \ \ \ \ \ \ \ 
	V_{0}=\begin{pmatrix}
			0 & 0 & 0 & 1 \\
			0 & 0 & 1 & 0 \\
			1 & 0 & 0 & 0 \\
			0 & 1 & 0 & 0 \\
		  \end{pmatrix} \ ,
\]
as the Chern connection $D_{H}$ of $H$ vanishes in these coordinates, since $H$ is constant. Therefore, equations \eqref{eq:ODEconstant} become a system of ODE with constant coefficients, that can be explicitly integrated:
\[
	\psi_{0}(s)=\exp \left(s\begin{pmatrix}
					       0 & 0 & e^{i\theta} & e^{-i\theta} \\
					       0 & 0 & e^{-i\theta} & e^{i\theta} \\
					       e^{-i\theta} & e^{i\theta} & 0 & 0 \\
					       e^{i\theta} & e^{-i\theta} & 0 & 0 \\					
					 \end{pmatrix}\right) \ .
\]
By observing that
\[
	S^{-1}\begin{pmatrix} 
			0 & 0 & e^{i\theta} & e^{-i\theta} \\
		    0 & 0 & e^{-i\theta} & e^{i\theta} \\
			e^{-i\theta} & e^{i\theta} & 0 & 0 \\
			e^{i\theta} & e^{-i\theta} & 0 & 0 \\	
		  \end{pmatrix}S= \begin{pmatrix}
		  					2\cos(\theta) & 0 & 0 & 0 \\
		  					0 & -2\sin(\theta) & 0 & 0 \\
		  					0 & 0 & -2\cos(\theta) & 0 \\
		  					0 & 0 & 0 & 2\sin(\theta) \\	  
		   					\end{pmatrix}
\]
for a constant unitary matrix 
\[
	S^{-1}=\frac{1}{2}\begin{pmatrix} 
		1 & 1 & 1 & 1 \\
		1 & -1 & -i & i \\
		1 & 1 & -1 & -1 \\
		1 & -1 & i & -i \\
		\end{pmatrix} \ ,
\]
we obtain that 
\[
	\psi_{0}(s)=S\begin{pmatrix}
				 e^{2s\cos(\theta)} & 0 & 0 & 0 \\
		  		 0 & e^{-2s\sin(\theta)} & 0 & 0 \\
		  		 0 & 0 & e^{-2s\cos(\theta)} & 0 \\
		  		 0 & 0 & 0 & e^{2s\sin(\theta)} \\	  
		   		\end{pmatrix}S^{-1} \ .
\]
If we fix the origin as base point $z_{0}$ in the definition of the harmonic map $f_{0}$ (see Section \ref{sec:background}), and write $z=se^{i\theta}$, then $f_{0}: \C \rightarrow \Sp(4, \R)/\U(2)$ is given by
\begin{align*}
	f_{0}(z)&=\overline{(\psi_{0}(z)^{-1})^{t}}\diag(1,1,1,1)\psi_{0}(z)^{-1}\\		
			&=\overline{(S^{-1})^{t}}\diag(e^{-4\Re(z)}, e^{4\Im(z)}, e^{4\Re(z)}, e^{-4\Im(z)})S^{-1}\\
			&=S \boldsymbol{\cdot} \diag(e^{-4\Re(z)}, e^{4\Im(z)}, e^{4\Re(z)}, e^{-4\Im(z)}) \ ,
\end{align*}
where we denoted with $\boldsymbol{\cdot}$ the action of an element $g\in \SL(4, \C)$. This shows that the image of $f_{0}$ is a maximal flat in the symmetric space. 

\subsection{The general case}\label{sec:general_case} 
In order to study the general case, i.e. when the quartic differential $q$ is an arbitrary polynomial $q(z)$ of degree $n\geq 1$, the main idea is to estimate the solution to Equation (\ref{eq:psi}) by comparing it with the solution to Equation (\ref{eq:ODEconstant}). In fact, the complement of a compact set containing the roots of $q(z)$ is covered by $(n+4)$ charts, which are conformal to the upper-half plane, where the quartic differential $q$ is constant. This suggests that in each chart the solution to Equation (\ref{eq:psi}) should look like the solution to Equation (\ref{eq:ODEconstant}), at least when we are far enough from the zeros of the polynomial $q(z)$. We will thus describe the asymptotic geometry of the associated minimal surface and we will focus, in particular, on studying an interesting \enquote{Stokes phenomenon}, that already occurred for affine spheres with polynomial Pick differential (\cite{DWpolygons}).   

\subsubsection{Standard half-planes and rays}
Given a quartic differential $q$, a natural coordinate $w$ for $q$ is a local coordinate on an open set of $\C$ in which $q=dw^{4}$. Such a coordinate always exists locally away from the zeros of $q$, as it is possible to choose a holomorphic fourth root of $q$ and define
\[
	w(z)=\int_{z_{0}}^{z} q^{\frac{1}{4}} \ .
\]
Notice, in particular, that a natural coordinate is not unique, but every two natural coordinates for $q$ differ by a multiplication by a fourth root of unity and an additive constant. \\

We define a $q$-half-plane (or a standard half-plane, when the reference to the differential $q$ is obvious) as a pair $(U,w)$, where $U\subset \C$ is open and $w$ is a natural coordinate for $q$ that maps diffeomorphically $U$ to the upper-half plane $\{\Im(w)>0\}$. Note that $U$ then determines $w$ up to addition of a real constant. \\

A path in $\C$ whose image in a natural coordinate for $q$ is a Euclidean ray with angle $\theta$ is called {\it a $q$-ray of angle $\theta$}. (Note that the angle is well-defined mod $\frac{\pi}{2}$.) This means that in a suitable natural coordinate, a $q$-ray is parameterized by $t \mapsto b+e^{i\theta}t$.\\
Similarly, a $q$-quasi-ray with angle $\theta$ is a path that can be parameterized so that its image in a natural coordinate $w$ is $t \mapsto e^{i\theta}t+o(t)$. \\

It turns out that every monic polynomial quartic differential $q$ admits a finite number of $q$-half-planes that cover the complement in $\C$ of a compact set containing the zeros of $q$. 
\begin{prop}[\cite{DWpolygons}]\label{prop:halfplanes} Let $q$ be a monic polynomial quartic differential and let $K$ be a compact subset of $\C$ containing the zeros of $q$. Suppose $q$ has degree $n \geq 1$. Then, there exist a compact subset $K'\supseteq K$ and a collection of $(n+4)$ $q$-half-planes $\{(U_{k}, w_{k})\}_{k=1, \dots, n+4}$ with the following properties:
\begin{enumerate}[i)]
	\item the complement of $\bigcup_{k}U_{k}$ is $K'$;
	\item the ray $\{\arg(z)=\frac{2\pi k}{n+4}\}$ is eventually contained in $U_{k}$;
	\item the rays $\{\arg(z)=\frac{2\pi(k\pm1)}{n+4}\}$ are disjoint from $U_{k}$;
	\item on $U_{k}\cap U_{k+1}$ we have $w_{k+1}=iw_{k}+c$ for some constant $c$ and each $w_{k}$, $w_{k+1}$ maps this intersection onto a sector of angle $\frac{\pi}{2}$ based at a real point. 
	\item any Euclidean ray in $\C$ is a $q$-quasi ray and is eventually contained in $U_{k}$ for some $k$.
\end{enumerate}
\end{prop} 

Let $r: \C \rightarrow \R^{+}$ be the $|q|^{\frac{1}{2}}$-distance from the zeros of $q$. We recall the following result that will be used in Section \ref{subsubsec:estimates}. 
\begin{prop}[\cite{DWpolygons}]\label{prop:good_halfplanes}Let $q$ be a monic polynomial quartic differential and let $K$ be a compact set containing the zeros of $q$. Then there are constant $A,a,R_{0}$ with $a>0$ so that for every point $p \in \C$ with $r(p)>R_{0}$, there exists a $q$-half-plane $(U,w)$ with $U\cap K=\emptyset$ such that $\Im(w(p))\geq r(p)-A$. In addition, on the boundary of this half-plane we have $r(x)\geq a|\Re(w(x))|$, for $x$ large.
\end{prop}

We remark that the monic condition in the above propositions is not restrictive, as every polynomial can be made monic via a biholomorphic change of coordinates on $\C$.  

\subsubsection{Comparing $\psi_{0}$ and $\psi$}
Let us fix the origin of $\C$ as base point $z_{0}$ in the construction of the harmonic map (see Section \ref{subsec:construction_min_surface}). By Proposition \ref{prop:halfplanes}, any point $z$ far enough from $z_{0}$ is connected by a ray $\gamma(s)=se^{i\theta}$, which is definitely contained in a standard half-plane. Therefore, there exists a time $s_{0}>0$, such that the ray $\gamma(s)$ lies in a standard half-plane for every $s \geq s_{0}$. We can write thus a differential equation satisfied by $\psi\psi_{0}^{-1}(s)$ for $s \geq s_{0}$ using Equation (\ref{eq:psi}) and Equation (\ref{eq:ODEconstant}): 
\begin{equation}\label{eq:ODEdifference}
	\begin{cases}
	\frac{d\psi\psi_{0}^{-1}}{ds}(s)=\psi\psi_{0}^{-1}(s)(\psi_{0}(s)R\psi_{0}(s)^{-1}) \\
	\psi\psi_{0}^{-1}(s_{0})=A_{0}
	\end{cases} \ ,
\end{equation}
for some matrix $A_{0} \in \Sp(4, \C)$, which represents the difference between $\psi_{0}$ and $\psi$ at the point $s_{0}$. In Equation (\ref{eq:ODEdifference}) we have denoted
\begin{align*}
	R &=\psi^{-1}(s)\frac{d\psi}{ds}(s)-\psi_{0}^{-1}(s)\frac{d\psi_{0}}{ds}(s) \\
 	  &=e^{i\theta}(U-U_{0})+e^{-i\theta}(V-V_{0})+e^{i\theta}D_{H} \ 
\end{align*}
the error between the connection $\nabla_{0}$ and $\nabla$. Let us denote by $\uu_{j}$ the functions
\[
	\uu_{1}=u_{1}-\frac{3}{8}\log(|q|^{2}) \ \ \  \text{and} \ \ \ \uu_{2}=u_{2}-\frac{1}{8}\log(|q|^{2})\ ,
\]
which represent the error between the solution to Equation (\ref{eq:system}) and the particular solution in the case of constant quartic differential. By Corollary \ref{cor:decay_uj}, the function $\uu_{j}$ decays as $|z| \to +\infty$. We can now write the error $R$ in terms of the function $\uu_{j}$ in the natural coordinate chart $w$ of the half-plane. \\
\indent First, since $q=dw^{4}$, the term $U-U_{0}$ vanishes. 
Moreover, since the metric $H$ is diagonal, it is easy to verify that the Chern connection $D_{H}$ is
\[
	D_{H}=H^{-1}\partial H= \begin{pmatrix}
							\partial \uu_{1} & 0 & 0 & 0 \\
							0 & -\partial \uu_{2} & 0 & 0 \\
							0 & 0 & -\partial \uu_{1} & 0 \\
							0 & 0 & 0 & \partial \uu_{2} \\
							\end{pmatrix} \ ,
\]
because in a natural coordinate $\uu_{j}=u_{j}$. Finally, by definition of $V$ and $V_{0}$, we have
\[
	V-V_{0}=\begin{pmatrix}
				0 & 0 & 0 & e^{\uu_{1}-\uu_{2}}-1 \\
				0 & 0 & e^{\uu_{1}-\uu_{2}}-1 & 0 \\
				e^{-2\uu_{1}}-1 & 0 & 0 & 0 \\
				0 & e^{2\uu_{2}}-1 & 0 & 0 \\		
			\end{pmatrix} \ .
\]
Let us denote $D=\diag(e^{2s\cos(\theta)}, e^{-2s\sin(\theta)}, e^{-2s\cos(\theta)}, e^{2s\sin(\theta)})$ and $R'=S^{-1}RS$, where $S$ is the unitary matrix introduced in Section \ref{subsec:constant_case}. We can then write the error term as
\begin{align*}
	\Theta(s)=\psi_{0}(s)R\psi_{0}(s)^{-1}&=SDS^{-1}RSD^{-1}S^{-1}=SDR'D^{-1}S^{-1}\\
								&=SD(R_{1}'+R_{2}')D^{-1}S^{-1} 
\end{align*}
with
\begin{align}
	R_{1}'&=e^{i\theta}S^{-1}\diag(\partial \uu_{1}, -\partial \uu_{2}, -\partial \uu_{1}, \partial \uu_{2})S\\ \notag
		&=\frac{e^{i\theta}}{4}\scalebox{0.80}{$\begin{pmatrix}
			0 & (1-i)\partial(\uu_{1}+\uu_{2}) & 2\partial(\uu_{1}-\uu_{2}) & (1+i)\partial(\uu_{1}+\uu_{2}) \\
			(1+i)\partial(\uu_{1}+\uu_{2}) & 0 & (1-i)\partial(\uu_{1}+\uu_{2}) & 2\partial(\uu_{1}-\uu_{2})\\
			2\partial(\uu_{1}-\uu_{2}) & (1+i)\partial(\uu_{1}+\uu_{2}) & 0 & (1-i)\partial(\uu_{1}+\uu_{2})\\
			(1-i)\partial(\uu_{1}+\uu_{2}) & 2\partial(\uu_{1}-\uu_{2}) & (1+i)\partial(\uu_{1}+\uu_{2}) & 0 \\
			\end{pmatrix}$}
\end{align}  \label{eqn: R1prime}
and
\begin{align*}
	&R_{2}'=e^{-i\theta}S^{-1}(V-V_{0})S \ .
%		  &=\displaystyle{\frac{e^{-i\theta}}{4}}\scalebox{0.6}{$\begin{pmatrix}
%		  e^{-2u_{1}}+2e^{u_{1}-u_{2}}+e^{2u_{2}}-4 & e^{-2u_{1}}-e^{2u_{2}} & e^{-2u_{1}}+e^{2u_{2}}-2e^{u_{1}-u_{2}} & e^{-2u_{1}}-e^{2u_{2}}\\
%		  ie^{2u_{2}}-ie^{-2u_{1}} & -ie^{-2u_{1}}-2ie^{u_{1}-u_{2}}-ie^{2u_{2}}-4i & -ie^{-2u_{1}}+ie^{2u_{2}} &-ie^{-2u_{1}}-ie^{-2u_{2}}+2ie^{u_{1}-u_{2}}\\
%		  2e^{u_{1}-u_{2}}-e^{-2u_{1}}-e^{2u_{2}} & -e^{-2u_{1}}+e^{2u_{2}} & -e^{-2u_{1}}-e^{2u_{2}}-2e^{u_{1}-u_{2}}+4 & -e^{-2u_{1}}+e^{2u_{2}}\\
%		  ie^{-2u_{1}}-ie^{2u_{2}} & ie^{-2u_{1}}+ie^{2u_{2}}-2ie^{u_{1}-u_{2}} & ie^{-2u_{1}}-ie^{2u_{2}} & ie^{-2u_{1}}+ie^{2u_{2}}+2ie^{u_{1}-u_{2}}-4i\\
%		  	\end{pmatrix}$} \ .
\end{align*}	
If we introduce the notation $u_{3}=-u_{2}$ and $u_{4}=-u_{1}$, an elementary but tedious computation shows that
\[
	R'_{kl}=\frac{e^{i\theta}}{4}\sum_{j=0}^{3}i^{(k-l)j}\partial \uu_{j+1}+\frac{e^{-i\theta}i^{1-k}}{4}\sum_{j \in \Z_{4}}i^{(k-l)j}e^{\uu_{j}-\uu_{j+1}} \ \ \ \text{for $k\neq l$}
\]
and
\[
	R'_{kk}=\frac{e^{-i\theta}(-i)^{k-1}}{4}(e^{-2\uu_{1}}+2e^{\uu_{1}-\uu_{2}}+e^{2\uu_{2}}-4) \ .
\]

In Section \ref{subsubsec:estimates} we will prove the following: 
\begin{prop}\label{prop:estimates_error} Let $r$ be the distance from the zeros of $q$. Then for $r \to +\infty$
\[
	R'_{kl}=O\left(\frac{e^{-2|1-i^{k-l}|r}}{\sqrt{r}}\right) \ \ \ \text{if $k\neq l$} \ ,
\] 
and
\[
	R'_{kk}=o\left(\frac{e^{-2\sqrt{2}r}}{\sqrt{r}}\right) \ .
\]
\end{prop}

These exponential decays allow us to find a limit of $\psi\psi_{0}^{-1}(s)$ along rays in any stable direction.

\begin{defi} Let $\gamma(t)=b+e^{i\theta}t$ be a ray in a standard half-plane. The direction of the ray is the angle $\theta \in [0,\pi]$. We say that the ray is stable if $\theta \notin \{0,\frac{\pi}{4}, \frac{\pi}{2}, \frac{3\pi}{4}, \pi \}$. Similarly, a quasi-ray is stable, if the direction of the associated ray is stable.
\end{defi}

The possible directions of stable rays form four intervals of length $\frac{\pi}{4}$ which we denote by
\[
	J_{++}=\left(0, \frac{\pi}{4}\right) \ \ J_{+}=\left(\frac{\pi}{4}, \frac{\pi}{2}\right) \ \ J_{-}=\left(\frac{\pi}{2}, \frac{3\pi}{4}\right) \ \ J_{--}=\left( \frac{3\pi}{4}, \pi\right) \ .
\]

The stability of rays and quasi-rays is related to the convergence of $\psi\psi_{0}^{-1}(\gamma(s))$:
\begin{lemma}\label{lm:existence_limits} If $\gamma$ is a stable ray or quasi-ray, then the limit $\lim_{s \to +\infty}\psi\psi_{0}^{-1}(\gamma(s))$ exists. Furthermore, among all such rays  only four limits are seen, i.e. there exist $L_{++}, L_{+}, L_{-}, L_{--} \in \Sp(4, \C)$ such that
\[	
	\lim_{s \to +\infty} \psi\psi_{0}^{-1}(\gamma(s))=\begin{cases}
						L_{++}  \ \ \ \text{if} \ \ \theta \in J_{++} \\
						L_{+} \ \ \ \ \ \text{if} \ \ \theta \in J_{+} \\
						L_{-} \ \ \ \ \ \text{if} \ \ \theta \in J_{-} \\
						L_{--} \ \ \ \text{if} \ \ \theta \in J_{--}
						\end{cases}
\]
\end{lemma}
\begin{proof} First we consider rays, and at the end of the proof we show that quasi-rays have the same behaviour.\\
\indent Let $\gamma$ be a ray and let us write $G(s)=\psi\psi_{0}^{-1}(\gamma(s))$. It satisfies the ODE
\[
	\begin{cases}
	\frac{dG}{ds}(s)=G(s)\Theta(s) \\
	G(0)=A_{0}
	\end{cases}
\]
for some $A_{0} \in \Sp(4, \C)$. Recalling the definition of $\Theta(s)=SDR'D^{-1}S^{-1}$, the decay of the error $\Theta(s)$ is determined by comparing the decay of $R'$ and the growth of the diagonal matrix $D=\diag(e^{2s\cos(\theta)},e^{-2s\sin(\theta)}, e^{-2s\cos(\theta)}, e^{2s\sin(\theta)})$. Conjugating $R'$ by the diagonal matrix $D(s)$ multiplies the entry $R'_{kl}$ by 
\[
	\lambda_{kl}=\exp\left(2s\left(\cos\left(\theta-\frac{(k-1)\pi}{2}\right)\right)-\cos\left(\theta-\frac{(l-1)\pi}{2}\right)\right) \ .
\]
Combining this with Proposition \ref{prop:estimates_error}, we deduce that for any stable ray, we have a definite exponential decay in the equation satisfied by $G$, i.e.
\[
	G(s)^{-1}G'(s)=O\left(\frac{e^{-\alpha s}}{\sqrt{s}}\right) 
\]
for some $\alpha>0$. Standard ODE techniques (see \cite[Appendix B]{DWpolygons}) then show that the limit $\lim_{s \to +\infty}G(s)$ exists.\\
\indent Now suppose that $\gamma_{1}$ and $\gamma_{2}$ are stable rays with angles $\theta_{1}$ and $\theta_{2}$ that belong to the same interval ($J_{++}$, $J_{+}$, $J_{-}$,  or $J_{--}$). We will show that $G_{1}(s)G_{2}(s)^{-1} \to \Id$ as $s \to +\infty$, where $G_{i}(s)=\psi\psi_{0}^{-1}(\gamma_{i}(s))$. This means that the limit does not depend on the direction of the ray in a same interval, thus concluding the proof.\\
\indent For any $s>0$, let $\eta_{s}(t)=(1-t)\gamma_{1}(s)+t\gamma_{2}(s)$ be the constant speed parameterization of the segment from $\gamma_{1}(s)$ and $\gamma_{2}(s)$. Let $g_{s}(t)=(\psi\psi_{0}^{-1}(\eta_{s}(0))^{-1}\psi\psi_{0}^{-1}(\eta_{s}(t))$, which satisfies
\[
	\begin{cases}
	g_{s}^{-1}(t)g_{s}'(t)= \Theta(\eta_{s}(t))\eta_{s}'(t) \\ 
	g_{s}(0)= \Id \\
	g_{s}(1)=G_{1}(s)^{-1}G_{2}(s)
	\end{cases} \ .
\]
Since $|\eta_{s}'(t)|=O(s)$, the analysis above shows that $g_{s}^{-1}g_{s}'(t)=O(\sqrt{s}e^{-\alpha s})$, for some $\alpha >0$, because the path $\eta_{s}(t)$ never crosses an unstable direction. In particular, by making $s$ large enough we can arrange for $g_{s}(t)^{-1}g'_{s}(t)$ to be uniformly small for all $t \in [0,1]$. Once again standard ODE methods (\cite[Lemma B.1 (i)]{DWpolygons}) give the desired convergence,
\[
	G_{1}(s)^{-1}G_{2}(s)=g_{s}(1) \to \Id \ \ \ \text{as} \ \ s \to +\infty \ .
\]
\indent Finally, suppose that $\gamma_{1}$ is a stable quasi-ray, and $\gamma_{2}$ is the ray that approximates $\gamma_{1}$ with direction $\theta$. We can study as above the homotopy $\eta_{s}(t)=(1-t)\gamma_{1}(s)+t\gamma_{2}(s)$ between the ray and the quasi-ray. Since we have a bound $|\eta_{s}'(t)|=o(s)$, the previous argument applies, thus $G(s)$ has the same limit along the stable quasi-ray $\gamma_{1}$ and along the associated stable ray $\gamma_{2}$.
\end{proof}

We now investigate how limits along rays in different intervals are related.
\begin{lemma}\label{lm:comparison_limits} Let $L_{++}, L_{+}, L_{-}, L_{--}$ be as in the previous lemma. Then there exist unipotent matrices $U_{+},U_{0}, U_{-}$, such that
\[
	L_{++}^{-1}L_{+}=SU_{+}S^{-1}, \ \ \ L_{+}^{-1}L_{-}=SU_{0}S^{-1} \ \ \ \text{and} \ \ \ L_{-}^{-1}L_{--}=SU_{-}S^{-1} \ .
\]
\end{lemma}
\begin{proof}We give a detailed proof for $L_{+}^{-1}L_{-}$, the other cases being analogous. \\
\indent Consider the rays $\gamma_{+}(s)=e^{i\frac{3\pi}{8}}s$ and $\gamma_{-}(s)=e^{i\frac{5\pi}{8}}s$. By the previous lemma $G_{+}(s)=\psi\psi_{0}^{-1}(\gamma_{+}(s))$ and $G_{-}(s)=\psi\psi_{0}^{-1}(\gamma_{-}(s))$ have respective limits $L_{+}$ and $L_{-}$. For any $s>0$, we can join $\gamma_{+}(s)$ and $\gamma_{-}(s)$ by a circular arc
\[
	\eta_{s}(t)=e^{i(t+\frac{3\pi}{8})}s \ \ \ \text{for}  \ \ \ t \in \left[0, \frac{\pi}{4}\right] \ .
\]
Let $g_{s}(t)=(\psi\psi_{0}(\eta_{s}(0))^{-1}\psi\psi_{0}(\eta_{s}(t))$, which satisfies
\[
	\begin{cases}
	g_{s}^{-1}(t)g_{s}'(t)= \Theta(\eta_{s}(t))\eta_{s}'(t) \\ 
	g_{s}(0)= \Id \\
	g_{s}(\pi/4)=G_{+}(s)^{-1}G_{-}(s) \ .
	\end{cases}
\]
Unlike the previous case, however, the coefficient $\Theta(\eta_{s}(t))$ is not exponentially small in $s$ throughout the interval. At $t=\frac{\pi}{8}$, conjugation by the diagonal matrix $D$ multiplies the $(4,2)$-entry of $R'$ by a factor $e^{4s}$, exactly matching the exponential decay rate of $R'$ and giving 
\[
	\Theta\left(\eta_{s}\left(\frac{\pi}{8}\right)\right)=O(\sqrt{s}) \ .
\]
However, this potential growth is seen only in the $(4,2)$-entry, because the other entries are scaled by smaller exponential factors. In fact, for $t \in \left[ 0, \frac{\pi}{4} \right]$, we have
\[
	\lambda_{42}=\exp\left(4\cos\left(t-\frac{\pi}{8}\right)\right)\leq \exp\left(4-\left(t-\frac{\pi}{8}\right)^{2}\right)  \ .
\]
We can thus separate the unbounded entry in $\Theta(\eta_{s}(t))$ and write
\[
	\Theta(\eta_{s}(t))=\Theta^{0}(\eta_{s}(t))+\mu_{s}(t)SE_{42}S^{-1}
\]
where $\Theta^{0}_{s}(t)=O(e^{-\alpha s})$ for some $\alpha>0$, $E_{42}$ is the elementary matrix, and
\[
	\mu_{s}(t)=O\left(|\eta_{s}'(t)|\lambda_{42}\frac{e^{-4s}}{\sqrt{s}}\right)=O(\sqrt{s}\exp(-(\pi/8-t)^{2}s)) \ .
\]
This upper bound is a Gaussian function on $t$, renormalized such that its integral over $\R$ is independent of $s$. Therefore, the function $\mu_{s}(t)$ is uniformly absolutely integrable over $t \in \left[0, \frac{\pi}{4}\right]$. We can apply \cite[Lemma B.2]{DWpolygons} and conclude that
\[
	\left\|g_{s}(\pi/4)-S\exp\left(E_{42}\int_{0}^{\frac{\pi}{4}}\mu_{s}(t)dt \right)S^{-1}\right\| \to 0 \ \ \ \text{as} \ \ \ s \to +\infty \ .
\]
Since $g_{s}(\pi/4) \to L_{+}^{-1}L_{-}$ as $s \to +\infty$, we obtain the desired unipotent difference.
\end{proof}

\subsubsection{Asymptotic behaviour of the minimal surface} We can now describe the asymptotic geometry of the minimal surface $S=f(\C)$ in the general case.

\begin{defi}
We say that two minimal surfaces $S_1$ and $S_2$ in an Riemannian manifold $Y$ are asymptotic if there is a domain Riemann surface $X$ and conformal harmonic parametrizations $u_i: X \to Y$ of $S_i$ so that $d_Y(u_1(x), u_2(x)) \to 0$ as $x$ leaves compacta in $X$.
\end{defi}

\begin{teo}\label{thm:asymptotic_flats} Let $q$ be a monic polynomial quartic differential of degree $n\geq 1$. Then the associated minimal surface $S$ is asymptotic to $2(n+4)$ maximal flats in $\Sp(4,\R)/\U(2)$.
\end{teo}

\begin{proof}We start by proving that in each standard half-plane $(U,w)$ given by Proposition \ref{prop:halfplanes} the surface $S$ is asymptotic to four maximal flats, one for each interval of stable directions.\\
\indent We give the detailed proof for the sector $J_{+}$, the other cases being analogous. We recall that $S$ is parameterized by the map 
\[
 f(w)=\overline{P^{t}(\psi(w)^{-1})^{t}}\diag(h_{1}(w), h_{2}^{-1}(w), h_{1}^{-1}(w), h_{2}(w))\psi(w)^{-1}P
\]
for some $P \in \Sp(4, \C)$. We compare $f(w)$ with the flat parameterized by
\[
	f_{0}(w)=\overline{(\psi_{0}(w)^{-1})^{t}}\psi_{0}(w)^{-1} \ :
\]
by Corollary \ref{cor:decay_uj} and Lemma \ref{lm:existence_limits}, the limit
\[
	\lim_{\substack{|w| \to +\infty \\ w \in J_{+}}}P^{-1}\psi(w)H^{-\frac{1}{2}}(w)\psi_{0}(w)^{-1}:=M_{+}
\]
exists for some $M_{+}$ in the standard copy of $\Sp(4,\R)$ fixed in Section \ref{sec:background}. We now claim that $S$ is asymptotic to the flat parameterized by 
\[
	f_{0}(w)=\overline{(M_{+}^{-1})^{t}(\psi_{0}(w)^{-1})^{t}}\psi_{0}(w)^{-1}M_{+}^{-1} \ .
\]
Namely, the map $g(w)=P^{-1}\psi(w)H(w)^{-\frac{1}{2}}\psi_{0}(w)^{-1}M_{+}^{-1}$ is an isometry sending $f_{0}(w)$ to $f(w)$ such that 
\[
	\lim_{\substack{|w| \to +\infty \\ w \in J_{+}}}g(w)=\Id
\]
and since, the action by isometries is linear, the same holds for its differential. \\
\indent This would give a total number of $4(n+4)$ maximal flats to which $S$ is asymptotic, but we can actually see that in two overlapping standard half-planes $U_{k}$ and $U_{k+1}$ two of the four flats coincide. In fact, by the above discussion, we can notice that the asymptotic flat depends only on the limit of $\psi(z)\psi_{0}^{-1}(z)$, which itself only depends on the half-plane $U_{k}$ in which $z$ eventually lies and on the specific sector $J_{++}^{k}, J_{+}^{k}, J_{-}^{k}, J_{--}^{k}$ in which $z$ is approaching infinity. Since in the intersection of the two charts $U_{k}$ and $U_{k+1}$ the natural coordinates differ by a multiplication by $i$ and by an additive constant, a quasi-ray of angle $\theta$ in the $w_{k}$-coordinate has direction $\theta+\frac{\pi}{2}$ in the $w_{k+1}$-coordinate. Therefore, the sector $J_{--}^{k+1}$ gets identified with $J_{+}^{k}$ and the sector $J_{-}^{k+1}$ gets identified with $J^{k}_{++}$ by the change of coordinates. Hence, the limits in those directions coincide and it follows that we only have a total number of $2(n+4)$ maximal flats.
\end{proof}

We can also describe precisely the combinatorics of the collection of flats at infinity.
\begin{prop} \label{prop: shared Weyl chambers} Two consecutive flats asymptotic to the minimal surface $S$ share four adjacent Weyl chambers at infinity.
\end{prop}
\begin{proof} Recall that a Weyl-Chamber at infinity is the stabilizer of a minimal parabolic subgroup $P \subset \Sp(4,\R)$ acting on the boundary at infinity of the symmetric space. In our case, a Weyl-Chamber at infinity can thus be described by a complete Lagrangian flag, that is by a collection of vector subspaces of $\R^{4}$
\[
		\mathcal{F}=\{ \{0\} \subset \ell \subset L \subset \ell^{\perp_{\omega}} \subset \R^{4}\}
\]
where $\ell$ is a line, $L$ is a Lagrangian plane and $\ell^{\perp_{\omega}}$ denotes the hyperplane orthogonal to $\ell$ with respect to the symplectic form $\omega$ on $\R^{4}$ that $\Sp(4,\R)$ preserves. Notice that the data of $\ell$ and $L$ already determine the flag uniquely. \\
We are going to show that the intersection of the two consecutive flats
\[
	F_{1}(w)=\overline{(M_{-}^{-1})^{t}(\psi_{0}(w)^{-1})^{t}}\psi_{0}(w)^{-1}M_{-}^{-1}
\]
and 
\[
	F_{2}(w)=\overline{(M_{+}^{-1})^{t}(\psi_{0}(w)^{-1})^{t}}\psi_{0}(w)^{-1}M_{+}^{-1}
\]
constructed in Theorem \ref{thm:asymptotic_flats} share four Weyl chambers at infinity; the proof for the other cases in analogous.
Recall that in Section \ref{sec:background}, we identified $\Sp(4,\R)$ as the subgroup of $\Sp(4,\C)$ fixed by the anti-linear involution $\lambda$, and we pointed out that the map
\begin{align*}
	\Sp(4,\C)/\SU(4) &\rightarrow \Sp(4,\C)/\SU(4) \\
		[g] &\mapsto \overline{(g^{-1})^{t}}g^{-1}
\end{align*}
induces an isometry between two models of the $\Sp(4,\R)$-symmetric space: as cosets and also as the space of $Q$-symmetric and $\Omega$-symplectic hermitian metrics on $\C^{4}$. Using such correspondence, the flats $F_{1}$ and $F_{2}$ can equivalently be described by the matrices
\[
	F_{1}(w)=M_{-}\psi_{0}(w) \ \ \ \ \text{and} \ \ \ \ F_{2}(w)=M_{+}\psi_{0}(w) \ .
\]
Moreover, from Lemma \ref{lm:existence_limits} and Lemma \ref{lm:comparison_limits} we know that $M_{-}=P^{-1}L_{+}SU_{0}S^{-1}$ and $M_{+}=P^{-1}L_{+}$, so, together with the fact that 
\[
	\psi_{0}(w)=S\diag(e^{2\Re(z)}, e^{-2\Im(z)}, e^{-2\Re(z)}, e^{2\Im(z)}) S^{-1} \ ,
\]
we deduce that the intersection at infinity of the flats $F_{1}$ and $F_{2}$ only depends on how the matrix $U_{0}$ acts on the Weyl chambers at infinity of the maximal flat of diagonal matrices. Since $U_{0}=\mathrm{Id}+\mu E_{42}$, for some $\mu \neq 0$, it is straightforward to check that $U_{0}$ does not preserve the Weyl chambers at infinity corresponding to the Lagrangian flags
\begin{align*}
	\{0\} &\subset \mathrm{Span}(e_{3}) \subset \mathrm{Span}(e_{3}, e_{2}) \subset \mathrm{Span}(e_{3}, e_{2}, e_{4}) \subset \R^{4} \\
	\{0\} &\subset \mathrm{Span}(e_{2}) \subset \mathrm{Span}(e_{3}, e_{2}) \subset \mathrm{Span}(e_{3}, e_{2}, e_{1}) \subset \R^{4} \\
	\{0\} &\subset \mathrm{Span}(e_{2}) \subset \mathrm{Span}(e_{1}, e_{2}) \subset \mathrm{Span}(e_{3}, e_{2}, e_{1}) \subset \R^{4} \\
	\{0\} &\subset \mathrm{Span}(e_{1}) \subset \mathrm{Span}(e_{1}, e_{2}) \subset \mathrm{Span}(e_{4}, e_{2}, e_{1}) \subset \R^{4} \ .
\end{align*}
Therefore, $F_{1}$ and $F_{2}$ share four adjacent Weyl chambers at infinity because $\Sp(4,\R)$ is a Lie group with root system of type $B_{2}$. 
\end{proof}

\subsubsection{The induced metric on the minimal surface}
Using the bounds (\ref{eq:bound_u1}) and (\ref{eq:bound_u2}) we prove that the harmonic map $f: \C \rightarrow \Sp(4,\R)/\U(2)$ is a quasi-isometric embedding if $\C$ is endowed with the flat metric with cone singularities $|q|^{\frac{1}{2}}$.

\begin{prop}\label{prop:quasi-iso} The induced metric on the minimal surface $S=f(\C)$ is quasi-isometric to $4|q|^{\frac{1}{2}}$, with quasi-isometric constant $1+ O(|q|^{-2})$ on the end of $S$. In particular, it is complete.
\end{prop}
\begin{proof} Recall that the induced metric $g_{f}$ on $S$ can be expressed in terms of the Higgs field $\varphi$
\[
	g_{f}=\trace(\varphi\varphi^{*H})=h_{1}^{2}|q|^{2}+2h_{2}h_{1}^{-1}+h_{2}^{-2} \ .
\]
Moreover, the sub-solution and super-solution found in Section \ref{sec:existence} provide the following upper- and lower-bounds for the metric $H$ satisfying Hitchin's equations
\[
	(|q|^{2}+C)^{-\frac{3}{8}} \leq h_{1} \leq |q|^{-\frac{3}{4}}
\]
\[
	(|q|^{2}+3C)^{-\frac{1}{8}} \leq h_{2} \leq |q|^{-\frac{1}{4}}
\]
for some positive constant $C$. We deduce that, as $|z| \to \infty$, we have
%\begin{align*}
%	g_{f}=h_{1}^{2}|q|^{2}+2h_{2}h_{1}^{-1}+h_{2}^{-2} &\geq (|q|^{2}+C)^{-\frac{3}{4}}|q|^{2}+2(|q|^{2}+3C)^{-\frac{1}{8}}|q|^{\frac{3}{4}}+|q|^{\frac{1}{2}} \\
%										&\geq |q|^{\frac{1}{2}}(1+O(|q|^{-2}))+2|q|^{\frac{1}{2}}(1+O(|q|^{-2}))+|q|^{\frac{1}{2}}\\
%										&\geq 4|q|^{\frac{1}{2}}(1+O(|q|^{-2})) \ .
%\end{align*}
\begin{align*}
	g_{f}=h_{1}^{2}|q|^{2}+2h_{2}h_{1}^{-1}+h_{2}^{-2} &\geq (|q|^{2}+C)^{-\frac{3}{4}}|q|^{2}+2(|q|^{2}+3C)^{-\frac{1}{8}}|q|^{\frac{3}{4}}+|q|^{\frac{1}{2}} \\
										&\geq |q|^{\frac{1}{2}}(1-\epsilon))+2|q|^{\frac{1}{2}}(1-3\epsilon))+|q|^{\frac{1}{2}}\\
										&\geq 4|q|^{\frac{1}{2}}(1+\epsilon)) \ .
\end{align*}
where the terms $\epsilon$ stand for terms that satisfy $\epsilon \asymp |q|^{-2}$.
As for the upper-bound, a similar argument shows that
\begin{align*}
	g_{f}=h_{1}^{2}|q|^{2}+2h_{2}h_{1}^{-1}+h_{2}^{-2} &\leq |q|^{\frac{1}{2}}+2|q|^{-\frac{1}{4}}(|q|^{2}+C)^{\frac{3}{8}}+(|q|^{2}+3C)^{\frac{1}{4}} \\
										&\leq |q|^{\frac{1}{2}}+2|q|^{\frac{1}{2}}(1+\epsilon))+|q|^{\frac{1}{2}}(1+3\epsilon))\\
										&\leq 4|q|^{\frac{1}{2}}(1+\epsilon)) \ .
\end{align*}
where again the terms $\epsilon$ represent terms of the comparability class $\epsilon \asymp |q|^{-2}$.
It thus follows that outside a compact set $K$ the induced metric is controlled by a multiple of $4|q|^{\frac{1}{2}}$, and decays to that quantity at a rate comparable to $|q|^{-2}$. The quasi-isometry between $4|q|^{\frac{1}{2}}$ and $g_{f}$ is then obtained by noticing that, since $K$ is compact, the metrics $4|q|^{\frac{1}{2}}$ and $g_{f}$ are trivially quasi-isometric on $K$; the claimed asymptotics of the quasi-isometric ratio is immediate. In particular, the induced metric $g_{f}$ on the minimal surface $S$ is complete because the metric $|q|^{\frac{1}{2}}$ is complete.
\end{proof}

\subsubsection{Estimates of the error term}\label{subsubsec:estimates}
This section is dedicated to the proof of Proposition \ref{prop:estimates_error}. Let us define the following auxiliary functions
\begin{align*}
	w_{1}&=-\frac{1}{2}\sum_{j \in \Z_{4}}i^{j}(\uu_{j}-\uu_{j+1})=\uu_{1}+\uu_{2}=-w_{3}\\
	w_{2}&=-\frac{1}{2}\sum_{j \in \Z_{4}}(-1)^{j}(\uu_{j}-\uu_{j+1})=2(\uu_{1}-\uu_{2}) \ ,
\end{align*}
where we are using the notation $\uu_{3}=-\uu_{2}$ and $\uu_{4}=-\uu_{1}$. We recall that in a natural coordinate $w$ on a half-plane, the functions $\uu_{1}$ and $\uu_{2}$ satisfy the following system of PDE
\begin{equation}\label{eq:system_halfplane}
	\begin{cases}
		\Delta \uu_{1}=e^{\uu_{1}-\uu_{2}}-e^{-2\uu_{1}} \\
		\Delta \uu_{2}=e^{2\uu_{2}}-e^{\uu_{1}-\uu_{2}} \ .
	\end{cases}
\end{equation}
Therefore, a simple computation shows that the error term $R_{kl}'$ (cf. \eqref{eqn: R1prime} can be written in terms of the derivatives of $w_{k-l}$, when $k>l$. In fact,
\begin{equation}\label{eq:relation_error_w}
	R'_{kl}=-\frac{i^{l-k}}{2(1-i^{l-k})}\partial w_{k-l}+\frac{i^{1-k}}{2(1-i^{l-k})(1-i^{k-l})}\Delta w_{k-l} \ . 
\end{equation}
Notice that the symmetries of the matrix $R'$ imply that the asymptotic behaviour of $R_{kl}'$ depends only on $k-l$ and it is sufficient to estimate the cases where $k-l=1,2$.\\

Proposition \ref{prop:estimates_error} will then be a consequence of the following estimate:
\begin{lemma}\label{lm:estimates_w} Let $r$ be the distance from the zeros of $q$. Then for $r \to +\infty$ we have
\[
	w_{1}=O\left(\frac{e^{-2\sqrt{2}r}}{\sqrt{r}}\right) \ \ \ \text{and} \ \ \ w_{2}=O\left(\frac{e^{-4r}}{\sqrt{r}}\right) \ .
\]
\end{lemma}

\begin{proof}[Proof of Proposition~\ref{prop:estimates_error}] Let us start with the terms $R_{kl}'$ for $k\neq l$. In view of Equation (\ref{eq:relation_error_w}), it is sufficient to estimate $\Delta w_{k-l}$ and $\partial w_{k-l}$. In the proof of Lemma \ref{lm:estimates_w}, we will see that $\Delta w_{k-l}=O(w_{k-l})$ for $r \to +\infty$, hence 
\[
	\Delta w_{k-l}=O\left(\frac{e^{-2|1-i^{k-l}|r}}{\sqrt{r}}\right)
\]
by Lemma \ref{lm:estimates_w}. The bound on $\partial w_{k-l}$ follows then from the Schauder estimates applied to 
$\Delta w_{k-l}$ in a ball of radius $r_0 \asymp 1$ about a point at distance comparable to $r$ from the zeroes of $q$.\\
\indent As for the terms on the diagonal, 
\[
	R_{kk}'=\frac{e^{-i\theta}(-i)^{k-1}}{4}(e^{-2\uu_{1}}+2e^{\uu_{1}-\uu_{2}}+e^{2\uu_{2}}-4) \ ,
\]
since $\uu_{1}$ and $\uu_{2}$ are infinitesimal as $r \to +\infty$, we deduce that $R_{kk}'=o(\uu_{j})$ for $j=1,2$. In particular, $R_{kk}'=o(w_{1})$ and the estimate follows.
\end{proof}

The proof of Lemma \ref{lm:estimates_w} relies on some results already proved in \cite{DWpolygons}.  
\begin{lemma}\label{lm:Bessel1} Let $g \in C^{0}(\R)\cap L^{1}(\R)$ be a positive function. Then for every positive constant $k$ there exists a function $h \in C^{\infty}(\h)\cap C^{0}(\overline{\h})$ such that
\[
	\begin{cases}
		\Delta h =kh \\
		h_{|_{\R}}=g 
	\end{cases} \ .
\]
Moreover, $h$ satisfies
\begin{align*}
	&0 \leq h \leq \sup(g) \\
	&h=O\left(\|g\|_{1}\frac{e^{-2\sqrt{k}y}}{\sqrt{y}}\right) \ \ \ \text{for} \ \ \  y=\Im(z) \to +\infty \ .
\end{align*}
\end{lemma}

\begin{lemma}\label{lm:Bessel2}Let $g \in C^{0}(\R)\cap L^{1}(\R)$ be a positive function satisfying $g \leq \frac{1}{4k'}$ for some $k'>0$. Then, there exists a function $v \in C^{\infty}(\h) \cap C^{0}(\overline{\h})$ such that
\begin{align*}
	&v_{|_{\R}} \geq g \\
	&\Delta v \leq kv-kk'v^{2} \\
	&v=O\left(\|g\|_{1}\frac{e^{-2\sqrt{k}y}}{\sqrt{y}}\right) \ \ \ \text{for} \ \ \ y=\Im(z) \to +\infty \ .
\end{align*}
\end{lemma}
\begin{proof}Consider an arbitrary solution $h$ of the equation $\Delta h=kh$. The function $v=h-k'h^{2}$ satisfies
\begin{align*}
	\Delta v -kv+k'kv^{2}&=\Delta h-k'\Delta h^{2}-kh+kk'h^{2}+kk'(h-k'h^{2})^{2}\\
			&=-k'(2h\Delta h+2|\nabla h|^{2})+2kk'h^{2}+kk'^{3}h^{4}-2kk'^{2}h^{3}\\
			&=-2k'|\nabla h|^{2}+kk'^{3}h^{4}-2kk'^{2}h^{3}\\
			&\leq kk'^{3}h^{4}-2kk'^{2}h^{3}\leq 0 \ 
\end{align*}
provided $h\leq \frac{2}{k'}$. This condition is satisfied if we take as $h$ the solution provided by Lemma \ref{lm:Bessel1} with boundary value $2g$, as
\[
	0\leq h \leq 2\sup(g) \leq \frac{1}{2k'}\leq \frac{2}{k'} \  .
\]
Therefore, $v=h-k'h^{2}$ satisfies 
\begin{align*}
	&0 < v < h \\
	&\Delta v \leq kv-kk'v^{2} \ . 
\end{align*}
In particular, by Lemma \ref{lm:Bessel1}, we deduce that $v=O\left(\|g\|_{1}\frac{e^{-2\sqrt{k}}y}{\sqrt{y}}\right)$ for $y \to +\infty$. Moreover, the condition $g \leq \frac{1}{4k'}$ implies that $v_{|_{\R}} \geq g$.	
\end{proof}

\begin{proof}[Proof of Lemma \ref{lm:estimates_w}] Let us start with $w_{1}$.\\
\indent By Corollary \ref{cor:decay_uj}, there exists a compact subset $K$, outside of which 
\[
	\uu_{1}=u_{1}-\frac{3}{8}\log(|q|^{2}) \leq \frac{1}{16} \ \ \ \text{and} \ \ \ \uu_{2}=u_{2}-\frac{1}{8}\log(|q|^{2}) \leq \frac{1}{16} \ .
\]
By Proposition \ref{prop:good_halfplanes}, every point $p$ sufficiently far from $K$ lies in a half-plane $(U,w)$ with $U\cap K=\emptyset$ and $\Im(w(p))\geq r(p)-C$ for some $C>0$ independent from $p$. We identify $(U,w)$ with $\h^{2}$ and we work in the $w$-coordinates. In particular, in these coordinates the functions $\uu_{j}$ satisfy the system of PDE (\ref{eq:system_halfplane}). Moreover, again by Proposition \ref{prop:good_halfplanes}, the function $w_{1}=\uu_{1}+\uu_{2}$ is positive and the restriction of $w_{1}$ to the real axis is integrable. Moreover, its $L^{1}$-norm can be bounded by some constant that depends only on the coefficients of $q$. 
We can thus apply Lemma \ref{lm:Bessel2} with boundary condition $g=w_{1}$ and $k=k'=2$, thus getting a function $v$ which satisfies $v=O\left(\frac{e^{-2\sqrt{2}r}}{\sqrt{r}}\right)$. It is now sufficient to prove that $w_{1} \leq v$, or, equivalently, that $\eta_{1}=w_{1}-v$ is always non-positive. Notice that $\eta_{1} \in C^{\infty}(\h^{2})\cap C^{0}(\overline{\h}^{2})$, and, since $\uu_{1}$ and $\uu_{2}$ are positive, the following inequality holds
\[
	\Delta w_{1}=e^{2\uu_{2}}-e^{-2\uu_{1}}\geq 2w_{1}+2(\uu_{2}^{2}-\uu_{1}^{2})\geq 2w_{1}-2w_{1}^{2} \ .
\]
Suppose by contradiction that $\eta_{1}$ is positive in some point, so that the set $Q=\eta_{1}^{-1}([\epsilon,+\infty)) \neq \emptyset$. Since $\eta_{1} \leq 0$ on $\partial \h^{2}$ and $\eta_{1}$ is infinitesimal for $|z| \to +\infty$, the set $Q$ is compact and $\eta_{1}$ has a maximum at some $p \in Q$. In this point, we have
\begin{align*}
	0 &\geq \Delta \eta_{1}(p)=\Delta w_{1}(p)-\Delta v(p) \\ 
	  & \geq 2w_{1}(p)-2w_{1}(p)^{2}-2v(p)+4v(p)^{2}\\
	  & \geq 2\eta_{1}(p)-2\eta_{1}(p)(w_{1}(p)+v(p))\\ 
	  &\geq 2\eta_{1}(p)-\eta_{1}(p)=\eta_{1}(p)
\end{align*}
and this contradicts the fact that $p\in Q$. \\
\indent We now use the estimate for $w_{1}$ to deduce the asymptotic behaviour of $w_{2}$. Since we do not know if $w_{2}$ is positive, we work with the function $w_{2}'=|w_{2}| \in C^{0}(\C) \cap H^{1}_{loc}(\C)$. Let $K$ be a compact set containing the roots of the quartic differential $q$ such that $w_{2}'\leq 1$ on $\C\setminus K$: this is possible because $w_{2}'\leq w_{1}$, and we proved that $w_{1}$ is infinitesimal. By Proposition \ref{prop:halfplanes}, we can cover the complement of a neighbourhood of $K$ with a finite number of standard half-planes $(U_{i}, \zeta_{i})$. In each of these, the function $w_{2}$ satisfies the following PDE
\[
	\Delta w_{2}=2(2e^{\tilde{u}_{1}-\tilde{u}_{2}}-e^{-2\tilde{u}_{1}}-e^{2\tilde{u}_{2}}) \ .
\]
Now, where $w_{2}\geq 0$, we have that
\begin{align*}
	e^{\tilde{u}_{1}-\tilde{u}_{2}} &\geq 1+\tilde{u}_{1}-\tilde{u}_{2}  \\
	e^{-2\tilde{u}_{1}}&\leq 1-2\tilde{u}_{1}+2\tilde{u}_{1}^{2}\\
	e^{2\tilde{u}_{2}}& \leq 1+2\tilde{u}_{2}+3\tilde{u}_{2}^{2}
\end{align*}
where the last inequality is true for $|\zeta_{i}|$ large, since the functions $\tilde{u}_{1}$ and $\tilde{u}_{2}$ are infinitesimal. Therefore, if $w_{2}\geq 0$, we have
\[
	\Delta w_{2} \geq w_{2}-3(\tilde{u}_{1}^{2}+\tilde{u}_{2}^{2}) \geq 4w_{2}-4w_{1}^{2}
\]
for $|\zeta_{i}|$ large enough.
Similarly, when $w_{2}<0$ we have
\begin{align*}
	e^{-2\tilde{u}_{1}}&\geq 1-2\tilde{u}_{1} \ \ \ \text{for $|\zeta_{i}|$ large enough} \\
	e^{2\tilde{u}_{2}}&\geq 1+2\tilde{u}_{2} \\
	e^{\tilde{u}_{1}-\tilde{u}_{2}}&\leq 1+\tilde{u}_{1}-\tilde{u}_{2}+\frac{(\tilde{u}_{1}-\tilde{u}_{2})^{2}}{2} \ \ \ \text{for $|\zeta_{i}|$ large enough} \ ,
\end{align*}
thus
\[
	\Delta(-w_{2})\geq 4(-w_{2})-(-w_{2})^{2} \geq 4(-w_{2})-4w_{1}^{2} \ .
\]
We deduce that, in each standard half-plane $(U_{i}, \zeta_{i})$, outside a compact set, the function $w_{2}'=|w_{2}|$ satisfies
\begin{equation*}\label{eq:ineq1}
	\Delta w_{2}'\geq 4 w_{2}'-4w_{1}^{2} \ .
\end{equation*}
Moreover, from the estimates for $w_{1}$, we know that, for $\zeta_{i}$ sufficiently large, we have
\begin{equation*}\label{eq:ineq2}
	w_{1}^{2} \leq C_{i}\frac{e^{-4\sqrt{2}|\zeta_{i}|}}{|\zeta_{i}|} \ .
\end{equation*}
Let $v$ be the solution of the boundary value problem
\[
	\begin{cases}
	\Delta v=4v-4w_{1}^{2} \\
	v_{|_{\R}}=w_{2}'
	\end{cases} \ .
\]
By a similar argument as that used for $w_{1}$, we have that $w_{2}'\leq v$ and the estimate for $w_{2}'$ is then a consequence of the following lemma.
\end{proof}

\begin{lemma}\label{lm:Bessel3}Let $g \in C^{0}(\R)\cap L^{1}(\R)$ be a positive function such that $g\leq 1$ and let $g'\in C^{\infty}(\h)$ be such that $g'=O(e^{-4\sqrt{2}r}/r)$ when $r$ goes to infinity. Then, the solution $v \in C^{\infty}(\h) \cap C^{0}(\overline{\h})$ to the boundary value problem
\[
	\begin{cases}
		\Delta v-4v=-g' \\
		v_{|_{\R}}=g 
	\end{cases} \ 
\]
satisfies $v=O\left(\|g\|_{1}\frac{e^{-4r}}{\sqrt{r}}\right)$ for $r \to +\infty$.
\end{lemma}
\begin{proof} It is sufficient to prove that there exists a constant $C>0$ and $1<\alpha<\sqrt{2}$ so that, $\eta=(C+1)h-Ch^{\alpha}$ is a super-solution, where $h$ is the function provided by Lemma \ref{lm:Bessel1} with $k=4$. First notice that, by our assumption on $g$, we have 
\[
	\eta_{|_{\R}}=(C+1)g-Cg^{\alpha} \geq (C+1)g-Cg =g
\]
Moreover,
\begin{align*}
	\Delta \eta - 4\eta +g' &=4(C+1)h-C\Delta(h^{\alpha})-4(C+1)h+4Ch^{\alpha}+g' \\
					&=-C\alpha(\alpha-1)|\nabla h|^{2}h^{\alpha-2}-C\alpha h^{\alpha-1}\Delta h +4Ch^{\alpha}+g' \\
					& \leq -4\alpha Ch^{\alpha}+4Ch^{\alpha}+g' = 4Ch^{\alpha}(1-\alpha)+g'
\end{align*}
is negative for $r$ sufficiently large by our assumption on the asymptotic decay on $g'$ and by Lemma \ref{lm:Bessel1}. We can thus choose $C$ sufficiently large so that it is negative everywhere, and $\eta$ is a super-solution as claimed. 
\end{proof}

\section{Immersions into the Grassmannian of symplectic planes}
In this section, we begin the proof of Theorem~\ref{thmB}.  We need to relate the solutions of the Hitchin equations \eqref{eq:system} to boundary values of maximal surfaces in $\h^{2,2}$. We accomplish this association via an intermediary identification (Proposition~\ref{prop:convex immersion} and subsequent remarks) of solutions of the Hitchin equations to convex embeddings of the plane $\C$ into a Grassmannian $\Gr_{2}(\sE_{\R})$. We then relate (Proposition~\ref{prop:rel_max_surface}) such a convex embedding to a maximal surface in $\h^{2,2}$.

\begin{notation}From this point on, we will denote by $\Sp(4,\R)$ the group of real matrices preserving the symplectic form $\Omega$. Recall that this differs by conjugation by $A \in \SU(4)$ from the group that we have used so far (see Remark \ref{rmk:conjugation}).
\end{notation}

\subsection{Surfaces in the Grassmannian of symplectic planes}\label{subsec:convex_surf}
Let us start with the data of an $\Sp(4,\R)$-Higgs bundle $(\sE, \varphi)$ over $\C$ with polynomial quartic differential $q=q(z)dz^{4}$ (see Section \ref{sec:background} for the definition). We denote by $W$ the (trivial) circle bundle over $\C$ and with $\pi:W \rightarrow \C$ the canonical projection. We define a global section of $\pi^{*}\sE \rightarrow W$ by
\begin{equation} \label{eqn: defn of section s}
s(z, \theta)=\pi^{*}\begin{pmatrix} 0 \\ h_{2}^{\frac{1}{2}}e^{i\theta} \\ 0 \\ 0 
		     \end{pmatrix}+
		     \pi^{*}\tau\begin{pmatrix} 0 \\ h_{2}^{\frac{1}{2}}e^{i\theta} \\ 0 \\ 0
		     \end{pmatrix} \ ,
\end{equation}
where $\tau:\sE \rightarrow \overline{\sE}$ is the real involution preserved by the flat connection $\nabla$ and the coordinates are expressed with respect to the frame $\{F(z)\}_{z \in \C}$ constructed in section~\ref{subsec:construction_min_surface}. 
Recalling that $\tau(v)=H^{-1}Q\overline{v}$, we obtain 
\[
	s(z, \theta)=\pi^{*}\begin{pmatrix} 0 \\ h_{2}^{\frac{1}{2}}e^{i\theta} \\ 0 \\ h_{2}^{-\frac{1}{2}} e^{-i\theta} \end{pmatrix} \ .
\]
Define $\sE_{\R}$ to be the fixed point set of $\tau$ in $\sE$.

Notice that by equation \eqref{eqn: defn of section s}, we have that the image of $s$ lies in the real sub-bundle $\pi^{*}\sE_{\R}=\Fix(\pi^{*}\tau)$, which is preserved by $\hat{\nabla}=\pi^*\nabla$. Then we compute
\[
	\hat{\nabla}_{\frac{\partial}{\partial \theta}}\hat{\nabla}_{\frac{\partial}{\partial \theta}}s(z, \theta)=-s(z, \theta) \ .
\]  
to conclude that the fibres of $W$ are developed onto real lines.

Therefore, if we denote by $\Gr_{2}(\sE_{\R})$ the bundle over $\C$, whose fibre over each point $z \in \C$ is the Grassmannian of $2$-planes in $(\sE_{\R})_{z}$, the map
\[
	f(z)=s(z, \theta) \wedge \hat{\nabla}_{\frac{\partial}{\partial \theta}} s(z, \theta)
\]
is a well-defined section of $\Gr_{2}(\sE_{\R})$, where we are identifying fiber-wise the Grassmannian of $2$-planes with the space of decomposable tensors in $\Lambda^{2}\sE_{\R}$ via the Pl\"ucker embedding. If we introduce the following $H$-unitary, real, global frame of $\sE_{\R}$
\[
	u_{1}(z)=\frac{1}{\sqrt{2}}\begin{pmatrix} 0 \\ h_{2}^{\frac{1}{2}} \\ 0 \\ h_{2}^{-\frac{1}{2}} \end{pmatrix} \ \ \ 
	u_{2}(z)=\frac{1}{\sqrt{2}}\begin{pmatrix} h_{1}^{-\frac{1}{2}} \\ 0 \\ h_{1}^{\frac{1}{2}} \\ 0 \end{pmatrix} 
\]
\[
	u_{3}(z)=\frac{1}{\sqrt{2}}\begin{pmatrix} 0 \\ ih_{2}^{\frac{1}{2}} \\ 0 \\ -ih_{2}^{-\frac{1}{2}} \end{pmatrix} \ \ \ 
	u_{4}(z)=\frac{1}{\sqrt{2}}\begin{pmatrix} ih_{1}^{-\frac{1}{2}} \\ 0 \\ -ih_{1}^{\frac{1}{2}} \\ 0 \end{pmatrix} \ ,
\]
it is straightforward to verify that $f(z)=u_{1}(z)\wedge u_{3}(z)$, hence $f$ selects in each fibre of $\sE_{\R}$ the plane generated by $u_{1}(z)$ and $u_{3}(z)$. Now, recall the definition of $\Omega$ in \eqref{eqn: definition of Omega}: the $\nabla$-parallel symplectic form $\Omega$ induces a $\nabla$-parallel symplectic form on $\sE_{\R}$, which, in the above frame, is expressed by the matrix
\[
	\omega_{\R}=\begin{pmatrix} 0 & 0 & -1 & 0 \\
				    0 & 0 & 0 & -1 \\
				    1 & 0 & 0 & 0 \\
				    0 & 1 & 0 & 0
		     \end{pmatrix}\ ;
\]
thus the image of $f$ entirely lies in the space of symplectic $2$-planes of $\sE_{\R}$.  \\
\\
\indent Let us now underline some properties of this map that will allow to recover the minimal surface in the $\Sp(4,\R)$-symmetric space. We learned the following from Fran\c{c}ois Labourie.

\begin{defi} Let $\Sigma$ be a Riemann surface. An immersion $f:\Sigma \rightarrow \Gr_{2}(\R^{4})$ is convex if for any point $p \in \Sigma$ and any tangent vector $X \in T_{p}\Sigma$ the map
\[
	B(X)=d_{p}f(X) \in T_{f(p)}\Gr_{2}(\R^{4})
\]
is invertible. 
\end{defi}

We develop next an interpretation of $T\Gr_{2}(\R^{4})$. We recall that given a plane $L \in \Gr_{2}(\R^{4})$, the tangent space $T_{L}\Gr_{2}(\R^{4})$ is identified with the vector space of linear homomorphisms $\Hom(L, \R^{4}/L)$ in the following way. Let $L=\Span(v,w)$ and let $\gamma:[0,1] \rightarrow \Gr_{2}(\R^{4})$ be a smooth path such that $\gamma(0)=L$. For every $t \in [0,1]$ we choose smoothly two vectors $v(t)$ and $w(t)$ so that $\gamma(t)=\Span(v(t), w(t))$. Using the Pl\"ucker embedding, we can thus write $\gamma(t)=v(t)\wedge w(t)$. Now, the tangent vector at $t=0$ is given by
\[
	\gamma'(0)=\frac{d}{dt}_{|_{t=0}}v(t)\wedge w(t)=v'(0)\wedge w(0)+v(0)\wedge w'(0) \ .
\]
The variation of the plane $L$ is expressed only by the components of $v'(0)$ and $w'(0)$ that do not lie in the plane $L$. Thus the tangent vector $\gamma'(0)$ is completely determined by the linear map
\[
	B(X):L \rightarrow \R^{4}/L 
\] 
where we construe $X\in T_{p}\Sigma$ as tangent to a curve $\gamma$ (as described above) with $f \circ \gamma \subset \Gr_{2}(\R^{4})$, and suppressing some of the notation, we take $B(\dot{\gamma})v=v'(0)$ (mod $L$) and $B(\dot{\gamma})w=w'(0)$ (mod $L$). 

\begin{prop} \label{prop:convex immersion} The immersion $f:\C \rightarrow \Gr_{2}(\R^{4})$ defined above is convex.
\end{prop}
\begin{proof}Recall that the flat connection $\nabla^{Gr}$ on $\Gr_{2}(\sE_{\R})$ may be defined in terms of the connection $\nabla=H^{-1}\partial H+ \varphi+\varphi^{*H}$ on $\sE$: in particular $\nabla^{Gr}(v \wedge w)$ is defined in terms of $\nabla v$ and $\nabla w$. To that end, suppose we have a basis $\{u_{1}(z), u_{3}(z)\}$ of $f(z)$ and $\{u_{2}(z), u_{4}(z)\}$ of $\R^{4}/f(z)$: to compute $\nabla^{Gr} u_i \wedge u_j$ for $i \neq j$, it suffices to compute $\nabla u_k$ for each $k$. We thus compute
\begin{align}\label{eq:derivatives 1}
\begin{split} 
	\nabla_{\frac{\partial}{\partial x}}u_{1}(z)&=h_{2}^{-1}u_{1}(z)+\frac{1}{2}h_{2}^{-1}\partial_{y}h_{2}u_{3}(z)+h_{1}^{-\frac{1}{2}}h_{2}^{\frac{1}{2}}u_{2}(z)\\
	\nabla_{\frac{\partial}{\partial x}}u_{3}(z)&=-\frac{1}{2}h_{2}^{-1}\partial_{y}h_{2}u_{1}(z)-h_{2}^{-1}u_{3}(z)-h_{1}^{-\frac{1}{2}}h_{2}^{\frac{1}{2}}u_{4}(z)\\
	\nabla_{\frac{\partial}{\partial y}}u_{1}(z)&=(h_{2}^{-1}-\frac{1}{2}h_{2}^{-1}\partial_{x}h_{2})u_{3}(z)-h_{1}^{-\frac{1}{2}}h_{2}^{\frac{1}{2}}u_{4}(z)\\
	\nabla_{\frac{\partial}{\partial y}}u_{3}(z)&=(h_{2}^{-1}+\frac{1}{2}h_{2}^{-1}\partial_{x}h_{2})u_{1}(z)-h_{1}^{-\frac{1}{2}}h_{2}^{\frac{1}{2}}u_{2}(z)
\end{split} 
\end{align}
and deduce that the homomorphisms $B(\frac{\partial}{\partial x})$ and $B(\frac{\partial}{\partial y})$ are represented by the matrices 
\[
	B\left(\frac{\partial}{\partial x}\right)=h_{1}^{-\frac{1}{2}}h_{2}^{\frac{1}{2}}\begin{pmatrix} 1 & 0 \\ 0 & -1 
	\end{pmatrix} \ \ \ \ 
	B\left(\frac{\partial}{\partial y}\right)=h_{1}^{-\frac{1}{2}}h_{2}^{\frac{1}{2}}\begin{pmatrix} 0 & -1 \\ -1 & 0 
	\end{pmatrix}
\]
with respect to the basis $\{u_{1}(z), u_{3}(z)\}$ of $f(z)$ and $\{u_{2}(z), u_{4}(z)\}$ of $\R^{4}/f(z)$. Since they are both invertible, the result follows.
\end{proof}

\begin{lemma}Let $f:\C \rightarrow \Gr_{2}(\R^{4})$ be the convex immersion constructed above. Then there exist complex structures $J_{1}$ on $f(z)$ and $J_{2}$ on $\R^{4}/f(z)$ such that, for every $z \in \C$, the map $B:T_{z}\C \rightarrow \Hom_{\C}(f(z), \R^{4}/f(z))$ is an isomorphism that intertwines with $J_1$ and $J_2$.
\end{lemma}
\begin{proof}Let us choose the basis $\{u_{1}(z), u_{3}(z)\}$ for $f(z)$ and let us identify $\R^{4}/f(z)$ with the plane generated by $\{u_{2}(z), u_{4}(z)\}$. We define the complex structures on these planes as follows
\[
	J_{1}(z)=\begin{pmatrix} 0 & -1 \\ 1 & 0 \end{pmatrix} \ \ \ \text{and}  \ \ \
	J_{2}(z)=\begin{pmatrix} 0 & 1 \\ -1 & 0 \end{pmatrix} \ .
\]
Using the explicit formulas for $B$ found in the previous proposition, it is easy to check that for every $X \in T_{z}\C$ the map $B(X)$ is $\C$-linear, i.e. $B(X)J_{1}(z)=J_{2}(z)B(X)$. Moreover, if $J$ denotes the standard complex structure on $\C$, we notice that
\[
	B\left(J\frac{\partial}{\partial x}\right)=J_{2}(z)B\left(\frac{\partial}{\partial x}\right) \ \ \ \text{and} \ \ \
	B\left(J\frac{\partial}{\partial y}\right)=J_{2}(z)B\left(\frac{\partial}{\partial y}\right) \ ,
\]
which implies that the linear map
\begin{align*}
	B: T_{z}\C &\rightarrow \Hom_{\C}(f(z), \R^{4}/f(z))\\
	     X &\rightarrow B(X) 
\end{align*}
is well-defined and $\C$-linear for every $z \in \C$. Since it is not trivial, it is an isomorphism.
\end{proof}

\begin{remark} We note that a convex embedding of $\C$ into $\Gr_{2}(\R^{4})$ induces a minimal immersion of $\C$ into the symmetric space $\Sp(4,\R)/\U(2)$. In particular, since $\R^{4}=f(z)\oplus \R^{4}/f(z)$ for every $z \in \C$, the complex structures provided by the previous lemma enables us to define a complex structure $J$ on $\R^{4}$, depending on the point $z \in \C$. Precisely, $J(z)=J_{1}(z)\oplus -J_{2}(z)$. We can also define a family of metrics on $\R^{4}$ depending on the point $z \in \C$ by 
\[
	H_{\omega_{\R},J}(z)(v,w)=\omega_{\R}(v, J(z)w) \ .
\]
It follows that the $H$-unitary frame $\{u_{1}(z), u_{2}(z), u_{3}(z), u_{4}(z)\}$ is $H_{\omega_{\R},J}(z)$-unitary at every point, thus $H_{\omega_{\R},J}(z)$ coincides with the harmonic metric $H$, and the $\nabla$-parallel transport of $H_{\omega_{\R},J}$, or equivalently of the complex structure $J$, produces the minimal surface in $\Sp(4,\R)/\U(2)$ associated to the given $\Sp(4,\R)$-Higgs bundle data.
\end{remark}

\subsection{Explicit parameterization for $q=dz^{4}$} \label{subsec:Titeica}
In the special case when the quartic differential is constant we can write explicitly a parameterization of the surface constructed previously. In analogy with affine spheres (\cite{DWpolygons}), we will refer to it as the standard flat maximal surface $T_0$.\\
\\
\indent To this aim, it is convenient to work in the global frame, say $\{w_{i}(z)\}_{i=1}^{4}$, for $\sE$ in which the matrix connection of the flat connection $\nabla_{0}$ is diagonal (see Section \ref{subsec:constant_case}). The change of frame is expressed by the constant unitary matrix $S$ introduced in Section \ref{subsec:constant_case}. We obtain, for the section $s$ in \eqref{eqn: defn of section s},
\begin{align*}
	s(z,\theta)&=\frac{1}{2}e^{i\theta}(w_{1}(z)-w_{2}(z)+w_{3}(z)-w_{4}(z))\\ &+\frac{1}{2}e^{i\theta}\tau(w_{1}(z)-w_{2}(z)+w_{3}(z)-w_{4}(z))
	=\frac{1}{2}e^{i\theta}\begin{pmatrix} 1 \\ -1 \\ 1 \\ -1 \end{pmatrix}+\frac{1}{2}e^{i\theta}\tau\begin{pmatrix} 1 \\ -1 \\ 1 \\ -1 \end{pmatrix}
\end{align*}
where the coordinates are now expressed with respect to the frame $\{w_{i}(z)\}_{i=1, \dots, 4}$. In this frame the real involution $\tau = H^{-1}Q$ is given by
\[
	\tau\begin{pmatrix} z_{1} \\ z_{2} \\ z_{3} \\ z_{4} \end{pmatrix}=
	    \begin{pmatrix} \overline{z}_{1} \\ -i\overline{z}_{2} \\ -\overline{z}_{3} \\ i\overline{z}_{4} \end{pmatrix} \ ,
\]
thus the frame $\{e_{1}(z)=w_{1}(z), e_{2}(z)=\frac{(1-i)}{\sqrt{2}}w_{2}(z), e_{3}(z)=iw_{3}(z), e_{4}(z)=\frac{(1+i)}{\sqrt{2}}w_{4}(z)\}$ is real and still diagonalizes the flat connection $\nabla_{0}$. Since we know that $s(z, \theta)$ will take value in $\sE_{\R}$, we will use coordinates with respect to this frame from now on.\\  
Moreover, the restriction of the $\nabla_{0}$-parallel symplectic form $\Omega$ induces a $\nabla_{0}$-parallel symplectic form on $\sE_{\R}$, that is given by the matrix
\[
	\omega_{\R}=\begin{pmatrix} 0 & 0 & -1 & 0 \\
				    0 & 0 &  0 & -1 \\
				    1 & 0 & 0 & 0 \\
				    0 & 1 & 0 & 0 
		    \end{pmatrix} \ .
\]
We can thus identify $\R^{4}$ endowed with the above symplectic form, with the fibre of $\sE_{\R}$ over a base point $0 \in \C$. By $\nabla_{0}$-parallel transporting $s$ at $(z,\theta)$ to this fibre over $0 \in \C$, we can first parameterize the image of $s$ (as images in the fixed $\R^4$) as
\begin{align*}
	s(z, \theta)&=\frac{1}{2}e^{i\theta}D(z)(w_{1}(z)-w_{2}(z)+w_{3}(z)-w_{4}(z))\\ 
	&+\frac{1}{2}e^{i\theta}\tau(D(z)(w_{1}(z)-w_{2}(z)+w_{3}(z)-w_{4}(z)))\\
	&\scalebox{0.8}{$=\frac{1}{2}(2\cos(\theta)e^{2\Re(z)}, \sqrt{2}(\sin(\theta)-\cos(\theta))e^{-2\Im(z)}, 2\sin(\theta)e^{-2\Re(z)}, \sqrt{2}(-\sin(\theta)-\cos(\theta))e^{2 \Im(z)})$}
\end{align*}
and, consequently, the standard flat maximal surface $T_0=f(\C)$ in the Grassmannian of symplectic planes (identified with a submanifold of $\Pp(\Lambda^{2}\R^{4})$ via the Pl\"ucker embedding)
\begin{align*}
	f(z)&=s(z,\theta)\wedge \nabla_{\frac{\partial}{\partial \theta}}s(z, \theta)\\
	&=e_{1} \wedge e_{3}+e_{2}\wedge e_{4} +\frac{1}{\sqrt{2}} e^{2\Re(z)-2\Im(z)}e_{1} \wedge e_{2}-\frac{1}{\sqrt{2}}e^{2\Re(z)+2\Im(z)}e_{1} \wedge e_{4}\\
	&+\frac{1}{\sqrt{2}}e^{-2\Re(z)+2\Im(z)}e_{3}\wedge e_{4}-\frac{1}{\sqrt{2}}e^{-2\Re(z)-2\Im(z)}e_{2}\wedge e_{3} \ ,
\end{align*}
where $e_{i}=e_{i}(0)$ for $i=1,\dots, 4$. 
\noindent We remark that $T_0$ coincides with the orbit of the point $f(0) \in \Gr_{2}(\R^{4})$ under the action of the diagonal matrices $D(z)=\diag(e^{2\Re(z)}, e^{-\Im(z)}, e^{-2\Re(z)}, e^{2\Im(z)})$ for $z \in \C$.\\
\\
\indent Moreover, by looking at limits along (quasi)-rays, we can describe the boundary at infinity of $T_0$ as a quadrilateral in the space of Lagrangians of $\R^{4}$, as Table \ref{table:1} shows.

\medskip
\begin{table}[!htb]
\begin{center}
\begin{tabular}{l l l}
\hline
\textbf{Type of path $\gamma$} & \textbf{Direction $\theta$} & \textbf{Projective limit $p_\gamma$ of $f_{0}(\gamma)$}\\
\hline
Quasi-ray& $\theta \in (0, \tfrac{\pi}{2})$ & $p_\gamma = [e_{1} \wedge e_{4}]$\\
Ray (of height $y$) & $\theta = \tfrac{\pi}{2}$ & $p_\gamma=[-e_{1}\wedge e_{4}+e^{-4y}e_{3} \wedge e_{4}] $ \\ & & $(p_\gamma \to [e_{1} \wedge e_{4}]$ as $y \to \infty)$\\
Quasi-ray& $ \theta \in (\tfrac{\pi}{2}, \pi)$ & $p_\gamma = [e_{3} \wedge e_{4}]$\\
Ray (of height $iy$) & $\theta = \pi$ & $p_\gamma=[e_{3}\wedge e_{4}-e^{-4y}e_{2} \wedge e_{3}]$ \\& & $(p_\gamma \to [e_{3}\wedge e_{4}]$ as $y \to \infty)$\\
Quasi-ray& $\theta\in (\pi, \tfrac{3\pi}{2})$  & $p_\gamma = [e_{2} \wedge e_{3}]$\\
Ray (of height $y$) & $\theta = \tfrac{3\pi}{2}$ & $p_\gamma=[e_{1}\wedge e_{2}-e^{-4y}e_{2}\wedge e_{3}]$ \\& & $(p_\gamma \to [e_{1}\wedge e_{2}]$ as $y \to \infty)$\\
Quasi-ray& $\theta\in (\tfrac{3\pi}{2}, 2\pi)$  & $p_\gamma = [e_{1} \wedge e_{2}]$\\
Ray (of height $iy$) & $\theta = 0$ & $p_\gamma=[e^{-4y}e_{1}\wedge e_{2}-e_{1}\wedge e_{4}]$ \\& & $(p_\gamma \to [e_{1}\wedge e_{4}]$ as $y \to \infty)$\\
\hline\\
\end{tabular}
\end{center}
\caption{Limits of the standard flat maximal surface along rays}\label{table:1}
\end{table}

\subsection{Relation with the maximal surface in $\h^{2,2}$}\label{subsec:maxsurface}
Exploiting the low-dimensional isogeny $\Pp\Sp(4,\R) \cong \SO_{0}(2,3)$, we can relate the convex surface $\Sigma$ in the Grassmannian of symplectic planes in $\R^{4}$ with a unique maximal surface in $\h^{2,2}$. We will see, in particular, that under the identification between the boundary at infinity of $\h^{2,2}$ (i.e. the Einstein Universe $\Ein^{1,2}$) and the space of Lagrangians of $\R^{4}$, the two surfaces share the same boundary at infinity.\\
\\
\indent Let us first recall how the low-dimensional isomorphism $\Pp\Sp(4,\R) \cong \SO_{0}(2,3)$ is accomplished. We denote by $\{e_{i}\}_{i=1,\dots, 4}$ the canonical basis of $\R^{4}$ and we consider the symplectic form $\omega=dx_{1}\wedge dx_{3}+dx_{2}\wedge dx_{4}$. Let $V=\Lambda^{2}\R^{4}$ be the vector space of skew-symmetric $2$-tensors on $\R^{4}$. A standard basis for $V$ is given by $\{e_{i} \wedge e_{j}\}_{1\leq i<j\leq 4}$. The symplectic form $\omega$ induces an inner product on $V$ via the relation
\begin{equation} \label{eqn:bracket on V}
	-2\langle \phi, \psi \rangle e_{1}\wedge e_{2} \wedge e_{3} \wedge e_{4}= \phi \wedge \psi \ .
	\end{equation}
It turns out that $\langle \cdot, \cdot \rangle$ is non-degenerate and has signature $(3,3)$. The non-degeneracy allows us to define a canonical $2$-tensor $\omega^{*}$ dual to the symplectic form $\omega$ by requiring that
\[
	-2 \langle \omega^{*}, v\wedge w \rangle = \omega(v,w) 
\] 
for every $v,w \in \R^{4}$. In our case, we have 
\[	\omega^{*}=e_{1} \wedge e_{3} +e_{2} \wedge e_{4} \ , \] 
and we notice that $\langle \omega^{*} , \omega^{*} \rangle=1$. The group $\Sp(4,\R)$ acts naturally by isometries on $(V, \langle \cdot , \cdot \rangle)$, and preserves $\omega^{*}$. Therefore, it acts isometrically on $(\omega^{*})^{\perp}$, which is a five-dimensional real vector space endowed with an inner product of signature $(2,3)$. Tracing this action, we can define a continuous group homomorphism $\Sp(4, \R) \rightarrow \SO_{0}(2,3)$, whose kernel only contains $\{\pm \Id\}$, thus giving the aforementioned isomorphism. \\
\\
\indent We emphasize that inside the projective space $\Pp(V)$ we can embed:
\begin{itemize}
	\item the Grassmannian of $2$-planes in $\R^{4}$, which correspond to the submanifold of decomposable $2$-tensors;
	\item the Grassmannian of symplectic planes in $\R^{4}$, which can be characterized as those decomposable tensors $\phi$ such that $\langle \phi, \omega^{*} \rangle \neq 0$;
	\item the Lagrangians of $\R^{4}$, which are in bijection with decomposable $2$-tensors orthogonal to $\omega^{*}$;
	\item the Einstein Universe $\Ein^{1,2}$ can be identified with the points $p \in \Pp((\omega^{*})^{\perp})$ such that $p \wedge p=0$;
	\item the pseudo-hyperbolic space $\h^{2,2}$, as the projectivization of the vectors $\phi \in V$ such that $\langle \phi, \phi \rangle <0$.
\end{itemize}
We notice, in particular, that this identifies Lagrangians planes in $\R^{4}$ with points of the Einstein Universe. \\

We wish to interpret our embedding $f:\C \rightarrow \Gr_{2}(\R^{4})$ in terms of pseudo-Riemannian geometry.

\indent Let $f:\C \rightarrow \Gr_{2}(\R^{4})$ be the (convex) embedding, depending on a choice of quartic differential $q$, constructed in Section \ref{subsec:convex_surf}. Since it takes values in the Grassmannian of symplectic planes, there exists a unique lift $\tilde{f}:\C \rightarrow V$ such that $\langle \tilde{f}(z), \omega_{\R}\rangle=\frac{1}{2}$. We define $\tilde{\sigma}:\C \rightarrow (\omega^{*})^{\perp}$ by $\tilde{\sigma}(z)=2\tilde{f}(z)-\omega^{*}$, and we denote by $\sigma$ its projection into $\Pp(V)$. By construction $\langle \tilde{\sigma}(z), \tilde{\sigma}(z) \rangle=-1$ for every $z \in \C$, hence $\sigma$ defines an embedding of the complex plane into $\h^{2,2}$. Of course, the map $\sigma$ still depends on the choice of quartic differential $q$.

\begin{prop}\label{prop:rel_max_surface} The map $\sigma:\C \rightarrow \h^{2,2}$ is harmonic and conformal, hence $\sigma(\C)$ is a maximal surface in $\h^{2,2}$.
\end{prop}
\begin{proof} Since $\omega_{\R}$ is $\nabla$-parallel, we can write 
\[	\sigma(z)=2f(z)-\omega_{\R}=u_{1}(z)\wedge u_{3}(z)-u_{2}(z)\wedge u_{4}(z) \ .
\]
Using Equation (\ref{eq:derivatives 1}) and the following covariant derivatives of $u_{2}(z)$ and $u_{4}(z)$

\begin{align} \label{eq:derivatives 2}
\begin{split}
	\nabla_{\frac{\partial}{\partial x}}u_{2}(z)&=h_{2}^{\frac{1}{2}}h_{1}^{-\frac{1}{2}}u_{1}(z)+\Re(q)h_{1}u_{2}(z)+(\Im(q)h_{1}-\frac{1}{2}h_{1}^{-1}\partial_{y}h_{1})u_{4}(z)\\
	\nabla_{\frac{\partial}{\partial x}}u_{4}(z)&=-h_{2}^{\frac{1}{2}}h_{1}^{-\frac{1}{2}}u_{3}(z)-\Re(q)h_{1}u_{4}(z)+(\Im(q)h_{1}+\frac{1}{2}h_{1}^{-1}\partial_{y}h_{1})u_{2}(z)\\
	\nabla_{\frac{\partial}{\partial y}}u_{2}(z)&=-h_{2}^{\frac{1}{2}}h_{1}^{-\frac{1}{2}}u_{3}(z)-\Im(q)h_{1}u_{2}(z)+(\Re(q)h_{1}+\frac{1}{2}h_{1}^{-1}\partial_{x}h_{1})u_{4}(z)\\
	\nabla_{\frac{\partial}{\partial y}}u_{4}(z)&=-h_{2}^{\frac{1}{2}}h_{1}^{-\frac{1}{2}}u_{1}(z)+\Im(q)h_{1}u_{4}(z)+(\Re(q)h_{1}-\frac{1}{2}h_{1}^{-1}\partial_{x}h_{1})u_{2}(z)
\end{split}
\end{align}

we deduce that 
\begin{align*}
	\langle \nabla_{\frac{\partial}{\partial x}} \sigma, \nabla_{\frac{\partial}{\partial y}}\sigma \rangle&=0 \\
	\langle \nabla_{\frac{\partial}{\partial x}} \sigma, \nabla_{\frac{\partial}{\partial x}} \sigma \rangle&=\langle \nabla_{\frac{\partial}{\partial y}} \sigma, \nabla_{\frac{\partial}{\partial y}} \sigma \rangle=4h_{1}^{-1}h_{2}
\end{align*}
which means that the embedding is conformal.\\
As for the harmonic condition, since $\h^{2,2}$ is umbilical in $\R^{2,3}$, it is sufficient to check that 
\[
	\nabla_{\frac{\partial}{\partial \bar{z}}}\nabla_{\frac{\partial}{\partial z}}\sigma =0 \ \ \ (\mathrm{mod} \ \sigma) \ .
\]
Again, using Equations (\ref{eq:derivatives 1}) and (\ref{eq:derivatives 2}), a direct computation shows that
\[
	\nabla_{\frac{\partial}{\partial \bar{z}}}\nabla_{\frac{\partial}{\partial z}}\sigma=4h_{1}^{-1}h_{2}\sigma
\]
and the proof is complete.
\end{proof}

\begin{remark} The proof of the proposition above also shows that the induced metric on the maximal surface is $4h^{-1}h_{2}|dz|^{2}$. In particular, following the argument of Proposition \ref{prop:quasi-iso}, it is quasi-isometric to $(\mathbb{C}, 4|q|^{\frac{1}{2}})$, hence complete.
\end{remark}

We remark that, viewing $\sigma(\C)$ and $f(\C)$ as embedded inside $\Pp(V)$, since we can choose lifts to $V$ that differ only by a translation by $\omega^{*}$, they share the same boundary at infinity, which is a curve in the Einstein Universe, or, equivalently, in the space of Lagrangians. We will study the properties of the boundary curve in the next section.

\begin{remark}The maximal surface in $\h^{2,2}$ with constant quartic differential coincides with the horospherical surface, described in \cite{Tambu_poly}, embedded in a copy of anti-de Sitter space inside $\h^{2,2}$ . In particular, its boundary at infinity is a future-directed, negative, light-like polygon in the Einstein Universe (see next section).
\end{remark}

Moreover, the proof of Proposition \ref{prop:rel_max_surface} shows that the embedding data of the maximal surface $\sigma(\C)$, which determines it up to post-compostion by global isometries, only depend on the quartic differential $q$ on the complex plane and on the solution to Equation (\ref{eq:system}). We record this remark for future use as a proposition.

\begin{prop} \label{prop: qd and metric enough for maximal surface}
	Two planar maximal surfaces $\sigma_1: \C \to \h^{2,2}$ and $\sigma_2: \C \to \h^{2,2}$, defined as in Proposition \ref{prop:rel_max_surface}, which share quartic differential and solutions $h$ to \eqref{eq:system} agree up to post-composition by a global isometry.
\end{prop}

\section{Moduli space of future-directed negative light-like polygons in $\Ein^{1,2}$}
As we will see in the next section, a future-directed, negative, light-like polygon will appear as a boundary at infinity of a maximal embedding of the complex plane in $\h^{2,2}$ with an associated polynomial quartic differential. In this section we define these geometric objects and parameterize their moduli space under the conformal action of $\SO_{0}(2,3)$.

\begin{defi}\label{def:light_poly}A light-like polygon in the Einstein Universe $\Ein^{1,2}$ is an oriented embedded one-dimensional simplicial complex $\Delta$ with a finite number of vertices such that every edge is a photon (i.e. contained in the projection of an isotropic plane of $\R^{2,3}$). We will also assume that $\Delta$ is a generator of $\pi_{1}(\Ein^{1,2})\cong \Z$. We say that a light-like polygon is negative if it can be lifted to a cone $\widetilde{\Delta}$ in $\R^{2,3} \setminus \{0\}$, such that the inner product of any two non-collinear points is non-positive and vanishes if and only if their projections belong to the same edge of $\Delta$.
\end{defi}

In the above definition, a cone in $\R^{2,3} \setminus \{0\}$ is intended as subset of $\R^{2,3}$ 
that is invariant under multiplication by positive scalars. There are two cones that occur as possible lifts, but the condition of being negative is unaffected by the choice of cone. Moreover, an orientation of $\Delta$ will be given by an enumeration in the set of vertices.

\begin{remark}\label{rmk:neg_vertices}It is sufficient to check the negativity condition between pairs of non-consecutive vertices. Namely, if $\{v_{i}\}_{i=1, \dots, n}$ are vectors that generate the half-lines of the cone $\widetilde{\Delta}$ that project to vertices, then every other point in $\widetilde{\Delta}$ can be written as $p=\lambda(tv_{i}+(1-t)v_{i+1})$ for some $i \in \{ 1, \dots, n\}$ (where the indices are to be intended modulo $n$), $\lambda >0$ and $t \in [0,1]$. Therefore, 
\begin{align*}
	\langle p,q \rangle &=\langle \lambda(tv_{i}+(1-t)v_{i+1}), \mu(sv_{j}+(1-s)v_{j+1})\rangle <0
	%&=\lambda\mu(ts\langle v_{i}, v_{j} \rangle+ t(1-s)\langle v_{i}, v_{j+1} \rangle + (1-t)s \langle v_{i+1}, v_{j} \rangle + (1-t)(1-s)\langle v_{i+1}, v_{j+1} \rangle < 0.
\end{align*}
as soon as $p$ and $q$ do not project onto the same edge, under the assumption that the inner product between any pair of non-consecutive vertices is negative.
\end{remark}

\begin{defi}Fix an oriented time-like $3$-plane $W=\Span(w_{1},w_{2},w_{3})$ in $\R^{2,3}$. A negative light-like polygon is future-directed if for every pair of vectors $v_{i},v_{i+1} \in \widetilde{\Delta}$ that project to consecutive vertices of $\Delta$, we have
\[
  \dVol(v_{i},v_{i+1},w_{1},w_{2},w_{3})>0 \ ,
\]
where $\dVol$ denotes the standard volume form in $\R^{2,3}$. 
\end{defi}

We will denote by $\mathcal{LP}^{-}_{k}$, the space of future-directed negative light-like polygons in $\Ein^{1,2}$ with $k$ vertices. The group $\SO_{0}(2,3)$ acts on $\mathcal{LP}_{k}^{-}$, since its conformal action on the Einstein Universe sends photons into photons, preserves the sign of the inner products and preserves the orientation and the time-orientation of $\R^{2,3}$. We indicate with $\mathcal{MLP}^{-}_{k}$ the quotient by this action.
We note that we can see $\mathcal{MLP}^{-}_{k}$ as the quotient $\mathcal{TLP}^{-}_{k}/\Z_{k}$, where $\mathcal{TLP}_{-}^{k}$ denotes the moduli space of future-directed negative light-like polygons with a marked vertex and $\Z_{k}$ acts by change of marking.

\begin{prop}\label{prop:4vertices}There exists a unique future-directed negative light-like quadrilateral in $\Ein^{1,2}$, up to the action of $\SO_{0}(2,3)$, i.e. the space $\mathcal{MLP}^{-}_{4}$ consists of only one point.
\end{prop}
\begin{proof}Let $\Delta$ be a future-directed negative light-like polygon in $\Ein^{2,1}$ with vertices $\{p_{1}, p_{2}, p_{3}, p_{4}\}$. Let $\widetilde{\Delta}$ be the cone in $\R^{2,3}\setminus \{0\}$ given by Definition \ref{def:light_poly}. We denote by $\{v_{1}, v_{2}, v_{3}, v_{4}\}$ some light-like vectors in $\R^{2,3}$ that generate the lifts of the vertices, i.e. the half-line that projects to the vertex $p_{i}$ is given by $\{tv_{i} \ | \ t>0\}$. We can arrange $v_{i}$ so that
\[
	\langle v_{1}, v_{3} \rangle = \langle v_{2}, v_{4} \rangle=-1 \ .
\]
In particular, the vectors $v_{i}$ are linearly independent.\\
Let us denote by $\{e_{1}, e_{2}, e_{3}, e_{4}, e_{5}\}$ a base of $\R^{2,3}$ such that the inner product in these coordinates is given by
\[
	\langle x,y \rangle= x_{1}y_{3}+x_{2}y_{4}+x_{3}y_{1}+x_{4}y_{2}-x_{5}y_{5} \ .
\]
We fix the orientation on $\R^{2,3}$ given by the volume form 
\[
 \dVol=dx_{1}\wedge dx_{2}\wedge dx_{3}\wedge dx_{4} \wedge dx_{5}
\]
and the time-like orientation induced by the time-like $3$-plane 
\[
 W=\Span(e_{1}-e_{3},e_{2}-e_{4},e_{5}) \ .
\]
Now, the linear map
\begin{align*}
	\Span(\{v_{i}\}_{i=1,\dots 4}) &\rightarrow \Span(\{e_{i}\}_{i=1,\dots,4})\\
			v_{i} &\mapsto \epsilon(i)e_{i} \ \ \ \ \text{with $\epsilon(i)=1$ if $i=1,2$ and $-1$ otherwise}
\end{align*}
preserves the inner product and the time-orientation, hence it can be extended to an element of $\SO_{0}(2,3)$. We thus deduce that every future-directed negative light-like polygon with $4$ vertices is equivalent to the standard quadrilateral $\Delta_{4}$ with vertices $([e_{1}], [e_{2}], [e_{3}], [e_{4}])$ and associated cone
\[
	\widetilde{\Delta_{4}}=\R^{+}([e_{1}, e_{2}] \cup [e_{2}, -e_{3}] \cup [-e_{3}, -e_{4}] \cup [-e_{4}, e_{1}])
\]
where $[e_{i}, \pm e_{j}]$ denotes the Euclidean segment joining $e_{i}$ and $\pm e_{j}$.  
\end{proof}

\begin{prop}\label{prop:5vertices}There exists a unique future-directed negative light-like polygon in $\Ein^{1,2}$ with $5$ vertices, up to the action of $\SO_{0}(2,3)$.
\end{prop}
\begin{proof}We use the same notation as in the previous proposition. It follows from the proof of 
Proposition \ref{prop:4vertices} that we can assume that $v_{1}=e_{1}, v_{2}=e_{2}$ and $v_{3}=-e_{3}$. Now the lift of the fourth vertex satisfies:
\[
	\langle v_{4}, v_{4} \rangle =0 \ \ \  \langle v_{4}, v_{5} \rangle=0 \ \ \ \langle v_{4}, e_{3}\rangle =0
\]
and 
\[
	 \langle e_{1}, v_{4} \rangle <0 \ \ \ \langle e_{2}, v_{4} \rangle <0 \ ,
\]
hence it must belong to the set 
\[
	U=\{ x \in \R^{2,3} \ | \ 2x_{1}x_{3}+2x_{2}x_{4}-x_{5}^{2}=0, x_{1}=0, x_{4}<0, x_{3}<0\}.
\]
We claim that the subgroup $H=\Stab(p_{1}) \cap \Stab(p_{2}) \cap \Stab(p_{3})<\SO_{0}(2,3)$ acts transitively on this set. Namely, 
\begin{itemize}
	\item [i)] if $x_{2}=0$, then necessarily $x_{5}=0$ and it is sufficient to act via diagonal matrices;
	\item [ii)] if $x_{2}\neq 0$, then we must have $x_{5} \neq 0$ and we can choose a representative of $v_{4}$ such that $x_{5}=\pm \sqrt{2}$. It is easy to check that diagonal matrices act transitively on points with $x_{5}=\sqrt{2}$ and on those with $x_{5}=-\sqrt{2}$. We then notice that if $v \in U$ with $x_{5}=\sqrt{2}$ and $v'\in U$ differs from $v$ only for the sign of the last component, then the polygon $\Delta$ with vertices $p_{1},p_{2},p_{3}$ and $[v]$ is the image of the polygon $\Delta'$ with vertices $p_{1},p_{2},p_{3}$ and $[-v']$ under the diagonal matrix $D=\diag(-1,-1,-1,-1,1)\in H$. 
\end{itemize}
Therefore, it is enough to show that there exists an element $A \in H$ that sends a point satisfying $i)$ to a point satisfying $ii)$. Now, the linear transformation $A$ determined by
\[
	A(e_{i})=e_{i}  \ \ \ \text{for} \ i=1,2,3 \ \ \ A(e_{4})=e_{2}+e_{4}+\sqrt{2}e_{5} \ \ \ \text{and} \ \ \ A(e_{5})=\sqrt{2}e_{2}+e_{5}
\]
preserves the inner product and sends $e_{3}+e_{4}$ to $e_{2}+e_{3}+e_{4}+\sqrt{2}e_{5}$. Moreover, it lies in $\SO_{0}(2,3)$, an explicit path connecting $A$ to the identity being given by the linear transformations $\{A_{t}\}_{t \in [0,1]}$ such that
\[
	A_{t}(e_{i})=e_{i}  \ \ \ \text{for} \ i=1,2,3 \ \ \ A_{t}(e_{4})=te_{2}+e_{4}+\sqrt{2t}e_{5} \ \ \ \text{and} \ \ \ A_{t}(e_{5})=\sqrt{2t}e_{2}+e_{5} \ .
\]
Hence it fulfills all our requirements. As a consequence, we can suppose that the fourth vertex is $p_{4}=[e_{3}+e_{4}]$. \\
By a similar reasoning, the lift of the last vertex belongs to
\[
	W=\{ x \in \R^{5} \ | \ 2x_{1}x_{3}+2x_{2}x_{4}-x_{5}^{2}=0, x_{3}=0, x_{1}+x_{2}=0, x_{1}>0, x_{4}<0 \}.
\]
As before, we can suppose that $x_{5}=\sqrt{2}$ and it is clear that the diagonal matrices of the form $\diag(a,a,a^{-1},a^{-1},1)$ with $a \in \R^{+}$ act transitively on $W$. \\
We thus deduce that every future-directed, negative light-like polygon with $5$ vertices is equivalent to the standard light-like penthagon $\Delta_{5}$ with vertices $([e_{1}], [e_{2}], [e_{3}], [e_{3}+e_{4}], [e_{1}-e_{2}-e_{4}+\sqrt{2}e_{5}])$ and associated cone
\begin{align*}
	\widetilde{\Delta_{5}}=\R^{+}(&[e_{1}, e_{2}] \cup [e_{2}, -e_{3}] \cup [-e_{3}, -e_{3}-e_{4}] \cup \\
		 &[-e_{3}-e_{4}, e_{1}-e_{2}-e_{4}+\sqrt{2}e_{5}] \cup [e_{1}-e_{2}-e_{4}+\sqrt{2}e_{5}, e_{1}])\ .
\end{align*}
\end{proof}

Let us now consider the first non-trivial case of hexagons. We find an explicit parametrization of $\mathcal{TLP}_{6}^{-}$ that surprisingly shows that this moduli space has two connected components.

\begin{prop}\label{prop:6vertices} The moduli space of marked future-directed negative light-like hexagons in $\Ein^{1,2}$ is a topological manifold homeomorphic to 
\[
    \{(s,t) \in \R^{2} \ | \ s\geq 0, \ st\neq 2\}/(0,t)\sim (0,-t)
\]
\end{prop}
\begin{proof} It follows from the proof of Proposition \ref{prop:5vertices} that, up to the action of $\SO_{0}(2,3)$, we can assume that $v_{1}=e_{1}$, $v_{2}=e_{2}$, $v_{3}=-e_{3}$ and $v_{4}=-e_{3}-e_{4}$. The lift of the fifth vertex $v_{5}$ must now satisfy
\[
	\langle v_{1}, v_{5} \rangle <0 \ \ \  \langle v_{2}, v_{5} \rangle<0 \ \ \ \langle v_{3}, v_{5}\rangle <0
\]
and 
\[
	 \langle v_{4}, v_{5} \rangle =0 \ \ \ \langle v_{5}, v_{5} \rangle =0 \ ,
\]
hence $v_{5}$ belongs to the set
\[
 U=\{ x \in \R^{2,3} \ | \ x_{1}>0, \  x_{1}=-x_{2}, \ x_{3}<0, \ x_{4}<0, \ 2x_{1}x_{3}+2x_{2}x_{4}-x_{5}^{2}=0\} \ .
\]
We can still renormalize partly the position of $v_{5}$ acting by the group $H=\Stab(p_{1})\cap \Stab(p_{2})\cap\Stab(p_{3})\cap \Stab(p_{4})$ which consists of diagonal matrices of the form $\diag(a,a,a^{-1},a^{-1},1)$ with $a\neq 0$. Because $x_{1}>0$ and $x_{3}<0$, we can find $a>0$ such that $a^{2}=\frac{-x_{3}}{x_{1}}$. This implies that $ax_{1}=-a^{-1}x_{3}$. Therefore, after renormalizing by the action of $H$ we can assume that $v_{5}$ has coordinates
\[
    v_{5}=(x_{1},-x_{1},-x_{1},x_{4},x_{5})
\]
for some $x_{1}>0$. Since we are interested only in the projective class of $v_{5}$ we can assume that $x_{1}=1$. Moreover, because $v_{5}$ must be isotropic, we deduce that 
\[
    x_{4}=-1-\frac{x_{5}^{2}}{2} \ .
\]
The lift of a generic fifth vertex has thus coordinates
\[
    v_{5}=\left(1,-1,-1,-1-\frac{x_{5}^{2}}{2},x_{5}\right)
\]
for some $x_{5} \in \R$. However, we notice that the matrix $\diag(-1,-1,-1,-1,1) \in \SO_{0}(2,3)$ fixes the vertices $p_{i}=[v_{i}]$ for $i=1, \dots, 4$ and sends $[v_{5}]$ to $[v_{5}']$ where $v_{5}'$ only differs from $v_{5}$ by the sign of the last component. Hence, we can assume $x_{5}\geq 0$. \\
\indent The lift of the sixth vertex $v_{6}$ must be chosen so that
\[
	\langle v_{2}, v_{6} \rangle <0 \ \ \  \langle v_{3}, v_{6} \rangle<0 \ \ \ \langle v_{4}, v_{6}\rangle <0
\]
and 
\[
	 \langle v_{1}, v_{5} \rangle =0 \ \ \ \langle v_{5}, v_{6} \rangle =0 \ \ \ \langle v_{6}, v_{6} \rangle =0 \ ,
\]
hence $v_{6}$ belongs to the set
\[
 U=\{ y \in \R^{2,3} \ | \ y_{1}>0, \  y_{1}>-y_{2}, \ y_{3}=0, \ y_{4}<0, \ 2y_{1}y_{3}+2y_{2}y_{4}-y_{5}^{2}=0, \ \langle v_{5}, v_{6} \rangle =0 \} \ .
\]
Because $y_{4}<0$ and we are only interested in the projective class of $v_{6}$ we can assume that $y_{4}=-1$. In particular, we deduce that $v_{6}$ is isotropic only if $y_{2}=-\frac{y_{5}^{2}}{2}$. Moreover,
\[
 0=\langle v_{5}, v_{6} \rangle=1-y_{1}-\frac{y_{5}^{2}}{2}\left(-1-\frac{x_{5}^{2}}{2}\right)-x_{5}y_{5}
\]
implies that 
\[
 y_{1}=1+\frac{y_{5}^{2}}{2}+\frac{y_{5}^{2}x_{5}^{2}}{4}-x_{5}y_{5} \ .
\]
On the other hand, we must have $y_{1}>-y_{2}$, and, imposing this condition, one finds
\[
    \frac{x_{5}^{2}y_{5}^{2}}{4}-x_{5}y_{5}+1>0
\]
which gives the constraint $x_{5}y_{5}\neq 2$. Therefore, a generic sixth vertex has coordinates
\[
 v_{6}=\left(1+\frac{y_{5}^{2}}{2}+\frac{y_{5}^{2}x_{5}^{2}}{4}-x_{5}y_{5}, -\frac{y_{5}^{2}}{2},0,-1,y_{5}\right)
\]
with $x_{5}\geq 0$ and $x_{5}y_{5}\neq 2$. It can be easily verified that for every such choice of $x_{5}$ and $y_{5}$ the associated polygons are future-directed.\\
Moreover, we observe that, if $x_{5}=0$, the matrix $\diag(-1,-1,-1,-1,1) \in \SO_{0}(2,3)$ stabilizes the first five vertices and sends $p_{6}=[v_{6}]$ to $[v_{6}']$ where $v_{6}'$ only differs from $v_{6}$ by the sign of the last component. Therefore, we conclude that
\[
    \mathcal{TLP}_{6}^{-}=\{(s,t) \in \R^{2} \ | \ s\geq 0, \ st\neq 2\}/(0,t)\sim (0,-t)
\]
which has two connected components. 
\end{proof}

\begin{remark} The special point in $\mathcal{TLP}_{6}^{-}$ with $s=t=0$ coorresponds to the unique light-like hexagon in $\Ein^{1,1}\subset \Ein^{1,2}$ (cfr. \cite{Tambu_poly}).
\end{remark}

In general, we can prove the following

\begin{teo}\label{thm:moduli_polygons}The moduli space of future-directed negative light-like polygons with $k\geq 6$ vertices is a (possibly disconnected) real orbifold of dimension $2(k-5)$. 
\end{teo}
\begin{proof} Recall that we can see the moduli space as the quotient
\[
	\mathcal{MLP}_{k}^{-}=\mathcal{TLP}_{k}^{-}/\Z_{k}
\]
where $\mathcal{TLP}_{k}^{-}$ denotes the moduli space of future-directed negative light-like polygons with a marked vertex. Thus it is enough to show that $\mathcal{TLP}_{k}^{-}$ is a (topological) manifold of dimension $2(k-5)$. \\
The marking on the polygon induces a natural enumeration of the vertices that is compatible with the orientation. From the proof of Proposition \ref{prop:6vertices}, we learn that, up to the action of $\SO_{0}(2,3)$, we can assume that $(p_{1}, p_{2}, p_{3}, p_{4},p_{5})=([e_{1}], [e_{2}], [e_{3}], [e_{3}+e_{4}],[e_{1}-e_{2}-e_{3}-(1+s^2/2)e_{4}+se_{5}])$ with $s \geq 0$. Now, all other vertices $p_{j}$ with $6\leq j \neq k-1$ must be chosen in an open subset of the light-cone of $p_{j-1}$ (the open subset is determined by the negativity conditions of the inner product with the previous non-consecutive vertices, which is sufficient to obtain a negative light-like polygon by Remark \ref{rmk:neg_vertices}, and by the future-directed condition). The only exception is given by $p_{k}$ which must lie in the intersection between the light-cone of $p_{k-1}$ and $p_{1}$, which is a real manifold of dimension $1$. Therefore, the moduli space of marked future-directed negative light-like polygon in $\Ein^{1,2}$ with $k$ vertices is a manifold of dimension
\[
	\dim_{\R}\mathcal{TLP}_{k}^{-}=1+2(k-6)+1=2(k-5) \ .
\]
\end{proof}

\begin{remark} As shown in the case of hexagons (Proposition \ref{prop:6vertices}), the open sets where the vertices $p_{k}$ with $k\geq 6$ can be chosen may not be connected, hence $\mathcal{TLP}^{-}_{k}$ may be disconnected.
\end{remark}

\section{From polynomial quartic differentials to light-like polygons}

%The main aim of this section is to 

We define a map
\[
  \tilde{\alpha}: \mathcal{TQ}_{n} \rightarrow \mathcal{TLP}^{-}_{n+4}
\]
between the space $\mathcal{TQ}_{n}$ of monic, centered polynomial quartic differentials of degree $n$ and the space $\mathcal{TLP}^{-}_{n+4}$ of future-directed negative light-like polygons with $(n+4)$ vertices (and a marked vertex) by associating, to a polynomial quartic differential, the boundary at infinity of the maximal surface in $\h^{2,2}$ constructed in Section \ref{subsec:maxsurface}. Equivalently,  we could imagine the target as the boundary at infinity of the convex immersion in the Grassmannian of symplectic planes of $\R^{4}$ described in Section \ref{subsec:convex_surf}.
This map $\tilde{\alpha}$  is equivariant with respect to the $\Z_{n+4}$ action and so induces a map
\[
  \alpha:\mathcal{MQ}_{n} \rightarrow \mathcal{MLP}^{-}_{n+4} 
\]
between the moduli space of polynomial quartic differentials of degree $n$ and future-directed negative light-like polygons in the Einstein Universe with $(n+4)$-vertices. 

The aim of this section is to show that this latter map $\alpha$ is a homeomorphism onto a connected component of $\mathcal{MLP}^{-}_{n+4}$. This is the content of Theorem~\ref{thmB}. \\

The section begins by showing that the maximal surface defined by a polynomial quartic differential has the asymptotic geometry we claim above:  this requires demonstrating a 'Stokes' phenomenon for solutions to \eqref{eq:system}. The proof of Theorem~\ref{thmB} then  proceeds as a succession of results of properties of the map $\alpha$: we prove $\alpha$ is continuous (Proposition~\ref{prop:continuity}) and proper (Corollary~\ref{cor:properness}), while $\tilde{\alpha}$ is injective (Proposition~\ref{prop:injectivity}).  These properties together then imply Theorem~\ref{thmB}, which appears here as the summary Theorem~\ref{main thm}.

\subsection{The boundary at infinity of a maximal surface in $\h^{2,2}$ with polynomial growth}
\indent Given a polyonimal quartic differential $q$, in Section \ref{subsec:convex_surf}, we constructed a convex embedding of $\C$ in the Grassmannian of symplectic planes, and showed that it shares the same boundary at infinity of a complete maximal surface $\Sigma$ in $\h^{2,2}$. We will refer to $\Sigma$ as the maximal surface in $\h^{2,2}$ with polynomial growth $q$. \\
\\
\indent We recall some general facts about maximal surfaces in $\h^{2,2}$. We denote by $\widehat{\h}^{2,2}$ the space of unitary time-like vectors in $\R^{2,3}$, which is a double cover of $\h^{2,2}$ under the natural projection $\pi:\R^{2,3}\setminus \{0\} \rightarrow \R\Pp^{5}$. Similarly we denote by $\widehat{\Ein}^{1,2}$ the space of isotropic vectors of $\R^{2,3}$ up to positive scalar multiples. Again the projection $\pi:\widehat{\Ein}^{1,2} \rightarrow \Ein^{1,2}$ is a double cover, and we can identify $\widehat{\Ein}^{1,2}$ with the boundary at infinity of $\widehat{\h}^{2,2}$. The map
\begin{align*}
	F: D^{2}\times S^{2} &\rightarrow \widehat{\h}^{2,2}\\
		(z,w) &\mapsto \left(\frac{2}{1-\|z\|^{2}}z, \frac{1+\|z\|^{2}}{1-\|z\|^{2}}w\right)
\end{align*}
is a diffeomorphism, hence $D^{2} \times S^{2}$ is a model for $\widehat{\h}^{2,2}$, if endowed with the pull-back metric
\[
	F^{*}g_{\widehat{\h}^{2,2}}=\frac{4}{(1-\|z\|^{2})^{2}}|dz|^{2}-\left(\frac{1+\|z\|^{2}}{1-\|z\|^{2}}\right)^{2}g_{S^{2}}.
\]
Here $g_{S^{2}}$ is the standard round metric on the unit sphere. The map $F$ extends to homeomorphism of the boundary
\begin{align*}
	\partial_{\infty}F: S^{1} \times S^{2} &\rightarrow \widehat{\Ein}^{1,2}\\
			(z,w) &\mapsto (z,w) \ .
\end{align*}

\begin{lemma}Under the homeomorphism $\partial_{\infty}F$, graphs of parameterised geodesics arcs in $S^{2}$ correspond to light-like segments in $\widehat{\Ein}^{1,2}$.
\end{lemma}
\begin{proof}Let $\gamma:[0,\theta_{0}]\subset S^{1} \rightarrow S^{2}$ be a unit-length parameterization of a geodesic. We can suppose that $\gamma(0)=(0,0,1)$ and $\gamma(\theta)=(\sin(\theta)\cos(\alpha), \sin(\theta)\sin(\alpha), \cos(\theta))$ for some $\alpha \in [0, 2\pi]$ for every $\theta \in [0, \theta_{0}]$. Now, the graph of $\gamma$ consists of points of the form
\[
	(\cos(\theta), \sin(\theta), \sin(\theta)\cos(\alpha), \sin(\theta)\sin(\alpha), \cos(\theta)) \in S^{1} \times S^{2} \ ,
\]
for $\theta \in [0, \theta_{0}]$. On the other hand, the light-like segment joining $(1,0,0,0,1)$ and 
$(\cos(\theta_{0}), \sin(\theta_{0}), \sin(\theta_{0})\cos(\alpha), \sin(\theta_{0})\sin(\alpha), \cos(\theta_{0})) \in \widehat{\Ein}^{1,2}$ is given by
\[
	\scalebox{0.8}{$\frac{1}{\sqrt{t^{2}+(1-t)^{2}}}((1-t)+t\cos(\theta_{0}), t\sin(\theta_{0}), t\sin(\theta_{0})\cos(\alpha), t\sin(\theta_{0})\sin(\alpha), (1-t)+t\cos(\theta_{0})) \ \ \ t \in [0,1]$}
\] 
and this coincides exactly with the graph of $\gamma$ if we define $\theta$ so that
\[
	\cos(\theta)=\frac{(1-t)+t\cos(\theta_{0})}{\sqrt{t^{2}+(1-t)^{2}}} \ \ \ \text{and} \ \ \ 
	\sin(\theta)=\frac{t\sin(\theta_{0})}{\sqrt{t^{2}+(1-t)^{2}}} \ .
\]
Notice that this is the correspondence between points in $\widehat{\Ein}^{1,2}$ and $S\times S^{2}$ given by the homeomorphism $\partial_{\infty}F$.
\end{proof}

\begin{lemma}\label{lm:negativity}Let $v,v' \in \widehat{\Ein}^{1,2}$ such that $\langle v, v'\rangle<0$. Then $\partial_{\infty}F^{-1}(v)=(\theta, p)$ and $\partial_{\infty}F^{-1}(v')=(\theta', p')$ with $d_{S^{2}}(p,p')<d_{S^{1}}(\theta,\theta')$.
\end{lemma}
\begin{proof}Without loss of generality we can suppose $v=(1,0,0,0,1)$, so that $\theta=0$ and $p=(0,0,1)$. A general point $v' \in \widehat{\Ein}^{1,2}$ can be written as
\[
	v'=(\cos(\theta'), \pm\sin(\theta'), \sin(\alpha)\cos(\beta), \sin(\alpha)\sin(\beta), \cos(\alpha))
\]
for some $\beta \in [0,2\pi]$ and $\theta',\alpha \in [0, \pi]$. By assumption
\[
	\langle v, v' \rangle =\cos(\theta')-\cos(\alpha)<0
\]
hence $\alpha<\theta'\leq \pi$. By construction
\[
	\partial_{\infty}F^{-1}(v')=(\theta', p'),
\]
with $p'=(\sin(\alpha)\cos(\beta), \sin(\alpha)\sin(\beta), \cos(\alpha))$. Therefore,
\[
	d_{S^{2}}(p,p')=\alpha<\theta'=d_{S^{1}}(\theta,\theta')\leq \pi
\]
as claimed. 
\end{proof}

This model is also useful to understand complete space-like surfaces in $\widehat{\h}^{2,2}$. 

\begin{prop}Let $\widehat{\Sigma}$ be a complete, connected, space-like surface in $\widehat{\h}^{2,2}$. Then $\widehat{\Sigma}$ is the graph of a $2$-Lipshitz map $f:D^{2} \rightarrow S^{2}$.
\end{prop} 
\begin{proof} Let $\pr_{1}:\widehat{\Sigma} \rightarrow D^{2}$ denote the restriction of the projection onto the first factor. Let $g_{\Sigma}$ be the induced metric on $\widehat{\Sigma}$. Then
\[
	g_{\Sigma}\leq \pr_{1}^{*}\left(\frac{4}{(1-\|z\|^{2})^{2}}|dz|^{2}\right) \ .
\]
Since $g_{\Sigma}$ is complete, we deduce that $\pr_{1}: \widehat{\Sigma} \rightarrow D^{2}$ is a proper immersion, hence a covering. Since $D^{2}$ is simply connected and $\widehat{\Sigma}$ is connected, $\pr_{1}$ is a diffeomorphism. Therefore, $\widehat{\Sigma}$ is the graph of a function $f:D^{2} \rightarrow S^{2}$. Since $\widehat{\Sigma}$ is space-like, for every $z \in D^{2}$ and $v \in T_{z}D^{2}$ we must have
\[
	\frac{4}{(1-\|z\|^{2})^{2}}\|v\|^{2}-\left(\frac{1+\|z\|^{2}}{1-\|z\|^{2}}\right)^{2}\|df_{z}(v)\|^{2} >0
\]
which implies that
\[
	\|df_{z}(v)\| \leq \frac{2}{1+\|z\|^{2}}\|v\| \leq 2\|v\| \ ,
\]
so $f$ is $2$-Lipschitz.
\end{proof}

In particular, it follows from the proof of the above proposition, that the boundary at infinity of $\widehat{\Sigma}$ is the graph of a $1$-Lipschitz map $\partial f:S^{1} \rightarrow S^{2}$. 

\begin{prop}\label{prop:boundary1}The boundary at infinity of a maximal surface $\Sigma$ in $\h^{2,2}$ is always non-positive. Moreover, the boundary is negative if and only if it does not contain any light-like segments.
\end{prop}
\begin{proof}Consider a lift $\widehat{\Sigma}$ of $\Sigma$ in $\widehat{\h}^{2,2}$. By the previous discussion, the lift $\widehat{\Sigma}$ is the graph of a $2$-Lipschitz map, and its boundary at infinity $\widehat{\Delta}$ is the graph of a $1$-Lipschitz map $\partial f:S^{1} \rightarrow S^{2}$. Let $v_{0}=(\theta_{0}, p_{0}=\partial f(\theta_{0})) \in \widehat{\Delta}$. Let $v \in \widehat{\Delta}$ which corresponds to a point of the form $(\theta, \partial f (\theta)=p)$. The positivity condition on the inner-product is equivalent to (Lemma \ref{lm:negativity})
\[
	d_{S^{2}}(p_{0}, p) \geq d_{S^{1}}(\theta, \theta_{0}). \ 
\]
But on the other hand, since $\partial f$ is $1$-Lipschitz we know that
\[
	d_{S^{2}}(p_{0},p)=d_{S^{2}}(\partial f(\theta_{0}), \partial f(\theta)) \leq d_{S^{1}}(\theta, \theta_{0}).
\]
with the equality if and only if the restriction of $\partial f$ to the arc between $\theta_{0}$ and $\theta$ is the parameterisation of a geodesic between $p_{0}$ and $p$. Therefore, the boundary at infinity of $\Sigma$ is always non-positive and fails to be negative if and only if it contains light-like segments.\\
\end{proof}

\begin{prop}\label{prop:boundary}If $q$ is a polynomial quartic differential of degree $n$, the boundary at infinity of $\Sigma$ is a future-directed, negative, light-like polygon with $n+4$ vertices.
\end{prop}
\begin{proof} By Proposition~\ref{prop:boundary1}, we only need to show that the boundary at infinity is a future-directed light-like polygon of $n+4$ vertices. For simplicity we describe the boundary at infinity of the embedding in the Grassmannian of symplectic planes. Recall that this is obtained by $\nabla$-parallel transport of the section
\[
	f(z)=u_{1}(z) \wedge u_{3}(z)
\]
of $\Pp(\Lambda^{2}E_{\R})$ to the fibre over a base point $0 \in \C$. Let $A \in \SU(4)$ be the following unitary matrix
\[
	A=\frac{1}{\sqrt{2}}\begin{pmatrix} 1 & 0 & i & 0 \\
					    0 & 1 & 0 & i \\
					    1 & 0 & -i & 0 \\
					    0 & 1 & 0 & -i 
			    \end{pmatrix}
\]
expressed in the frame $F(z)=\{v_{i}(z)\}_{i=1, \dots 4}$ introduced in Section \ref{subsec:construction_min_surface}. By definition (see Section \ref{subsec:Titeica}),
\[
	u_{1}(z)=H(z)^{-\frac{1}{2}}Av_{2}(z) \ \ \ \text{and} \ \ \ u_{3}(z)=H(z)^{-\frac{1}{2}}Av_{4}(z) \ .
\]
The parallel transport $T(z):(\sE_{\R})_{z} \rightarrow (\sE_{\R})_{0}$ expressed in the real frame $\{u_{i}(z)\}_{i=1,\dots,4}$ is given by 
\[
	T(z)=A^{-1}H(0)^{\frac{1}{2}}\psi(z)H(z)^{-\frac{1}{2}}A	\ .
\]
Now, recall from Theorem~\ref{thm:asymptotic_flats} that in each standard half-plane $(U,w)$ and in each stable direction $J_{\alpha}$, the limit $M_{\alpha}$ of $H(0)^{\frac{1}{2}}\psi(w)H(w)^{-\frac{1}{2}}\psi_{0}(w)^{-1}$ exists as $|w| \to +\infty$. 
Note that we may write 
\begin{align}\begin{split}
    &\lim_{|w| \to \infty} 	A^{-1}H(0)^{\frac{1}{2}}\psi(w)H(w)^{-\frac{1}{2}}A \\
    &\lim_{|w| \to \infty}= A^{-1}H(0)^{\frac{1}{2}}\psi(w)H(w)^{-\frac{1}{2}}\psi_{0}(w)^{-1}A[A^{-1}\psi_{0}(w)A]\\
&=A^{-1}M_{\alpha}A\lim_{|w| \to \infty}A^{-1}\psi_{0}(w)A \ .
\end{split}
\end{align}
Therefore, working in this coordinate chart, we deduce that as $|w| \to +\infty$ along a stable direction, the embedding $f(w)$ has a limit point $p_{\alpha}$, given by the image of the limit point of the standard flat maximal surface, under the composition of $M_{\alpha}'=A^{-1}M_{\alpha}A$ and $A^{-1}S'\in \Sp(4,\R)$, where $M_{\alpha}$ is the limit found in Theorem \ref{thm:asymptotic_flats} and $S'=S\diag(1, \frac{1-i}{\sqrt{2}},i,\frac{1+i}{\sqrt{2}})$. We thus obtain the limits described in Table~2.

\medskip
\begin{table}[!htb]
\begin{center}
\begin{tabular}{l l l}
\hline
\textbf{Type of path $\gamma$} & \textbf{Direction $\theta$} & \textbf{Projective limit $p_\gamma$ of $\sigma(\gamma)$}\\
\hline
Quasi-ray& $\theta\in (0, \tfrac{\pi}{4})$  & $p_{++} = M_{++}'A^{-1}S'[e_{1}\wedge e_{4}]$\\
Quasi-ray& $\theta\in (\tfrac{\pi}{4}, \tfrac{\pi}{2})$  & $p_{+} = M_{+}'A^{-1}S'[e_{1}\wedge e_{4}]$\\
Quasi-ray& $\theta \in (\tfrac{\pi}{2}, \tfrac{3\pi}{4})$ & $p_{-} = M_{-}'A^{-1}S'[e_{3}\wedge e_{4}]$\\
Quasi-ray& $ \theta \in (\tfrac{3\pi}{4}, \pi)$ & $p_{--} = M_{--}'A^{-1}S'[e_{3}\wedge e_{4}]$\\

\hline\\
\end{tabular}
\end{center}
\caption{Limits of the maximal surface $\Sigma$ in a standard half-plane}\label{table:2}
\end{table}
Using that $M_{\alpha}=H(0)^{\frac{1}{2}}L_{\alpha}$, with $L_{\alpha}$ being the limit of $\psi(w)H(w)^{-\frac{1}{2}}\psi_{0}(w)^{-1}$ for $|w|\to +\infty$ along a stable direction in the sector $J_{\alpha}$, as well as the relations between the matrices $M_{++},M_{+},M_{-},M_{--}$ provided by Lemma \ref{lm:comparison_limits}, we obtain that
\begin{align*}
	M_{--}'A^{-1}S'[e_{3}\wedge e_{4}]&=M_{-}'A^{-1}S'[e_{3}\wedge e_{4}]=M_{+}'A^{-1}S'[e_{3}\wedge e_{4}] \\
	M_{++}'A^{-1}S'[e_{1}\wedge e_{4}]&=M_{+}'A^{-1}S'[e_{1}\wedge e_{4}]:
\end{align*}
that is, only two different points appear as limits along stable directions in each standard half-plane and those are obtained as images of the Stokes matrix $M_{+}'$.

Since $M_{+}'$, being an element of $\Sp(4,\R)$, preserves the symplectic form
\[
	\omega^{*}=e_{1}\wedge e_{3}+e_{2}\wedge e_{4}
\]
the two points $p_{--}=p_-$ and $p_+=p_{++}$ lie on a common photon: 
%by equation \eqref{eqn:bracket on V}, we have that
from the relationships above and equation \eqref{eqn:bracket on V}, we find that
\begin{align*}
<p_-, p_+> & = 	<M_{+}'A^{-1}S'[e_{3}\wedge e_{4}], M_{+}'A^{-1}S'[e_{1}\wedge e_{4}]>\\
&= <[e_{3}\wedge e_{4}], [e_{1}\wedge e_{4}]>\\
&=0.
\end{align*}
Here the middle equality follows from the invariance of the bracket by an action of an element $M_{+}'A^{-1}S' \in \Sp(4,\R)$.

The displayed computation of course expresses that $p_-$ and $p_+$ are orthogonal 
%because both $M_{\alpha}'$ and $A^{-1}S' \in \Sp(4,\R)$. We 
and so we deduce by the proof of Proposition \ref{prop:boundary1} that the light-like segment joining them is contained entirely in the boundary at infinity of $\Sigma$.
 Therefore, the asymptotic boundary of $\Sigma$ in each standard half-plane consists of a light-like segment with extreme points $M_{+}'A^{-1}S'[e_{3}\wedge e_{4}]$ and $M_{+}'A^{-1}S'[e_{1}\wedge e_{4}]$ .\\
Now, two segments belonging to two consecutive standard half-planes share a common vertex, because two standard half-planes intersect in only an open sector of angle $\pi/2$. (It is critical here that the intersection is restricted to but a single open sector angle $\pi/2$; a more relaxed definition of standard half-plane would likely not permit us to conclude that the intersection is but a single vertex.) Since there are $n+4$ standard half-planes, we conclude that the boundary at infinity of $\Sigma$ is a negative light-like polygon $\Delta$ with $n+4$ vertices. \\
Moreover, $\Delta$ is future-directed because, by definition, the time-orientation of the polygon depends only on the positions of two consecutive vertices and these are inherited from the standard flat maximal surface. 
\end{proof}

\subsection{Definition of the map and continuity}
The previous propositions allow us to define a map
\[
	\tilde{\alpha}:\mathcal{TQ}_{n} \rightarrow \mathcal{TLP}_{n+4}^{-}
\]
that associates, to a polynomial quartic differential of degree $n$, the boundary at infinity of the maximal surface in $\h^{2,2}$ with polynomial growth $q$. In order to show that $\tilde{\alpha}$ induces a map $\alpha$ between the moduli spaces, we need to check the equivariance of $\tilde{\alpha}$ with respect to the $\Z_{n+4}$-action.
First of all, let us describe how to encode the $\Z_{n+4}$-action. Given a monic polynomial quartic differential $q$ of degree $n$ in the complex plane, there are $n+4$ canonical directions corresponding to the set
\[
	D=\left\{ z \in \C \ | \ \arg(z)=\frac{2\pi j}{n+4} \ j=0, \dots, n+3 \right\} \ .
\]
Those can be understood as follows: if $q=z^{n}dz^{4}$, these are exactly the directions in which the quartic differential takes positive real values; in the general case, these directions are characterized by the fact that they are contained eventually in a unique standard $q$-half-plane, where they correspond to quasi-rays with angle $0$. If we fix one direction $\theta_{0}=\arg(z_{0})+\epsilon$ with $z_{0} \in D$ and $\epsilon>0$ small, we can see the action of the cyclic group $\Z_{n+4}$ as a rotation in this set.\\
\indent Let $\sigma:\C \rightarrow \h^{2,2}$ be a maximal embedding associated to $q$. Let $\Delta$ denote the future-directed negative light-like polygon in the boundary at infinity of $\Sigma=\sigma(\C)$. The direction $\theta_{0}$ gives a marking on $\Delta$ as follows: the path $\sigma(e^{i\theta_{0}}t)$ converges to a point in $\Delta$ as $t \to +\infty$. By Proposition \ref{prop:boundary1}, the limit point is a vertex $v_{0} \in \Delta$. We can then define
\begin{align*}
	\tilde{\alpha}:\mathcal{TQ}_{n} &\rightarrow \mathcal{TLP}^{-}_{n+4} \\
	q &\mapsto (\Delta, v_{0})\ .
\end{align*}
If we change $\sigma$ to $\sigma'$ by post-composition with an isometry $g$ of $\h^{2,2}$, the boundary at infinity becomes $\Delta'=g(\Delta)$, hence the two marked light-like polygons $(\Delta, v_{0})$ and $(\Delta', v_{0}')$ are equivalent under the action of $\SO_{0}(2,3)$, hence the map $\tilde{\alpha}$ is well-defined.
Moreover, if we change $\sigma$ by a pre-composing with the generator of the $\Z_{n+4}$-action $T(z)=\zeta_{n+4}^{-1}z$, then $\sigma'=\sigma \circ T$ is a maximal embedding with polynomial growth $T_{*}q$. Its boundary at infinity remains $\Delta$, but the limit point of the ray $\sigma'(e^{i\theta_{0}}t)$ changes to $v_{0}'$, which coincides with the limit point of the ray $\sigma(e^{i(\theta_{0}+2\pi/(n+4))}t)$. By the description of the limit points along rays, given in Proposition \ref{prop:boundary1}, we find $v_{0}'=v_{1}$, hence the map $\tilde{\alpha}$ is $\Z_{n+4}$-equivariant.
Notice, moreover, that if $T^{j}_{*}(q)=q$ for some $j=1, \dots, n+4$, then $\sigma$ and $\sigma'=\sigma \circ T^{j}$ are maximal embedding of $\C$ into $\h^{2,2}$ with the same embedding data, hence they differ by post-composition with an isometry of $\h^{2,2}$. Therefore, $\tilde{\alpha}(q)=\tilde{\alpha}(T^{j}_{*}(q))$.

\begin{prop}\label{prop:continuity}The map $\tilde{\alpha}$ is continuous. 
\end{prop}
\begin{proof}
Let $q_{n}$ be a sequence of monic and centered polynomial quartic differentials converging to $q$. Let $(\Delta_{n}, v_{n})$ and $(\Delta, v)$ be marked future-directed, negative, light-like polygons representing $\alpha(q_{n})$ and $\alpha(q)$, respectively. \\
\indent We need to show that $(\Delta_{n}, v_{n})$ converges to $(\Delta, v)$ up to the conformal action of $\SO_{0}(2,3)$. We claim first that the maximal surfaces $\Sigma_{n}=\sigma_{n}(\C)$ in $\h^{2,2}$ associated to $q_{n}$ converge to the maximal surface $\Sigma=\sigma(\C)$ associated to $q$, up to isometries. In fact, since $q_{n}$ is convergent, the supersolution estimates of Lemma \ref{lm:supersolution} and standard Schauder estimates give a uniform bound on the $C^{1,\alpha}$ norm of the functions $(u_{1})_{n}$ and $(u_{2})_{n}$, solutions to Equation (\ref{eq:system}) on compact sets. Hence they weakly converge to a weak solution $(u_{1},u_{2})$ of  
\[
	\begin{cases}
	\Delta u_{1}=e^{u_{1}-u_{2}}-e^{-2u_{1}}|q|^{2}\\
	\Delta u_{2}=e^{2u_{2}}-e^{u_{1}-u_{2}}
	\end{cases} \ .
\]
By elliptic regularity, the limits are strong solutions, and by uniqueness they must coincide with those found in Section \ref{sec:existence}. Therefore, $(u_{1})_{n}$ and $(u_{2})_{n}$ actually converge smoothly on compact sets to $u_{1}$ and $u_{2}$. Recalling that the induced metric on the maximal surface $\Sigma$ is given by $4e^{u_{1}-u_{2}}|dz|^{2}$ and the second fundamental form only depends on $q$, we deduce (cf. Proposition~\ref{prop: qd and metric enough for maximal surface}) that the embedding data of $\Sigma_{n}$ converge to the embedding data of $\Sigma$, thus $\Sigma_{n}$ converges to $\Sigma$ up to global isometries of $\h^{2,2}$.\\
\indent In particular, $\Delta_{n}$ converges to $\Delta$. Moreover, since $\sigma_{n}$ converges to $\sigma$ smoothly on compact sets, and each $\sigma_n$ has image with polygonal boundary, the limit points of the rays $\sigma_{n}(e^{i\theta_{0}}t)$ converge to the limit point of the ray $\sigma(e^{i\theta_{0}}t)$, hence $v_{n}$ converges to $v$.
\end{proof}

\subsection{Properness}
Let $[q_{i}]$ be a sequence of polynomial quartic differentials of degree $n$ that leaves every compact set in the moduli space $\mathcal{MQ}_{n}$. We consider a representative $q_{i} \in [q_{i}]$ that has a root at the origin. Therefore, we can write
\[
    q_{i}=q_{i}(z)dz^{4}=(z^{n}+a_{n-1,i}z^{n-1}+\dots+a_{1,i}z)dz^{4}
\]
for some $a_{j,i} \in \C$, and we must necessarily have that $|a_{j,i}| \to +\infty$ as $i \to +\infty$, for some $j=1, \dots, n-1$, up to subsequences. The idea now is to re-scale the variable $z$ appropriately, so that $q_{i}$ converges to a polynomial quartic differential of lower degree.

\begin{lemma}\label{lm:sequence}There exists a unique sequence of complex numbers $\lambda_{i}$ such that 
\[
    q_{i}(\lambda^{-1}_{i}w)\lambda_{i}^{-4}dw^{4}
\]
converges, up to subseqeunces, to a monic non-constant polynomial quartic differential $q_{\infty}$ of lower degree.
\end{lemma}
\begin{proof}We describe an algorithm to find such a sequence $\lambda_{i}$. Let $j$ be the largest index such that $|a_{j,i}|$ diverges as $i \to +\infty$ and define $\lambda_{i}$ so that $\lambda_{i}^{-j-4}a_{j,i}=1$. Let us then consider the index $j-1$. Two things can happen:
\begin{enumerate}
    \item if $\lambda_{i}^{-j-3}a_{j-1,i}$ is uniformly bounded, we keep the same sequence $\lambda_{i}$ and we move to the index $j-2$;
    \item if $\lambda_{i}^{-j-3}a_{j-1,i}$ is unbounded, we replace $\lambda_{i}$ with $\lambda_{i}'$ such that 
    $(\lambda_{i}')^{-j-3}a_{j-1,i}=1$. Notice that we must necessarily have
    \[
        \lim_{i \to \infty} \frac{\lambda_{i}}{\lambda_{i}'}=0 \ ,
    \]
    therefore $\lambda_{i}^{-j-4}a_{j,i}$ (where we use the new $\lambda_i$) tends to $0$. We then move to the index $j-2$.
\end{enumerate}
When we arrive at the index $1$, the sequence $\lambda_{i}$ that we end up with has the property that the product $\lambda_{i}^{-j-4}a_{j,i}$ is uniformly bounded for every $j=1, \dots, n-1$, and every subsequential limit of $q_{i}(\lambda_{i}^{-1}w)\lambda_{i}^{-4}dw^{4}$ is monic, non-constant and of degree strictly less than $n$. 
\end{proof}

We will say that $q_{\infty}$ is a re-scaled limit of $q_{i}$. We will show the following:
\begin{prop}\label{prop:behaviour_boundary} Let $[q_{i}] \in \mathcal{MQ}_{n}$ be a diverging sequence. Assume that its re-scaled limit $q_{\infty}$ has degree $1\leq m\leq n-1$. Let $\Delta_{i}$ be the boundary at infinity of the maximal surface $\Sigma_{i}$ corresponding to $[q_{i}]$. Then there is a sequence $A_{i} \in \SO_{0}(2,3)$ such that $A_{i}\Delta_{i}$ converges to a future-directed, negative, light-like polygon in $\Ein^{1,2}$ with $m+4$ vertices.
\end{prop}

In particular, we deduce 
\begin{cor}\label{cor:properness} The map $\alpha$ is proper.
\end{cor}
\begin{proof}We have to show that, if $[q_{i}]\in \mathcal{MQ}_{n}$ is a diverging sequence of polynomial quartic differentials of degree $n$, then $\alpha([q_{i}])$ is a diverging sequence in $\mathcal{MLP}^{-}_{n+4}$. Let $\Delta_{i}$ be polygons representing $\alpha([q_{i}])$. By Proposition \ref{prop:behaviour_boundary}, there is a sequence $A_{i} \in \SO_{0}(2,3)$ such that $A_{i}\Delta_{i}$ converges to a future-directed negative light-like polygon in $\Ein^{1,2}$ with fewer vertices that is not a quadrilateral. The following lemma shows that $[\Delta_{i}]$ cannot be contained in a compact set of $\mathcal{MLP}^{-}_{n+4}$.
\end{proof}

\begin{lemma}Let $[\Delta_{n}]$ be a sequence of equivalence classes of future-directed negative light-like polygons in $\Ein^{1,2}$ with $k$-vertices converging to $[\Delta] \in \mathcal{MLP}^{-}_{k}$. Let $B_{n}$ be any sequence in $\SO_{0}(2,3)$. Then, the only light-like polygons that can appear as limits of subsequences of $B_{n}\Delta_{n}$ are
\begin{itemize}
	\item either equivalent to $\Delta$ in the moduli space $\mathcal{MLP}^{-}_{k}$;
	\item or quadrilaterals.
\end{itemize}
\end{lemma}
\begin{proof}For this proof, it is convenient to use the quadratic form in $\R^{5}$ with signature $(2,3)$ given by
\[
	\langle x ,x \rangle=2x_{1}x_{5}+2x_{2}x_{4}-x_{3}^{2} \ 
\]
with respect to the canonical basis $\{e_{1}, \dots, e_{5}\}$ of $\R^{5}$. We denote by $A$ the group of diagonal matrices in $\SO_{0}(2,3)$ and with $A^{+}\subset A$ the semigroup of diagonal matrices with eigenvalues $\lambda_{1}\geq \lambda_{2} \geq 1 \geq \frac{1}{\lambda_{2}} \geq \frac{1}{\lambda_{1}}$ in this order. This choice induces a KAK-decomposition, where $K$ is a maximal compact subgroup of $\SO_{0}(2,3)$. We can thus write $B_{n}=K_{n}'D_{n}K_{n}$ with $D_{n} \in A^{+}$ and $K_{n}, K_{n}' \in K$. Up to subsequences, we can assume that $K_{n}$ and $K_{n}'$ converge to $K$ and $K'$ respectively. Moreover, we can assume that the vertices of $\Delta_{n}$ converge to the vertices of $\Delta$. We will see that the subsequential limits of $B_{n}\Delta_{n}$ depend on the behaviour of the eigenvalues of $D_{n}$. \\
\indent Let us first assume that $\lambda_{1,n}$ is uniformly bounded. Then, up to subsequences, we can assume that $D_{n}$ converges to a diagonal matrix $D$ so that the isometries $B_{n}$ converge to $B=K'DK$.  It is clear then that $B_{n}\Delta_{n}$ converges to $B\Delta$. This, in particular, means that acting on $\Delta_{n}$ by a converging sequence of isometries does not change the equivalence class in $\mathcal{MLP}^{-}_{k}$ of the limit. \\
\indent Let us now consider the case of $\lambda_{2,n}$ unbounded and choose a subsequence so that $\lambda_{2,n}\to +\infty$ as $n\to +\infty$. By the previous remark, it is sufficient to understand the behaviour of $D_{n}$ on the polygons $\Delta_{n}'=K_{n}\Delta_{n}$ which converge to $\Delta'=K\Delta$. Let us choose vectors $v_{i,n}$ and $v_{i}$ for $i=1, \dots, k$ in $\R^{5}$ such that each $v_{i,n}$ projects to a vertex $p_{i,n} \in \Delta_{n}'$, the vector $v_{i}$ projects to the vertex $p_{i} \in \Delta'$ and $v_{i,n}$ converges to $v_{i}$ as $n \to +\infty$. This means that if we denote by $a_{j,i,n}$ (resp. $a_{j,i}$) the component of $v_{i,n}$ (resp. $v_{i}$) along $e_{j}$, we have $\lim_{n \to +\infty}a_{j,i,n}=a_{j,i}$ for every $j=1, \dots, 5$. We then note the following:
\begin{enumerate}[(i)]
	\item if $\lambda_{1,n}a_{1,i,n}$ and $\lambda_{2,n}a_{2,i,n}$ are not both uniformly bounded, then $D_{n}v_{i,n}$ limits to a point on the photon $\Pp(\Span(e_{1},e_{2}))$;
	\item if both $\lambda_{1,n}a_{1,i,n}$ and $\lambda_{2,n}a_{2,i,n}$ are bounded, we denote by $q_{i}$ the projective limit of $D_{n}v_{i,n}$.
\end{enumerate}
We notice that $(ii)$ can only happen for those indices $1\leq i \leq k$ such that $a_{1,i}=a_{2,i}=0$. Since the vectors $v_{i}$ project to vertices of a negative light-like polygon, they are isotropic, thus necessarily $a_{3,i}=0$ as well. This implies that such vectors $v_{i}$ are orthogonal to each other, so there can be at most two of them in $\Delta'$. This shows that, the only light-like polygon that can be the limit of $\Delta_{n}'$ is a quadrilateral. \\
\indent The argument is similar if $\lambda_{2,n}$ is bounded but $\lambda_{1,n}$ tends to $+\infty$. In this case, up to subsequences, all vertices $[v_{i,n}]$ that do not converge to $[e_{5}]$ (so all vertices, except at most one) necessarily limit to a point in the light-cone of $e_{1}$. The only light-like polygon with this configuration of vertices is a quadrilateral.
\end{proof}

\begin{proof}[Proof of Proposition \ref{prop:behaviour_boundary}] Consider the change of coordinates on $\C$ given by $w=\lambda_{i}z$, where $\lambda_{i}$ is the sequence found in Lemma \ref{lm:sequence}. In the $w$-plane, the quartic differential can be written as
\[
    q_{i}=\hat{q}_{i}(w)dw^{4} \ \ \ \text{where} \ \ \  \hat{q}_{i}(w)=\lambda_{i}^{-4}q_{i}(\lambda_{i}^{-1}w) ,
\]
hence it subconverges uniformly on compact sets to some $q_{\infty}$ of degree $1\leq m\leq n-1$. Recall that the induced metric on the maximal surface $\Sigma_{i}$ is given by $4e^{u_{1,i}-u_{2,i}}|dz|^{2}$ where $u_{1,i}$ and $u_{2,i}$ are the solution of the system of PDE
\[
	\begin{cases}
	\Delta u_{1,i}(z)=e^{u_{1,i}(z)-u_{2,i}(z)}-e^{-2u_{1,i}(z)}|q_{i}(z)|^{2} \\
	\Delta u_{2,i}(z)=e^{2u_{2,i}(z)}-e^{u_{1,i}(z)-u_{2,i}(z)}     \ .
	\end{cases}
\]
Changing to the $w$-coordinates, the functions
\[
    v_{1,i}(w)=u_{1,i}(\lambda_{i}^{-1}w)-3\log(|\lambda_{i}|) \ \ \ \text{and} \ \ \
    v_{2,i}(w)=u_{2,i}(\lambda_{i}^{-1}w)-\log(|\lambda_{i}|)
\]
satisfy the differential equations
\begin{equation}\label{eq:w-plane}
	\begin{cases}
	\Delta v_{1,i}(w)=e^{v_{1,i}(w)-v_{2,i}(w)}-e^{-2v_{1,i}(w)}|\hat{q}_{i}(w)|^{2} \\
	\Delta v_{2,i}(w)=e^{2v_{2,i}(w)}-e^{v_{1,i}(w)-v_{2,i}(w)}     \ .
	\end{cases}
\end{equation}
Because the coefficients of polynomial $\hat{q}_{i}(w)$ converge, the sub-solution and super-solutions found in  Section \ref{sec:existence} show that 
\begin{align*}
	\frac{3}{8}\log(|\hat{q}_{i}(w)|) &\leq v_{1,i}(w) \leq \frac{3}{8}\log(|\hat{q}_{i}(w)|+C) \\
	\frac{1}{8}\log(|\hat{q}_{i}(w)|) &\leq v_{2,i}(w) \leq \frac{1}{8}\log(|\hat{q}_{i}(w)|+3C) \ ,
\end{align*}
hence $v_{1,i}$ and $v_{2,i}$ are uniformly bounded on compact sets, and from Equation (\ref{eq:w-plane}) we also have a uniform bound on their laplacian. Therefore, standard elliptic theory tells us that $v_{1,i}$ and $v_{2,i}$ converge smoothly to solutions $v_{1,\infty}$ and $v_{2,\infty}$ of the limiting system of PDE
\[
    \begin{cases}
	\Delta v_{1,\infty}(w)=e^{v_{1,\infty}(w)-v_{2,\infty}(w)}-e^{-2v_{1,\infty}(w)}|q_{\infty}(w)|^{2} \\
	\Delta v_{2,\infty}(w)=e^{2v_{2,\infty}(w)}-e^{v_{1,\infty}(w)-v_{2,\infty}(w)}     \ .
	\end{cases}
\]
On the other hand, by uniqueness (Remark \ref{rmk:uniqueness}) of the solution to (\ref{eq:system}), the pair $(v_{1, \infty}, v_{2, \infty})$ defines the embedding data of a maximal surface $\Sigma_{\infty}$ in $\h^{2,2}$ with polynomial quartic differential $q_{\infty}$, and hence polygonal boundary with $\deg(q_{\infty})+4$ vertices. It follows that the sequence of maximal surfaces $\Sigma_{i}$ converges smoothly on compact sets to $\Sigma_{\infty}$, up to global isometries of $\h^{2,2}$. In particular, the boundary at infinity of $\Sigma_{i}$ converges to the boundary at infinity of $\Sigma_{\infty}$, which is a light-like polygon with $m+4$ vertices, up to the conformal action of $\SO_{0}(2,3)$. 
\end{proof}

\subsection{Injectivity} The following lemma is crucial in order to prove the injectivity of the map $\tilde{\alpha}$:

\begin{lemma}\label{lm:inj} Let $\Gamma \subset \Ein^{1,2}$ be the graph of a $1$-Lipschitz map. If there exists a complete maximal surface $\Sigma \subset \h^{2,2}$ spanning $\Gamma$, then it is unique. 
\end{lemma}

The proof will be an adaptation "at infinity" of the argument used in \cite{CTT} in the case of boundary curves invariant by a cocompact group, which consisted of an appplication of the maximum principle to a carefully chosen function defined on $\widehat{\Sigma}\times \widehat{\Sigma}'$. In our non-compact context, we will need the following version of the maximum principle:

\begin{teo}[\cite{Omori}]\label{thm:Omori} Let $M$ be a complete Riemannian manifold with sectional curvature bounded below. If $g$ is a smooth function on $M$ with $\sup(g)<+\infty$, then there exists a sequence of points $x_{k} \in M$ such that
\[
	\lim_{k\to \infty} g(x_{k})=\sup(g) \ \ \ \ \ |\mathrm{grad}(g)_{x_{k}}|\leq \frac{1}{k} \ \ \ \ \ \mathrm{Hess}(g)_{x_{k}}(w,w) \leq \frac{\|w\|^{2}}{k}  \ \forall  w \in T_{x_{k}}M
\]
\end{teo}

\begin{proof}[Proof of Lemma \ref{lm:inj}] Suppose, by contradiction, that there exists another complete maximal surface $\Sigma'$ with boundary at infinity $\Gamma$. We choose their lifts $\widehat{\Sigma}$ and $\widehat{\Sigma}'$ to $\widehat{\h}^{2,2}$ in such a way that they share the same boundary at infinity. As a consequence, the function
\begin{align*}
	B: \widehat{\Sigma} \times \widehat{\Sigma}' &\rightarrow \R \\
			(u,v) &\mapsto \langle u, v \rangle 
\end{align*} is always non-positive (\cite[Lemma 3.24]{CTT}). Moreover, if $\widehat{\Sigma}$ and $\widehat{\Sigma}'$ are distinct, then we can find a pair of points $(u_{0},v_{0}) \in \widehat{\Sigma}\times \widehat{\Sigma}'$ such that $B(u_{0},v_{0})>-1$. In particular, $-1<\sup(B)\leq 0$ (\cite[Lemma 3.25]{CTT}). 
Notice that by a general result of Ishihara (\cite{Ishihara}), maximal surfaces in $\h^{2,2}$ have uniformly bounded second fundamental form, thus the Riemannian manifold $M=\widehat{\Sigma}\times \widehat{\Sigma}'$ has bounded sectional curvature. By the Omori maximum principle, we can find a sequence of points $(u_{n}, v_{n})$ such that 
\[
	\lim_{k\to \infty} B(u_{n},v_{n})=\sup(B) \ \ \ \ \ |\mathrm{grad}(B)_{(u_{n},v_{n})}|\leq \frac{1}{n} \ \ \ \ \ \mathrm{Hess}(B)_{(u_{n},v_{n})}(\dot{\gamma}_{n},\dot{\gamma}_{n}) \leq \frac{\|\dot{\gamma}\|^{2}}{n}  
\]
for every geodesic path $\gamma_{n}:[-\epsilon, \epsilon] \rightarrow M$ with $\gamma_{n}(0)=(u_{n},v_{n})$. We will follow the construction of (\cite{CTT}) in order to find a sequence of paths $\gamma_{n}$ which will give a contradiction. \\
The second derivative of $B$ along a geodesic path $\gamma(t)=(u(t),v(t))$ is given by
\begin{align*}
	\mathrm{Hess}(B)(\dot{u},\dot{v})_{(u(0),v(0))}&=\frac{d}{dt^{2}}_{|_{t=0}}B(u(t),v(t))\\
			&=2\langle \dot{u}, \dot {v} \rangle + \| \dot{u}\|^{2}\langle u(0),v(0) \rangle  + \| \dot{v}\|^{2}\langle u(0),v(0) \rangle \\
			&+\langle II(\dot{u},\dot{u}), v(0) \rangle+\langle II'(\dot{v},\dot{v}),u(0)\rangle \ , 
\end{align*}
where $II$ and $II'$ denote the second fundamental form of $\widehat{\Sigma}$ and $\widehat{\Sigma}'$ respectively. Since $\widehat{\Sigma}$ and $\widehat{\Sigma}'$ are maximal surfaces, the quadratic forms $\beta(\dot{u})=\langle II(\dot{u},\dot{u}), v(0)\rangle$ and $\beta'(\dot{v})=\langle II'(\dot{v},\dot{v}),u(0)\rangle$ are traceless. Let $\lambda$ and $\lambda'$ be their positive eigenvalues. \\
Let $(u_{n}, v_{n})$ be the sequence of points given by the Omori maximum principle. We explain how to choose $\gamma_{n}$ assuming that the positive eigenvalue $\lambda_{n}$ of $\beta$ is larger than the positive eigenvalue $\lambda'_{n}$ of $\beta'$ at the point $(u_{n}, v_{n})$; for the other case it will be sufficient to interchange the role of $\dot{u_{n}}$ and $\dot{v_{n}}$ described below. \\
We choose tangent vectors $(\dot{u_{n}}, \dot{v_{n}}) \in T_{(u_{n}, v_{n})}M$ such that 
\[
	\| \dot{u}_{n}\|=1 \ \ \ \ \ \dot{v}_{n}=\frac{p(\dot{u}_{n})}{\|p(\dot{u}_{n})\|}
\]
where $\beta(\dot{u}_{n})=\lambda_{n}$ and $p$ is the orthogonal projection onto $T_{v_{n}}\widehat{\Sigma}'$. 
This choice of $\dot{u}_{n}$ and consequent choice of $\dot{v}_{n}$ will force the terms $\langle II(\dot{u},\dot{u}), v(0) \rangle+\langle II'(\dot{v},\dot{v}),u(0)\rangle \ \geq 0$, as in \cite{CTT}.
Next, since we have an orthogonal decomposition 
\[
	\R^{2,3}=\mathrm{Span}(v_{n}) \perp T_{v_{n}}\widehat{\Sigma}' \perp N_{v_{n}}\widehat{\Sigma}'
\]
we may write $\dot{u}_{n}=k_{n}v_{n}+p(\dot{u}_{n})+w_{n}$, with $w_{n} \in N_{v_{n}}\widehat{\Sigma}'$. Since the normal bundle of a space-like surface in $\h^{2,2}$ is negative definite, we have that $\|w_{n}\| \leq 0$. Hence,
\[
	1=\| \dot{u}_{n}\|^{2}=-k_{n}^{2}+\|p(\dot{u}_{n})\|^{2}+\|w_{n}\| \leq -k_{n}^{2}+\|p(\dot{u}_{n})\|^{2}
\]
which implies that $\|p(\dot{u}_{n})\|^{2}\geq 1+k_{n}^{2}$. Therefore, 
\[
	\langle \dot{u}_{n}, \dot{v}_{n}\rangle=\langle k_{n}v_{n}+p(\dot{u}_{n})+w_{n}, \frac{p(\dot{u}_{n})}{\|p(\dot{u}_{n})\|} \rangle=\|p(\dot{u}_{n})\| \geq \sqrt{1+k_{n}^{2}}.
\]
We notice that $k_{n}$ decays to zero as $n$ goes to infinity because
\[
	|k_{n}|=|\langle \dot{u}_{n}, v_{n} \rangle|=|dB_{(u_{n},v_{n})}(\dot{u}_{n}, 0)|=|g(\mathrm{grad}(B)_{(u_{n},v_{n})}, (\dot{u}_{n},0))| \leq \frac{\|\dot{u}_{n}\|}{n} 
\]
where we denoted with $g$ the Riemannian metric on $M$. \\
The Omori maximum principle then gives that
\[
	\frac{2}{n} \geq \mathrm{Hess}(B)_{(u_{n},v_{n})}(\dot{u}_{n}, \dot{v}_{n}) \geq 2\sqrt{1+k_{n}^{2}}+2B(u_{n},v_{n})
\]
and letting $n$ go to infinity we obtain that
\[
	0\geq \limsup_{n\to +\infty}2\sqrt{1+k_{n}^{2}}+2B(u_{n},v_{n})=2+2\sup(B) \ .
\]
Thus, $\sup(B)\leq -1$, but this contradicts the fact that $-1<\sup(B)\leq 0$. 
\end{proof}

\begin{prop}\label{prop:injectivity} The map $\tilde{\alpha}$ is injective.
\end{prop}
\begin{proof} Let $q,q' \in \mathcal{TQ}_{n}$ be different monic and centered polynomial quartic differentials. If there exists $j \in \{1, \dots, n\}$ such that $q'=T_{*}^{j}q$, where $T(z)=\zeta_{n+4}z$ is a generator of the $\Z_{n+4}$-action, then the equivariance of the map already implies that $\tilde{\alpha}(q)\neq \tilde{\alpha}(q')$, because the marking of the polygon at infinity is changed. Otherwise, suppose by contradiction that $\tilde{\alpha}(q)=\tilde{\alpha}(q')$. Then, we can choose maximal surfaces $\Sigma$ and $\Sigma'$ with polynomial growth $q$ and $q'$ with the same boundary at infinity $\Delta$. By Lemma \ref{lm:inj}, the surfaces $\Sigma$ and $\Sigma'$ must coincide, and, in particular, have the same embedding data. Therefore, there exists a biholomorphism $T'$ of $\C$ such that $T'_{*}q'=q$, but this is impossible because $q$ and $q'$ do not lie in the same $\Z_{n+4}$-orbit and they are both monic and centered. 
\end{proof}

\begin{teo}\label{main thm} The map $\tilde{\alpha}$ induces a homeomorphism 
\[
	\alpha: \mathcal{MQ}_{n} \rightarrow \mathcal{MLP}_{n+4}^{-} \ 
\]
between the moduli space of polynomial quartic differential on the complex plane of degree $n$ and a connected component of the moduli space of future-directed, negative light-like polygons in the Einstein Universe with $n+4$ vertices.
\end{teo}
\begin{proof}The map $\tilde{\alpha}:\mathcal{TQ}_{n} \rightarrow \mathcal{TLP}^{-}_{n+4}$ is continuous and injective by Proposition \ref{prop:injectivity}. It is also proper by Corollary \ref{cor:properness}, hence it is a homeomorphism onto a connected component of $\mathcal{TLP}^{-}_{n+4}$ by the Invariance of the Domain. Since it is $\Z_{n+4}$-equivariant, it decends to a homeomorphism $\alpha:\mathcal{MQ}_{n} \rightarrow \mathcal{MLP}^{-}_{n+4}$ between connected components of the moduli spaces. 
\end{proof}

\section{Estimates along rays}
Let $X=(S,J)$ be a closed Riemann surface and let $q$ be a holomorphic quartic differential on $X$. Recall from Section \ref{sec:background} that, out of these data, one can construct an $\Sp(4,\R)$-Higgs bundle $(\mathcal{E}, \varphi)$ over $X$ in the $\Sp(4,\R)$-Hitchin component where
$\mathcal{E}=K^{\frac{3}{2}}\oplus K^{-\frac{1}{2}}\oplus K^{-\frac{3}{2}} \oplus K^{\frac{1}{2}}$ and 
\[
    \varphi=\begin{pmatrix}
			0 & 0 & q & 0 \\
			0 & 0 & 0 & 1 \\
			0 & 1 & 0 & 0 \\
			1 & 0 & 0 & 0  \\
		\end{pmatrix}  \ .
\]
In this setting, it is well-known that the solution of Hitchin's self-duality equations is unique and diagonal of the form $g=\diag(g_{1}, g_{2}^{-1}, g_{1}^{-1}, g_{2})$. Works of Collier-Li (\cite{Collier-Li}) and Mochizuki (\cite{Mochizuki_harmonicbundles}) describe the asymptotic behaviour of the metric $g_{s}$ along rays of quartic differentials $q_{s}=sq_{0}$ away from the zeros of $q_{0}$. Here, we use the harmonic metric $H$ found in Theorem \ref{thm:existence} in order to construct sub- and supersolutions that will describe the asymptotics of the harmonic metrics $g_{s}$ at and near a zero. We will prove the following result:

\begin{teo} \label{thm:rays} Assume $p \in X$ is a zero of order $k \geq 1$ of the quartic differential $q_{0}$. Let $\sigma$ denote the conformal hyperbolic metric on $X$. Then, there is a sequence of radii $r_{s} \to 0$ such that
\[
    {g_{1,s}^{-1}}_{|_{B(p,r_{s})}}=O(s^{\frac{3}{k+4}}\sigma^{\frac{3}{4}}) \ \ \ \text{and} \ \ \
    {g_{2,s}^{-1}}_{|_{B(p,r_{s})}}=O(s^{\frac{1}{k+4}}\sigma^{\frac{1}{4}})
\]
along the ray $q_{s}=sq_{0}$ as $s\to +\infty$.
\end{teo}

As a geometric corollary to this analytic result, we will find that that the harmonic metrics $g_s$ \enquote{localize} in the sense that the maximal surfaces associated to the quartic differentials $sq_0$, equipped with a basepoint $p$, converge in the Gromov-Hausdorff sense to the polygonal maximal surface associated to the divisor of $q_0$ at $p$.  We describe this more carefully in Corollary~\ref{cor: localization of maximal surfaces along rays} at the end of the section.

To begin the proof of Theorem~\ref{thm:rays}, let us fix local coordinates $(V,z)$ around $p$ such that $q_{0}(z)=z^{k}dz^{4}$. We introduce new coordinates $(V, \zeta_{s})$ defined by
\[
	\zeta_{s}=s^{\frac{1}{k+4}}z
\]
so that
\[
	\hat{q}_{s}(\zeta_{s}):=(\zeta_{s})_{*}q_{s}=\zeta_{s}^{k}d\zeta_{s}^{4} \ .
\]
We notice that if $U=z(V)=\{ z \in \C \ | \ |z|<\epsilon \}$, then $U_{s}=\zeta_{s}(V)=\{ z \in \C \ | \  |z|<\epsilon s^{\frac{1}{k+4}}\}$, thus the sequence $\{ (V, sq_{0}, p)\}_{s \geq 0}$ converges geometrically to $(\C, \zeta^{k}d\zeta^{4}, 0)$. \\

Let $\sigma$ be the hyperbolic metric on $X$ compatible with the complex structure. We denote
\begin{align*}
	\sigma&=\sigma(z)|dz|^{2}=\hat{\sigma}(\zeta_{s})|d\zeta_{s}|^{2} \ \ \ \  \text{where} \ \ \ \hat{\sigma}(\zeta_{s})=s^{\frac{-2}{k+4}}\sigma(z) \\
	\Delta_{\sigma}&=\sigma(z)^{-1}\partial_{z}\partial_{\bar{z}} \\
	\Delta_{\hat{\sigma}}&=\hat{\sigma}(\zeta_{s})^{-1}\partial_{\zeta_{s}}\partial_{\bar{\zeta_{s}}} .
\end{align*}

Let $g=(g_{1}, g_{2}^{-1}, g_{1}^{-1}, g_{2})$ be the solutions of Hitchin's equations on $(\sE,\varphi)$. We define two functions $\psi_{1}, \psi_{2}:X\rightarrow \R$ by the property that
\[
	\frac{1}{g_{1}}=e^{\psi_{1}}\sigma^{\frac{3}{2}} \ \ \ \  \text{and} \ \ \ \ \frac{1}{g_{2}}=e^{\psi_{2}}\sigma^{\frac{1}{2}}  \ .
\]
The pair $(\psi_{1}, \psi_{2})$ is the solution of the system of PDEs, defined on the whole surface,
\begin{equation}\label{eq:PDE1}
	\begin{cases}
	\Delta_{\sigma}\psi_{1}=e^{\psi_{1}-\psi_{2}}-e^{-2\psi_{1}}\frac{|q|^{2}}{\sigma^{4}}+\frac{3}{4}\kappa(\sigma) \\
	\Delta_{\sigma}\psi_{2}=e^{2\psi_{2}}-e^{\psi_{1}-\psi_{2}}+\frac{1}{4}\kappa(\sigma)
	\end{cases}
\end{equation}
where $\kappa(\sigma)$ denotes the Gaussian curvature of $\sigma$. \\

We denote by $\{(\psi_{1}^{s}, \psi_{2}^{s})\}_{s\geq 0}$ the solution to the above system along the ray $\{q_{s}\}_{s\geq 0}$. When studying the equation on $V$, or in general in a neighbourhood of a zero of order $k$ for $q_{0}$, it will be convenient to rescale the background metric $\sigma$ to a metric $\sigma_{s}$ so that 
\[
	\kappa(\sigma_{s})=-s^{\frac{-2}{k+4}} \ .
\]
It is straightforward to verify that the metric $\sigma_{s}=s^{\frac{2}{k+4}}\sigma$ satisfies the above condition and we can write in local coordinates
\[
	\sigma_{s}=\hat{\sigma}_{s}(\zeta_{s})|d\zeta_{s}|^{2} \ \ \ \ \ \text{where} \ \ \  \hat{\sigma}_{s}(\zeta_{s})=s^{\frac{2}{k+4}}\sigma(\zeta_{s}) \ . 
\]
Rewriting Equation (\ref{eq:PDE1}) using the background metric $\sigma_{s}$ in the coordinate $\zeta_{s}$, we obtain
\begin{equation*}\label{eq:PDE2}
	\begin{cases}
	\Delta_{\hat{\sigma}_{s}}\left( \psi_{1}^{s}-\frac{3}{k+4}\log(s) \right) =e^{\psi_{1}^{s}-\psi_{2}^{s}-\frac{2}{k+4}\log(s)}-e^{-2(\psi_{1}^{s}-\frac{3}{k+4}\log(s))}\frac{|\hat{q}_{s}|^{2}}{\hat{\sigma}^{4}s^{\frac{8}{k+4}}}+\frac{3}{4}\frac{\kappa(\sigma)}{s^{\frac{2}{k+4}}} \\ 
	\Delta_{\hat{\sigma}_{s}}\left( \psi_{2}^{s}-\frac{1}{k+4}\log(s) \right)= e^{2(\psi^{s}_{2}-\frac{1}{k+4}\log(s))}-e^{(\psi_{1}^{s}-\frac{3}{k+4}\log(s))}-e^{(\psi_{2}^{s}-\frac{1}{k+4}\log(s))}+\frac{1}{4}\frac{\kappa(\sigma)}{s^{\frac{2}{k+4}}} \ .
	\end{cases}
\end{equation*}
Therefore, we deduce that the functions
\[
	v_{1}^{s}(\zeta_{s})=\psi_{1}^{s}(\zeta_{s})-\frac{3}{k+4}\log(s)  \ \ \ \ \ \ 
	v_{2}^{s}(\zeta_{s})=\psi_{2}^{s}(\zeta_{s})-\frac{1}{k+4}\log(s) 
\]
are solutions of 
\begin{equation}\label{eq:PDE3}
	\begin{cases}
	\Delta_{\hat{\sigma}_{s}}v_{1}^{s}=e^{v_{1}^{s}-v_{2}^{s}}-e^{-2v_{1}^{s}}\frac{|\hat{q}_{s}|^{2}}{\hat{\sigma}_{s}^{4}}+\frac{3}{4}\kappa(\hat{\sigma}_{s}) \\
	\Delta_{\hat{\sigma}_{s}}v_{2}=e^{2v_{2}^{s}}-e^{v^{s}_{1}-v^{s}_{2}}+\frac{1}{4}\kappa(\hat{\sigma}_{s}) \ .
	\end{cases}
\end{equation}

We notice that the coefficients of the above equations converge to the planar Hitchin's equations with polynomial quartic differential $q_{\infty}=\zeta^{k}d\zeta^{4}$ because, as $s$ tends to $+\infty$, we have
\begin{align*}
	|\hat{q}_{s}|^{2}=|\zeta_{s}^{k}|^{2}|d\zeta_{s}^{4}|^{2} &\to |q_{\infty}|^{2} \\ 
	\kappa(\hat{\sigma}_{s}) &\to 0 \\ 
	\hat{\sigma}_{s}(\zeta_{s}) &\to 1 \ .
\end{align*}

\noindent This suggests that $v_{j}^{s}$ restricted on $V$ should converge to the solutions of the planar equations, here making crucial use of the uniqueness result of Remark~\ref{rmk:uniqueness}. We prove this using the sub- and supersolution method, together with the uniqueness result Proposition~\ref{prop:injectivity}. \\

Let us start with the subsolutions. Let $\{U_{i}\}_{i=1}^{N}$ be pairwise disjoint natural coordinate charts such that each of them is centered at a zero of $q_{0}$ of order $k_{i}$. We can assume that each natural coordinate identifies $U_{i}$ with $\{ |z| < \epsilon \} \subset \C$. We define the following functions on $X$:
\[
	w_{1}^{s}(z)=\max \left(w_{1}^{s, U_{i}}(z), \frac{3}{8}\log\left(\frac{|q_{s}|^{2}}{\sigma^{4}}\right)\right) \ \ \ \ \ \
	w_{2}^{s}(z)=\max \left(w_{2}^{s,U_{i}}(z),\frac{1}{8}\log\left(\frac{|q_{s}|^{2}}{\sigma^{4}}\right)\right).                     
\]
Here, for $z \in U_{i}$
\[
	w_{1}^{s,U_{i}}(z)=\frac{3}{k_{i}+4}\log(s)+u_{1,i}(s^{\frac{1}{k_{i}+4}}z)-B_{i} 
\]
\[
    w_{2}^{s,U_{i}}(z)=\frac{1}{k_{i}+4}\log(s)+u_{2,i}(s^{\frac{1}{k_{i}+4}}z)-B_{i} \ \ 
\]
where $B_{i}>0$ needs to be chosen, and $(u_{1,i}, u_{2,i})$ is the solution to the planar Hitchin equations with polynomial quartic differential $\zeta^{k_{i}}d\zeta^{4}$ found in Theorem \ref{thm:existence}.

\begin{lemma}\label{lm:subsolution_welldef}There exists a constant $B_{0}>0$ such that for all $B_{i}>B_{0}$ the functions 
$w_{j}^{s}$ for $j=1,2$ are continuous for every $s$ sufficiently large.
\end{lemma}
\begin{proof} Let us first notice that at a zero of $q_{0}$, the functions $w_{j}^{s}$ are well-defined because each $w_{j}^{s,U_{i}}$ takes a finite value at $z=0$. We need to show that we can choose $B_{i}>0$ so that if $z \in \partial U_{i}$ we have
\[
    w_{1}^{s,U_{i}}(z) \leq \frac{3}{8}\log\left(\frac{|q_{s}|^{2}}{\sigma^{4}}\right) \ \ \ \text{and} \ \ \ 
    w_{2}^{s,U_{i}}(z) \leq \frac{1}{8}\log\left(\frac{|q_{s}|^{2}}{\sigma^{4}}\right)
\]
for every $s$ large enough. We give the details for $w_{1}^{s,U_{i}}$, the other case being analogous. By definition of $w_{1}^{s,U_{i}}$ the above inequality can be re-written as
\begin{equation}\label{eq:estimate_1}
    \frac{3}{k_{i}+4}\log(s)+u_{1,i}(s^{\frac{1}{k_{i}+4}}z)-B_{i}\leq \frac{3}{8}\log\left(\frac{|q_{s}|^{2}}{\sigma^{4}}\right) \ .
\end{equation}
The estimates on $u_{1,i}$ (see Equation (\ref{eq:bound_u1})) tell us that there exists $c_{i}>0$ and $s_{0}>0$ such that for every $s>s_{0}$ and for every $z \in \partial U_{i}$ we have
\begin{align}\begin{split}\label{eq:estimate_2}
 w_{1}^{s,U_{i}}(z)&=\frac{3}{k_{i}+4}\log(s)+u_{1,i}(s^{\frac{1}{k_{i}+4}}z)-B_{i}\\
           &\leq \frac{3}{k_{i}+4}\log(s)+\frac{3}{4}\log(s^{\frac{k_{i}}{k_{i}+4}}|z|^{k_{i}})+c_{i}-B_{i}\\
            &=\frac{3}{4}\log(s)+\frac{3k_{i}}{4}\log|\epsilon|+c_{i}-B_{i}\ .
\end{split}
\end{align}
On the other hand, if $z \in \partial U_{i}$, the right-hand side of Equation (\ref{eq:estimate_1}) becomes
\begin{align}\begin{split}\label{eq:estimate_3}
    \frac{3}{8}\log\left(\frac{|q_{s}|^{2}}{\sigma^{4}}\right)&=\frac{3}{8}\log\left(\frac{s^{2}|z|^{2k_{i}}}{\sigma^{4}(z)}\right)\\
    &\geq\frac{3}{4}\log(s)+\frac{3k_{i}}{4}\log(|\epsilon|)+\frac{3}{8}\min_{|z|=\epsilon}\log\left(\frac{1}{\sigma^{4}(z)}\right) \ .
\end{split}
\end{align}
Comparing Equation (\ref{eq:estimate_2}) and Equation (\ref{eq:estimate_3}), we observe that the inequality in (\ref{eq:estimate_1}) is satisfied for $s>s_{0}$ if we choose a positive $B_{i}>B_{0}$ with
\[
    B_{0}\geq c_{i}-\frac{3}{8}\min_{|z|=\epsilon}\log\left(\frac{1}{\sigma^{4}(z)}\right) \ .
\]
A similar inequality on $B_{0}$ is obtained when studying the function $w_{2}^{s,U_{i}}$, and it is then sufficient to take $B_{0}$ large enough to satisfy all the inequalities found in this way for each $U_{i}$.
\end{proof}

\begin{lemma}\label{lm:subsolution_rays} There exist constants $B_{i}>0$ such that the functions $w_{j}^{s}$ are subsolutions to Equation (\ref{eq:PDE1}) for $s$ large enough.
\end{lemma}
\begin{proof} Let us first show that the functions $w_{j}^{s,U_{i}}$ for $j=1,2$ are subsolutions on $U_{i}$. By the discussion at the beginning of this subsection, it is actually easier to consider the functions
\[
    w_{1}^{s,U_{i}}-\frac{3}{k_{i}+4}\log(s) \ \ \ \text{and} \ \ \ w_{2}^{s,U_{i}}-\frac{1}{k_{i}+4}\log(s) \ ,
\]
and show that they are subsolutions for Equation (\ref{eq:PDE3}). Let $F_{1}= F_{1}(v_1^s, v_2^s)$ and $F_{2}= F_{2}(v_1^s, v_2^s)$ denote the functions on the right-hand side of Equation (\ref{eq:PDE3}). We have to show that we can choose $B_{i}>B_{0}$ so that
\[
	\begin{cases}
	\Delta_{\hat{\sigma}_{s}}\left(w_{1}^{s, U_{i}}-\frac{3}{k_{i}+4}\log(s)\right)\geq F_{1}\left(w_{1}^{s, U_{i}}-\frac{3}{k_{i}+4}\log(s), w_{2}^{s, U_{i}}-\frac{1}{k_{i}+4}\log(s)\right) \\
	\Delta_{\hat{\sigma}_{s}}\left(w_{2}^{s, U_{i}}-\frac{1}{k_{i}+4}\log(s)\right)\geq F_{2}\left(w_{1}^{s, U_{i}}-\frac{3}{k_{i}+4}\log(s), w_{2}^{s, U_{i}}-\frac{1}{k_{i}+4}\log(s)\right) \ .
	\end{cases}
\]
Recalling that the pair $u_{1,i}$ and $u_{2,i}$ are the solution to the planar Hitchin equations with polynomial quartic differential $\zeta^{k_{i}}d\zeta^{4}$, the above system is equivalent to
\[
	\begin{cases}
	e^{u_{1,i}-u_{2,i}}-e^{-2u_{1,i}}e^{2B_{i}}\frac{|\zeta^{k_{i}}|^{2}}{\hat{\sigma}_{s}^{4}}+\frac{3}{4}\kappa(\sigma_{s})\leq e^{u_{1,i}-u_{2,i}}-e^{-2u_{1,i}}|\zeta^{k_{i}}|^{2} \\
	e^{2u_{2,i}}e^{-2B_{i}}-e^{u_{1,i}-u_{2,i}}+\frac{1}{4}\kappa(\sigma_{s})\leq e^{2u_{2,i}}-e^{u_{1,i}-u_{2,i}} 
	\end{cases}
\]
which gives
\[
	\begin{cases}
	-e^{-2u_{1,i}}|\zeta^{k_{i}}|^{2}(\frac{e^{2B_{i}}}{\hat{\sigma}_{s}^{4}}-1)+\frac{3}{4}\kappa(\sigma_{s})\leq 0\\
	e^{2u_{2,i}}(e^{-2B_{i}}-1)+\frac{1}{4}\kappa(\sigma_{s})\leq 0 \ .
	\end{cases}
\]
The above conditions are satisfied provided $B_{i}>0$ and $s>0$ are large enough because $\hat{\sigma}_{s} \to 1$ and $\kappa(\sigma_{s}) \to 0$ as $s\to +\infty $, and of course, the function $u_{j,i}$ are bounded on $U_i$. Clearly, if necessary, we can increase $B_{i}$ so that $B_{i}>B_{0}$. \\
It is straighforward to check that the pair $\left(\frac{3}{8}\log\left(\frac{|q_{s}|^{2}}{\sigma^{4}}\right), \frac{1}{8}\log\left(\frac{|q_{s}|^{2}}{\sigma^{4}}\right) \right)$ is a solution outside the zeros of $q_{0}$, hence, in particular, it is a subsolution. The functions $(w_{1}^{s}, w_{2}^{s})$ are then subsolutions as well because of the monotonicity property of the functions $G_{i}$ in the right hand side of Equation (\ref{eq:PDE1}). Assume, for instance, that, at a point $z \in U_{i}$, we have
\[
    w_{1}^{s}(z)=w_{1}^{s,U_{i}}(z) \ \ \ \text{and} \ \ \ w_{2}^{s}(z)=\frac{1}{8}\log\left(\frac{|q_{s}(z)|^{2}}{\sigma^{4}}\right)
\]
Then we have
\[
 \nabla_{\sigma}w_{1}^{s}(z)=\nabla_{\sigma}w_{1}^{s,U_{i}}(z) \geq G_{1}(w_{1}^{s,U_{i}}(z), w_{2}^{s,U_{i}}(z)) \geq G_{1}\left(w_{1}^{s,U_{i}}(z),\frac{1}{8}\log\left(\frac{|q_{s}(z)|^{2}}{\sigma^{4}}\right)\right)
\]
where the first inequality comes from the fact that the pair $(w_{1}^{s,U_{i}}, w_{2}^{s, U})$ is a subsolution and the second inequality follows from $G_{1}$ being decreasing in the second variable. Similarly,
\begin{align*}
  \nabla_{\sigma}w_{2}^{s}(z)&=\frac{1}{8}\nabla_{\sigma}\log\left(\frac{|q_{s}(z)|^{2}}{\sigma^{4}}\right) \geq G_{2}\left(\frac{3}{8}\log\left(\frac{|q_{s}(z)|^{2}}{\sigma^{4}}\right), \frac{1}{8}\log\left(\frac{|q_{s}(z)|^{2}}{\sigma^{4}}\right)\right)\\  
    &\geq G_{2}\left(w_{1}^{s,U_{i}}, \frac{1}{8}\log\left(\frac{|q_{s}(z)|^{2}}{\sigma^{4}}\right)\right) \ .
\end{align*}
The other cases can be proved analogously.
\end{proof}

Let us now move on to the supersolutions.

\begin{lemma}\label{lm:constant_super} There exist positive constants $C_{1}(s)$ and $C_{2}(s)$ such that $(C_{1}(s), C_{2}(s))$ is a supersolution of Equation (\ref{eq:PDE1}) along the ray $q_{s}=sq_{0}$ with respect to the background metric $\sigma_{s}=s^{\frac{2}{k+4}}\sigma$. Moreover, we can choose them so that, as $s \to +\infty$, we have
\[
    C_{1}(s)-3C_{2}(s)=O(s^{-\frac{2}{k+4}}) \ \ \ \text{and} \ \ \ C_{2}(s)-\frac{k}{4(k+4)}\log(s)=o(1) \ .
\]
\end{lemma}
\begin{proof} The pair $(C_{1}(s), C_{2}(s))$ is a supersolution of Equation (\ref{eq:PDE1}) with respect to the background metric $\sigma_{s}$ if it satisfies
\begin{equation}\label{eq:estimate_constant}
    \begin{cases}
     e^{C_{1}(s)-C_{2}(s)}-e^{-2C_{1}(s)}\frac{|q_{s}|^{2}}{\sigma_{s}^{4}}+\frac{3}{4}\kappa(\sigma_{s}) \geq 0 \\
     e^{2C(s)}-e^{C_{1}(s)-C_{2}(s)}+\frac{1}{4}\kappa(\sigma_{s}) \geq 0
    \end{cases}
\end{equation}
Dividing the second equation by $e^{2C(s)}$, we get
\[
  1-e^{C_{1}(s)-3C_{2}(s)}-\frac{1}{4}s^{\frac{-2}{k+4}} \geq 0 \ .
\]
This is satisfied for $s$ sufficiently large if we choose for example
\[
    C_{1}(s)=3C_{2}(s)-\frac{1}{2}s^{\frac{-2}{k+4}} \ .
\]
Let us now verify that this choice makes also the first inequality in (\ref{eq:estimate_constant}) true for $s$ large enough. We can estimate
\begin{align*}
    e^{C_{1}(s)-C_{2}(s)}&-e^{-2C_{1}(s)}\frac{|q_{s}|^{2}}{\sigma_{s}^{4}}-\frac{3}{4}s^{\frac{-2}{k+4}}\\ &\geq e^{C_{1}(s)-C_{2}(s)}-e^{-2C_{1}(s)}\max_{X}\left(\frac{|q_{s}|^{2}}{\sigma_{s}^{4}}\right)-\frac{3}{4}s^{\frac{-2}{k+4}}\\ 
    &= e^{2C_{2}(s)}e^{\frac{s^{\frac{-2}{k+4}}}{2}}-e^{-6C_{2}(s)}e^{\frac{s^{\frac{-2}{k+4}}}{2}}O(s^{\frac{2k}{k+4}})-\frac{3}{4}s^{\frac{-2}{k+4}}\\
    &= e^{2C_{2}(s)}\left(e^{\frac{s^{\frac{-2}{k+4}}}{2}}-e^{-8C_{2}(s)}e^{\frac{s^{\frac{-2}{k+4}}}{2}}O(s^{\frac{2k}{k+4}})-\frac{3}{4}e^{-2C_{2}(s)}s^{\frac{-2}{k+4}}\right)
\end{align*}
and we can simply define $C_{2}(s)$ by the property that
\[
    e^{\frac{s^{\frac{-2}{k+4}}}{2}}-e^{-8C_{2}(s)}e^{\frac{s^{\frac{-2}{k+4}}}{2}}O(s^{\frac{2k}{k+4}})-\frac{3}{4}e^{-2C_{2}(s)}s^{\frac{-2}{k+4}}=0
\]
Notice that this implies that necessarily $C_{2}(s)$ diverges as $s \to +\infty$, precisely,
\[
    C_{2}(s)-\frac{k}{4(k+4)}\log(s) \to 0 \ .
\]
\end{proof}

\begin{cor}The constants
\[
    \widehat{C_{1}}(s)=C_{1}(s)+\frac{3}{k+4}\log(s) \ \ \ \text{and} \ \ \ \widehat{C_{2}}(s)=C_{2}(s)+\frac{1}{k+4}\log(s)
\]
are supersolutions for Equation (\ref{eq:PDE1}) with respect to the background metric $\sigma$.
\end{cor}

Recall that $(V, z)$ is a coordinate chart centered at the zero $p$ of $q_{0}$ of order $k$.
We then improve these supersolutions $(\widehat{C_{1}}(s), \widehat{C_{2}}(s))$ on $V$ using the solutions to the planar Hitchin's equations.  We define the functions
\[
	W_{1}^{s}=\min(W_{1}^{s,V}, \widehat{C_{1}}(s)) \ \ \ \ \ 
	W_{2}^{s}=\min(W_{2}^{s,V}, \widehat{C_{2}}(s))                      
\]
with
\[
	W_{1}^{s,V}(z)=\frac{3}{k+4}\log(s)+u_{1}(s^{\frac{1}{k+4}}z)+2A 
\]
\[
    W_{2}^{s,V}(z)=\frac{1}{k+4}\log(s)+u_{2}(s^{\frac{1}{k+4}}z)+A \ ,
\]
where $A$ is a positive constant to be chosen later and $(u_{1},u_{2})$ is the solution to the planar Hitchin's equations with quartic differential $\zeta^{k}d\zeta^{4}$. 

\begin{lemma}\label{lm:supersolution_welldef}There exists a constant $A_{0}>0$ such that for every $A>A_{0}$ the functions $W_{j}^{s}$ are continuous for every $s$ large enough.
\end{lemma}
\begin{proof} The argument is similar to that in the proof of Lemma \ref{lm:subsolution_welldef}. 
We first notice that at the point $p$, we will always have $W_{j}^{s}(p)=W_{j}^{s,V}(p)$ for $s$ large enough because $u_{j}$ is uniformly bounded in $s$ at $z=0$, whereas $C_{j}(s)$ diverge as $s \to +\infty$. It is then sufficient to check that we can find $A>0$ such that for $s$ sufficiently large and for all $z \in \partial V$ the following inequalities hold 
\begin{equation}\label{eq:estimate_4}
    W_{1}^{s,V}(z) \geq \widehat{C_{1}}(s) \ \ \ \text{and} \ \ \ W_{2}^{s,V}(z) \geq \widehat{C_{2}}(s) \ .
\end{equation}
Let us consider the first condition.
Now, of course,
\[
W_{1}^{s,V}=\frac{3}{k+4}\log(s)+u_{1}(s^{\frac{1}{k+4}}z)+2A
\]
 %By definition of $W_{1}^{s,V}$ the above condition can be re-written as
%\begin{equation}\label{eq:estimate_4}
%    u_{1}(s^{\frac{1}{k+4}}z)+2A\geq C_{1}(s)
%\end{equation}
and so the estimates on $u_{1}$ (see Equation (\ref{eq:bound_u1})) tell us that there exists $c_{1}>0$ and $s_{0}>0$ such that for every $s>s_{0}$ and for every $z \in \partial V$ we have
\begin{equation}\label{eq:estimate_5}
    W_{1}^{s,V} \geq \frac{3}{k+4}\log(s)+\frac{3k}{4(k+4)}\log(s)+\frac{3k}{4}\log(|\epsilon|)-c_{1}+2A \ .
\end{equation}
On the other hand, by Lemma \ref{lm:constant_super}, there exists a constant $d_{1}$ such that for $s \geq s_{0}$ we have
\begin{equation}\label{eq:estimate_6}
        \widehat{{C}_{1}}(s)=C_{1}(s)+\frac{3}{k+4}\log(s)\leq \frac{3k}{4(k+4)}\log(s)+d_{1}+\frac{3}{k+4}\log(s) \ .
\end{equation}
Comparing Equations (\ref{eq:estimate_5}) and (\ref{eq:estimate_6}), the inequality in (\ref{eq:estimate_4}) is satisfied if we choose
\[
    A>\frac{1}{2}\left(d_{1}-c_{1}-\frac{3k}{4}\log(|\epsilon|)\right) \ .
\]
The same argument applied to $W_{2}^{s,V}$ gives \[
        A>d_{2}-c_{2}-\frac{k}{4}\log(|\epsilon|) \ .
\]
The proof follows by choosing 
\[
    A_{0}=\max\left(\frac{1}{2}\left(d_{1}-c_{1}-\frac{3k}{4}\log(|\epsilon|)\right), d_{2}-c_{2}-\frac{k}{4}\log(|\epsilon|), 0 \right) \ .
\]
\end{proof}

\begin{lemma}\label{lm:supersolution_rays} There exists a constant $A>A_{0}$ such that the functions $W_{j}^{s}$ are supersolutions to Equation (\ref{eq:PDE1}) for $s$ large enough. 
\end{lemma}
\begin{proof} It is sufficient to prove that $W_{j}^{s,V}$ for $j=1,2$ are supersolutions on $V$, because then the same argument as Lemma \ref{lm:subsolution_rays} applies. By the discussion at the beginning of this subsection, it is actually easier to consider the functions
\[
    W_{1}^{s,V}-\frac{3}{k_{i}+4}\log(s) \ \ \ \text{and} \ \ \ W_{2}^{s,V}-\frac{1}{k_{i}+4}\log(s) \ ,
\]
and show that they are supersolutions to Equation (\ref{eq:PDE3}). Let $F_{1}$ and $F_{2}$ denote the functions on the right-hand side of Equation (\ref{eq:PDE3}). We have to show that we can choose $A>A_{0}$ so that
\[
	\begin{cases}
	\Delta_{\hat{\sigma}_{s}}\left(W_{1}^{s, V}-\frac{3}{k+4}\log(s)\right)\leq F_{1}\left(W_{1}^{s, V}-\frac{3}{k+4}\log(s), W_{2}^{s, V}-\frac{1}{k+4}\log(s)\right) \\
	\Delta_{\hat{\sigma}_{s}}\left(W_{2}^{s, V}-\frac{1}{k+4}\log(s)\right)\leq F_{2}\left(W_{1}^{s, V}-\frac{3}{4}\log(s), W_{2}^{s, V}-\frac{1}{k+4}\log(s)\right) \ .
	\end{cases}
\]
Recalling that $(u_{1}, u_{2})$ is the solution to the planar Hitchin equations with polynomial quartic differential $\zeta^{k_{i}}d\zeta^{4}$, the above system is equivalent to
\[
	\begin{cases}
	e^{u_{1}-u_{2}}(e^{A}-1)-e^{-2u_{1}}|\zeta^{k}|^{2}\left(\frac{e^{-4A}}{\hat{\sigma}_{s}}-1\right)+\frac{3}{4}\kappa(\sigma_{s})\geq 0 \\
	e^{2u_{2}}(e^{2A}-1)+e^{u_{1}-u_{2}}(1-e^{A})+\frac{1}{4}\kappa(\sigma_{s})\geq 0
	\end{cases} \ .
\]
In the first inequality we notice that, since $\kappa(\sigma_{s})$ tends to $0$, for every $A$ large enough we can make the sum of the first and last term positive. Note that $u_2$ is bounded away from $-\infty$ by Lemma \ref{lm:subsolution}, so we can take the sum of the first and third terms to grow like $e^{2A}$ in $A$. Moreover, the second term can be made non-negative for $A$ large enough and $s$ large enough because $\hat{\sigma}_{s}$ tends to $1$. 
Dividing by $e^{2u_{2}}$, the second inequality is equivalent to 
\[
	(e^{2A}-1)+e^{u_{1}-3u_{2}}(1-e^{A})+\frac{1}{4}\kappa(\sigma_{s})e^{-2u_{2}}\geq 0 \ .
\]
From (\ref{eq:bound_u1}) and (\ref{eq:bound_u2}), we know that $u_{1}(\zeta)-3u_{2}(\zeta)=o(1)$ as $|\zeta| \to +\infty$. Therefore the condition holds for $A$ large enough, because the coefficient  $e^{2A}-1$ is dominant. 
\end{proof}

\begin{proof}[Proof of Theorem \ref{thm:rays}] In the local chart $(V,\zeta_{s})$ around the zero $p$ of order $k$ of the quartic differential $q_{0}$, the sub- and super-solutions found in Lemma \ref{lm:subsolution_rays} and Lemma \ref{lm:supersolution_rays} imply that the densities $(\psi_{1}^{s}, \psi_{2}^{s})$ satisfy
\begin{align}\begin{split}\label{eq:comparison}
    w_{1}^{s} \leq \psi_{1}^{s} \leq W_{1}^{s} \ \ \ \text{and} \ \ \
    w_{2}^{s} \leq \psi_{2}^{s} \leq W_{2}^{s} \ .
\end{split}
    \end{align}
We already remarked that the sequence of coordinate charts $(V, \zeta_{s})$ converges geometrically to $(\C, \zeta)$. Consider a point $q \in \C$ that is the limit of the sequence $\zeta_{s}=s^{\frac{1}{k+4}}z_{s}$, with $z_{s}=s^{-\frac{1}{k+4}}\zeta(q) \in U$. Evaluating Equation \ref{eq:comparison} at $\zeta_{s}$, by definition of the functions $w_{j}^{s}$, $W_{j}^{s}$ and $\psi_{j}$, we obtain
\begin{align}\begin{split}
    u_{1}^{s}(\zeta_{s})-B \leq &v_{1}^{s}(z_{s}) \leq u_{1}^{s}(\zeta_{s})+2A \\
    u_{2}^{s}(\zeta_{s})-B \leq &v_{2}^{s}(z_{s}) \leq u_{2}^{s}(\zeta_{s})+A \ .
\end{split}\end{align}
Since $\zeta_{s} \to \zeta(q)$ as $s \to +\infty$, the above inequalities give uniform bound (independent of $s$) on every compact set in $V$ for the functions $v_{1}^{s}$ and $v_{2}^{s}$ and their Laplacian. Therefore, they converge $C^{1, \alpha}$ on compact sets to functions $v_{1}^{\infty}$ and $v_{2}^{\infty}$ defined on the limiting plane $(\C, \zeta)$. Since $(v_{1}^{s}, v_{2}^{s})$ is a sequence of solutions of Equation \ref{eq:PDE3}, the limit $(v_{1}^{\infty},v_{2}^{\infty})$ is a weak solution of the system of PDEs obtained by taking the limit as $s \to +\infty$ of the coefficients. As observed before, this is the planar Hitchin's equation with polynomial quartic differential $q_{\infty}=\zeta^{k}d\zeta^{4}$. Hence, by the injectivity portion of Theorem~\ref{thmB} (see especially Proposition~\ref{prop:injectivity}), the functions $v_{i}^{\infty}$ are the solutions found in Theorem \ref{thm:existence}. In particular they are smooth and the convergence of $v_{i}^{s}$ to $v_{i}^{\infty}$ is smooth as well. This shows that
for the sequence of radii $r_{s}=s^{-\frac{1}{k+4}}$, the harmonic metric $g_{s}$ on the ball centered at $p$ and radius $r_{s}$ satisfies
\[
  g_{1,s}^{-1}=e^{\psi_{1}^{s}}\sigma^{\frac{3}{2}}=e^{v_{1}^{s}}s^{\frac{3}{k+4}}\sigma^{\frac{3}{2}}=O(s^{\frac{3}{k+4}}\sigma^{\frac{3}{2}})
\]
and similarly
\[
 g_{2,s}^{-1}=e^{\psi_{2}^{s}}\sigma^{\frac{1}{2}}=e^{v_{2}^{s}}s^{\frac{1}{k+4}}\sigma^{\frac{1}{2}}=O(s^{\frac{1}{k+4}}\sigma^{\frac{1}{2}})
\]
as $s \to +\infty$.
\end{proof}

As a consequence of this proof -- especially the use of the uniqueness of the solutions $v_{i}^{\infty}$ in the last paragraph -- we can describe the pointed Gromov-Hausdorff limit of the family of maximal surfaces in $\h^{2,2}$ associated to the Higgs bundles $(\sE, \varphi_{s})$, as defined in \cite{CTT}.

\begin{cor} \label{cor: localization of maximal surfaces along rays} Let $p \in X$ be a zero of order $k$ of the quartic differential $q_{0}$. Let $\Sigma_{s}$ denote the maximal surfaces in $\h^{2,2}$ associated to the family of Higgs bundles $(\sE, \varphi_{s})$ and denote by $I_{s}$ its induced metric. Let $(\Sigma_{\infty}, I_{\infty})$ be the conformally planar maximal surface with polynomial quartic differential $\zeta^{k}d\zeta^{4}$ endowed with its induced metric. Then $(\Sigma_{s}, I_{s},p)$ converges, up to subsequences and composition by global isometries, to $(\Sigma_{\infty}, I_{\infty},p)$ in the pointed Gromov-Hausdorff topology.
\end{cor}
\begin{proof} Let $(V,z)$ be a natural coordinate chart for $q_{0}$ so that $q_{0}(z)=z^{k}dz^{4}$. We already showed that the new coordinate charts $(V, \zeta_{s})$ defined by $\zeta_{s}=s^{\frac{1}{k+4}}z$ converge geometrically to the complex plane $(\C, \zeta)$ and in these coordinates $q_{s}=\zeta_{s}^{k}d\zeta^{4}$ tends to $q_{\infty}=\zeta^{k}d\zeta^{4}$ uniformly on compact sets. It is thus sufficient to prove that the induced metrics $I_{s}$ restricted to $V$ converge to $I_{\infty}$ smoothly on compact sets, up to subsequences. This follows immediately from the proof of Theorem \ref{thm:rays}. In fact, up to subsequences,
\[
    {I_{s}}{_{|_V}}=4e^{\psi_{1}^{s}-\psi_{2}^{s}}\sigma(z)|dz|^{2}=4e^{v_{1}^{s}-v_{2}^{s}}\hat{\sigma}_{s}(\zeta_{s})|d\zeta_{s}|^{2} \to 4e^{v_{1}^{\infty}-v_{2}^{\infty}}|d\zeta|^{2}=I_{\infty} \ .
\]
\end{proof}

\bibliographystyle{alpha}
\bibliography{bs-bibliography}

\begin{thebibliography}{HTTW95}

\bibitem[Ama76]{Amann}
Herbert Amann.
\newblock Fixed point equations and nonlinear eigenvalue problems in ordered
  {B}anach spaces.
\newblock {\em SIAM Rev.}, 18(4):620--709, 1976.

\bibitem[Bar10]{Baraglia}
David Baraglia.
\newblock G2 geometry and integrable system.
\newblock {\em arXiv:1002.1767}, 2010.

\bibitem[BB04]{BB_wild}
Olivier Biquard and Philip Boalch.
\newblock Wild non-abelian {H}odge theory on curves.
\newblock {\em Compos. Math.}, 140(1):179--204, 2004.

\bibitem[Boa14]{Boalch_Stokes}
P.~P. Boalch.
\newblock Geometry and braiding of {S}tokes data; fission and wild character
  varieties.
\newblock {\em Ann. of Math. (2)}, 179(1):301--365, 2014.

\bibitem[CL17]{Collier-Li}
Brian Collier and Qiongling Li.
\newblock Asymptotics of {H}iggs bundles in the {H}itchin component.
\newblock {\em Adv. Math.}, 307:488--558, 2017.

\bibitem[Cor88]{Corlette}
Kevin Corlette.
\newblock Flat {$G$}-bundles with canonical metrics.
\newblock {\em J. Differential Geom.}, 28(3):361--382, 1988.

\bibitem[CTT19]{CTT}
Brian Collier, Nicolas Tholozan, and J\'er\'emy Toulisse.
\newblock The geometry of maximal representations of surface groups into
  {S}{O}(2,n).
\newblock {\em To appear in Duke Math. J.}, 2019.

\bibitem[Don87]{Donaldson_selfduality}
S.~K. Donaldson.
\newblock Twisted harmonic maps and the self-duality equations.
\newblock {\em Proc. London Math. Soc. (3)}, 55(1):127--131, 1987.

\bibitem[DW15]{DWpolygons}
David Dumas and Michael Wolf.
\newblock Polynomial cubic differentials and convex polygons in the projective
  plane.
\newblock {\em Geom. Funct. Anal.}, 25(6):1734--1798, 2015.

\bibitem[DW20]{DW_rays}
David Dumas and Michael Wolf.
\newblock Rays of convex real projective structures.
\newblock {\em In preparation}, 2020.

\bibitem[FN17]{FN}
Laura Fredrickson and Andrew Neitzke.
\newblock From ${S}^{1}$-fixed points to {W}-algebra representations.
\newblock {\em arXiv:1709.06142}, 2017.

\bibitem[Gup17]{Gupta_wild}
Subhojoy Gupta.
\newblock Harmonic maps and wild teichm\"uller spaces.
\newblock {\em To appear in Journal of Topol-ogy and Analysis,
  arXiv:1708.04780}, 2017.

\bibitem[GW08]{GW_PSL4}
Olivier Guichard and Anna Wienhard.
\newblock Convex foliated projective structures and the {H}itchin component for
  {${\rm PSL}_4({\bf R})$}.
\newblock {\em Duke Math. J.}, 144(3):381--445, 2008.

\bibitem[Hit87]{Hitchin_selfduality}
N.~J. Hitchin.
\newblock The self-duality equations on a {R}iemann surface.
\newblock {\em Proc. London Math. Soc. (3)}, 55(1):59--126, 1987.

\bibitem[Hit92]{Hitchin_component}
N.~J. Hitchin.
\newblock Lie groups and {T}eichm\"{u}ller space.
\newblock {\em Topology}, 31(3):449--473, 1992.

\bibitem[HTTW95]{HTTW}
Zheng-Chao Han, Luen-Fai Tam, Andrejs Treibergs, and Tom Wan.
\newblock Harmonic maps from the complex plane into surfaces with nonpositive
  curvature.
\newblock {\em Comm. Anal. Geom.}, 3(1-2):85--114, 1995.

\bibitem[Ish88]{Ishihara}
Toru Ishihara.
\newblock Maximal spacelike submanifolds of a pseudo-{R}iemannian space of
  constant curvature.
\newblock {\em Michigan Math. J.}, 35(3):345--352, 1988.

\bibitem[Lab07]{Labourie_cubic}
Fran\c{c}ois Labourie.
\newblock Flat projective structures on surfaces and cubic holomorphic
  differentials.
\newblock {\em Pure Appl. Math. Q.}, 3(4, Special Issue: In honor of Grigory
  Margulis. Part 1):1057--1099, 2007.

\bibitem[Lab17]{Lab:cyclicsurf}
Fran\c{c}ois Labourie.
\newblock Cyclic surfaces and {H}itchin components in rank 2.
\newblock {\em Ann. of Math. (2)}, 185(1):1--58, 2017.

\bibitem[Li19]{QL_introduction}
Qiongling Li.
\newblock An introduction to {H}iggs bundles via harmonic maps.
\newblock {\em SIGMA Symmetry Integrability Geom. Methods Appl.}, 15:Paper No.
  035, 30, 2019.

\bibitem[LM19]{Li_Mochizuki_cyclic_noncompact}
Qiongling Li and Takuro Mochizuki.
\newblock Cyclic higgs bundles over noncompact surfaces.
\newblock {\em In preparation}, 2019+.

\bibitem[Lof01]{Loftin_RP2}
John~C. Loftin.
\newblock Affine spheres and convex $\mathbb{RP}^n$-manifolds.
\newblock {\em Amer. J. Math.}, 123(2):255--274, 2001.

\bibitem[LTW20]{LabTW20}
Fran\c{c}ois Labourie, J\'er\'emy Toulisse, and Michael Wolf.
\newblock Maximal surfaces in pseudo-hyperbolic spaces.
\newblock 2020.
\newblock in preparation.

\bibitem[Mes07]{Mess}
Geoffrey Mess.
\newblock Lorentz spacetimes of constant curvature.
\newblock {\em Geom. Dedicata}, 126:3--45, 2007.

\bibitem[Moc14]{Mochizuki_harmonicbundles}
Takuro Mochizuki.
\newblock Harmonic bundles and {T}oda lattices with opposite sign {II}.
\newblock {\em Comm. Math. Phys.}, 328(3):1159--1198, 2014.

\bibitem[Mor58]{Morrey_regularity}
Charles~B. Morrey, Jr.
\newblock On the analyticity of the solutions of analytic non-linear elliptic
  systems of partial differential equations. {I}. {A}nalyticity in the
  interior.
\newblock {\em Amer. J. Math.}, 80:198--218, 1958.

\bibitem[Omo67]{Omori}
Hideki Omori.
\newblock Isometric immersions of {R}iemannian manifolds.
\newblock {\em J. Math. Soc. Japan}, 19:205--214, 1967.

\bibitem[Sim88]{Simpson_Hodge}
Carlos~T. Simpson.
\newblock Constructing variations of {H}odge structure using {Y}ang-{M}ills
  theory and applications to uniformization.
\newblock {\em J. Amer. Math. Soc.}, 1(4):867--918, 1988.

\bibitem[Sim90]{Simpson_harmonicbundles}
Carlos~T. Simpson.
\newblock Harmonic bundles on noncompact curves.
\newblock {\em J. Amer. Math. Soc.}, 3(3):713--770, 1990.

\bibitem[Sim09]{Simpson_Katz}
Carlos Simpson.
\newblock Katz's middle convolution algorithm.
\newblock {\em Pure Appl. Math. Q.}, 5(2, Special Issue: In honor of Friedrich
  Hirzebruch. Part 1):781--852, 2009.

\bibitem[Tam19]{Tambu_poly}
Andrea Tamburelli.
\newblock Polynomial quadratic differentials on the complex plane and
  light-like polygons in the einstein universe.
\newblock {\em Advances in Mathematics}, 352C:483--515, 2019.

\bibitem[Wie18]{Wienhard_ICM}
Anna Wienhard.
\newblock An invitation to higher teichm\"uller theory.
\newblock {\em Proc. Int. Cong. of Math. Rio de Janeiro}, 1:1007--1034, 2018.

\bibitem[Wol89]{Wolf_thesis}
Michael Wolf.
\newblock The {T}eichm\"{u}ller theory of harmonic maps.
\newblock {\em J. Differential Geom.}, 29(2):449--479, 1989.

\bibitem[Wol91]{WolfRays}
Michael Wolf.
\newblock High energy degeneration of harmonic maps between surfaces and rays
  in {T}eichm\"{u}ller space.
\newblock {\em Topology}, 30(4):517--540, 1991.

\end{thebibliography}

\end{document}